\documentclass[11pt, reqno]{article}

\usepackage{amsmath,amssymb,amsfonts,amsthm}
\usepackage{color}

\usepackage[colorlinks=true,linkcolor=blue,citecolor=blue,urlcolor=blue,pdfborder={0 0 0}]{hyperref}
\usepackage{cleveref}

\usepackage{caption}
\usepackage{thmtools}

\usepackage{mathrsfs}

\crefname{theorem}{Theorem}{Theorems}
\crefname{thm}{Theorem}{Theorems}
\crefname{lemma}{Lemma}{Lemmas}
\crefname{claim}{Claim}{Claims}
\crefname{lem}{Lemma}{Lemmas}
\crefname{remark}{Remark}{Remarks}
\crefname{prop}{Proposition}{Propositions}
\crefname{defn}{Definition}{Definitions}
\crefname{corollary}{Corollary}{Corollaries}
\crefname{conjecture}{Conjecture}{Conjectures}
\crefname{question}{Question}{Questions}
\crefname{chapter}{Chapter}{Chapters}
\crefname{section}{Section}{Sections}
\crefname{part}{Part}{Parts}
\crefname{figure}{Figure}{Figures}

\theoremstyle{plain}
\newtheorem{thm}{Theorem}[section]

\newtheorem{lemma}[thm]{Lemma}

\newtheorem{theorem}[thm]{Theorem}

\newtheorem{corollary}[thm]{Corollary}

\newtheorem{prop}[thm]{Proposition}

\theoremstyle{definition}

\theoremstyle{remark}
\newtheorem{remark}[thm]{Remark}

\numberwithin{equation}{section}

\renewcommand{\P}{\mathbb P}
\newcommand{\E}{\mathbb E}

\newcommand{\R}{\mathbb R}
\newcommand{\Z}{\mathbb Z}
\newcommand{\N}{\mathbb N}

\newcommand{\F}{\mathfrak F}

\newcommand{\cE}{\mathcal E}

\newcommand{\cI}{\mathcal I}

\newcommand{\cW}{\mathcal W}

\newcommand{\sA}{\mathscr A}
\newcommand{\sB}{\mathscr B}

\newcommand{\sD}{\mathscr D}
\newcommand{\sE}{\mathscr E}

\newcommand{\sI}{\mathscr I}

\newcommand{\sT}{\mathscr T}

\newcommand{\fF}{\mathfrak F}

\newcommand{\fP}{\mathfrak P}

\newcommand{\fT}{\mathfrak T}

\newcommand{\LE}{\mathsf{LE}}

\usepackage[margin=1in]{geometry}
\usepackage{tikz}
\usepackage{pgfplots}
\usepackage{extarrows}
\usepackage{titling}
\usepackage{commath}
\usepackage{wasysym}
\usepackage{mathtools}
\usepackage{titlesec}
\newcommand{\eps}{\varepsilon}
\usepackage{bbm}
\usepackage{setspace}
\usepackage{enumitem}
\usepackage{booktabs}

\newcommand{\bP}{\mathbf P}

\newcommand{\Eta}{\mathrm{H}}
\newcommand{\Av}{\operatorname{Av}}

\def\B{\mathcal{B}}
\renewcommand{\epsilon}{\varepsilon}

\newcommand{\1}{{\text{\Large $\mathfrak 1$}}}

\renewcommand{\emptyset}{\varnothing}

\newcommand{\til}{\widetilde}

\newcommand{\pr}[1]{\mathbb{P}\!\left(#1\right)}
\renewcommand{\E}[1]{\mathbb{E}\!\left[#1\right]}
\newcommand{\estart}[2]{\mathbb{E}_{#2}\!\left[#1\right]}
\newcommand{\prstart}[2]{\mathbb{P}_{#2}\!\left(#1\right)}
\newcommand{\prcond}[3]{\mathbb{P}_{#3}\!\left(#1\;\middle\vert\;#2\right)}
\newcommand{\econd}[2]{\mathbb{E}\!\left[#1\;\middle\vert\;#2\right]}

\newcommand{\tn}{|\kern-.1em|\kern-0.1em|}

\newcommand{\cpc}[2]{\mathrm{Cap}_{#1}(#2)}

\newcommand{\vr}[1]{\mathrm{Var}\left(#1\right)}

\newcommand{\cc}[1]{\mathrm{Cap}\left(#1\right)}

\newcommand\be{\begin{equation}}
\newcommand\ee{\end{equation}}

\def\bP{\mathbb{P}}

\def\eps{\varepsilon}

\newcommand{\lr}[1]{{\rm{LE}}(#1)}
\newcommand{\cpp}[1]{\mathrm{Cap}\left(#1\right)}

\newcommand{\ler}[2]{{\mathsf{LE}}_{#1}\!\left(#2 \right)}

\newcommand{\Ceff}{\mathscr{C}_\mathrm{eff}}
\newcommand{\bG}{\mathbf{G}}

\newcommand{\AB}{\mathsf{AB}}

\newcommand{\radint}{\operatorname{rad}_\mathrm{int}}
\newcommand{\radext}{\operatorname{rad}_\mathrm{ext}}

\newcommand{\AvC}{\operatorname{AvC}}

\pretitle{\begin{flushleft}\Large}
\posttitle{\par\end{flushleft}\vskip 0.5em}
\preauthor{\begin{flushleft}}
\postauthor{
\\
\vspace{0.5em}
\footnotesize{$*$ The Division of Physics, Mathematics and Astronomy, California Institute of Technology}
\\
Email:
\href{mailto:t.hutchcroft@caltech.edu}{t.hutchcroft@caltech.edu}\\
$\dagger$ Statslab, DPMMS, University of Cambridge\\
Email: \href{mailto:p.sousi@statslab.cam.ac.uk}{p.sousi@statslab.cam.ac.uk}

\end{flushleft}}

\predate{\begin{flushleft}}
\postdate{\par\end{flushleft}}
\title{\bf Logarithmic corrections to scaling in the four-dimensional uniform spanning tree}

\renewenvironment{abstract}
 {\par\noindent\textbf{\abstractname.}\ \ignorespaces}
 {\par\medskip}

\titleformat{\section}
  {\normalfont\large \bf}{\thesection}{0.4em}{}

  \titleformat{\part}
  {\Large \sc \bf}{\partname~\thepart:}{0.4em}{}

\author{{\bf Tom Hutchcroft$^*$ and Perla Sousi$^\dagger$}}
\begin{document}

\date{\small{\today}}

\maketitle

\begin{abstract}
We compute  the precise logarithmic corrections to mean-field scaling for various  quantities describing the uniform spanning tree of the four-dimensional hypercubic lattice $\Z^4$. We are particularly interested in the distribution of the \emph{past} of the origin, that is, the finite piece of the tree that is separated from infinity by the origin. We prove that the probability that the past contains a path of length $n$ is of order $(\log n)^{1/3}n^{-1}$, that the probability that the past contains at least $n$ vertices is of order $(\log n)^{1/6} n^{-1/2}$, and that the probability that the past reaches the boundary of the box $[-n,n]^4$ is of order $(\log n)^{2/3+o(1)}n^{-2}$. 
 An important part of our proof is to prove concentration estimates for the capacity of the four-dimensional loop-erased random walk which may be of independent interest.

 Our results imply that the Abelian sandpile model also  exhibits non-trivial polylogarithmic corrections to mean-field scaling in four dimensions, although it remains open to compute the precise order of these corrections.
\end{abstract}

\newpage

\tableofcontents

\newpage

\setstretch{1.1}

\section{Introduction}\label{sec:intro}

Many models in probability and statistical mechanics are believed to have an \emph{upper-critical dimension} $d_c$ above which they exhibit \emph{mean-field} critical behaviour. This means that when $d>d_c$ the model behaves at criticality in roughly the same way on a $d$-dimensional lattice as it does in ``geometrically trivial" settings such as the complete graph or the $3$-regular tree. In low dimensions $d < d_c$ the geometry of the lattice affects the model in a non-negligible way so that its behaviour is substantially different to the high-dimensional case. At the upper-critical dimension $d=d_c$ itself the model's behaviour is expected to be \emph{almost} mean-field: in particular, several quantities of interest are expected to differ from their mean-field values by a polylogarithmic factor when $d=d_c$ and a polynomial factor when $d<d_c$. For many natural models the upper-critical dimension is equal to $4$, and understanding these models at the upper-critical dimension is closely related to important problems in constructive quantum field theory in $3+1$ space-time dimensions~\cite{aizenman2019marginal,fernandez2013random}. While important progress has been made on various models including the Ising model~\cite{aizenman2019marginal}, weakly self-avoiding walk \cite{MR3339164,MR3345374}, loop-erased random walk \cite{Lawlerlog,Loop}, and the $\varphi^4$ model \cite{MR3269689}, the class of models that can be understood rigorously at the upper-critical dimension remains very limited.

In this paper we analyze the logarithmic corrections to mean-field scaling in the four dimensional \emph{uniform spanning tree}, particularly with regard to the distribution of the \emph{past} of the origin. 
Our results complement those of Lawler \cite{Lawlerlog} and Schweinsberg \cite{schweinsberg2009loop}, who computed the logarithmic corrections to scaling for some other features of the model.
 Before stating our results, let us first recall the definition of the model, referring the reader to \cite{LP:book,barlow2014loop,1804.04120} for further background\footnote{
A further introduction to the subject along with an informal overview of the arguments of the present paper and of~\cite{1804.04120,hutchcroft2015interlacements} can be found in the first author's lectures at the 2020 \href{https://www.math.ubc.ca/Links/OOPS/}{Online Open Probability Summer School}, available at \url{https://www.youtube.com/playlist?list=PLpgoGs2cboIXEVeL0OoKRP6rKpAdW5Kiz}.
 }. 
 A uniform spanning tree of a finite connected graph is simply a spanning tree of the graph chosen uniformly at random; the \textbf{uniform spanning forest} of the hypercubic lattice $\Z^d$ is defined to be the weak limit of the uniform spanning trees of the boxes $\Lambda_r=[-r,r]^d \cap \Z^d$, or equivalently of any other exhaustion of $\Z^d$ by finite connected subgraphs. This limit was proven to exist independently of the choice of exhaustion by Pemantle \cite{Pem91}, who also proved that the uniform spanning forest of $\Z^d$ is almost surely connected, i.e., a single tree, if and only if $d \leq 4$: this is a consequence of the fact that two independent walks on $\Z^d$ intersect infinitely often a.s.\ if and only if $d\leq 4$ \cite{ErdosTaylor}, and is closely related to the fact that the upper-critical dimension of the uniform spanning tree is $4$. 
 In light of these results, we refer to the uniform spanning forest of $\Z^d$ as the uniform spanning \emph{tree} when $d\leq 4$.
There are various interesting senses in which the four-dimensional uniform spanning tree only just manages to be connected: For example, it can be shown that the length of the path connecting two neighbouring vertices has an extremely heavy $(\log n)^{-1/3}$ tail \cite{Lawlerlog}.

Besides connectivity, the other basic topological features of the uniform spanning forest are also now understood in every dimension. Indeed, following partial results of Pemantle \cite{Pem91}, it was proven by Benjamini, Lyons, Peres, and Schramm \cite{USFBenLyPeSc} that every tree in the uniform spanning forest of $\Z^d$ is \emph{one-ended} almost surely when $d\geq 2$. This means that for every vertex $x\in \Z^d$ there is exactly one simple path to infinity emanating from $x$ which, by Wilson's algorithm \cite{Wilson96,USFBenLyPeSc}, is distributed as an infinite loop-erased random walk.
 See also \cite{LMS08,hutchcroft2015interlacements,H15} for further related results.
  In order to quantify this one-endedness and better understand the geometry of the trees, we seek to analyze
the distribution of the \emph{finite} pieces of the tree that hang off this infinite spine. 

Let us now introduce some relevant notation. Let $\fT$ be the uniform spanning tree of $\Z^4$.
For each $x\in \Z^4$, the \textbf{past}\footnote{The character $\fP$ we use to denote the past is $\backslash$mathfrak$\{$P$\}$.} $\fP(x)$ of $x$ in $\fT$ is defined to be the union of $x$ with the finite connected components of $\fT \setminus \{x\}$.  We refer to the graph distance on $\fT$ as the \textbf{intrinsic}  distance (a.k.a.\ chemical distance) and the graph distance on $\Z^4$ as the \textbf{extrinsic} distance.
We write $\radint(\fP(0))$ and $\radext(\fP(0))$ for the \textbf{intrinsic} and \textbf{extrinsic radii} of $\fP(0)$, that is, the maximum intrinsic or extrinsic distance between $0$ and another point in $\fP(0)$ as appropriate.
In high dimensions, it is proven in \cite[Theorem 1.1]{1804.04120} that the past has intrinsic diameter at least $n$ with probability of order $n^{-1}$ and volume at least $n$ with probability of order $n^{-1/2}$;
 the same as the probabilities that the survival time and total progeny of a critical, finite-variance branching process are at least $n$ respectively \cite{LPP95}.
Our first main theorem computes the logarithmic corrections to this behaviour in four dimensions, giving up-to-constants estimates on the probability that the past has large intrinsic radius or volume.

\begin{theorem}[Volume and the intrinsic one-arm]
\label{thm:wusf}
Let $\fT$ be the uniform spanning tree of $\Z^4$ and let $\fP=\fP(0)$ be the past of the origin. Then
\begingroup
\addtolength{\jot}{0.5em}
\begin{align}
 \pr{
 \radint(\fP) \geq n 
 } &\asymp \frac{(\log n)^{1/3}}{n} 
 \hspace{3cm} \text{ and}
\label{exponent:intrad}
\\
 \pr{|\fP| \geq n } &\asymp \frac{(\log n)^{1/6}}{n^{1/2}} 
\label{exponent:vol}
\end{align}
\endgroup
for every $n\geq 2$.
\end{theorem}

Our proof builds upon both the ideas developed to analyze the high-dimensional uniform spanning forest in \cite{1804.04120}
 and on Lawler's results on the logarithmic corrections for \emph{loop-erased random walk} in four dimensions \cite{Lawlerlog,Law91}. We discuss Lawler's results in detail in \cref{subsec:LERWbackground}.
An outline of the proof, including a heuristic derivation of the exponents appearing here from Lawler's results, is given in \cref{subsec:about_the_proof}. As in \cite{1804.04120}, our proof relies heavily on the analysis of the \emph{interlacement Aldous--Broder algorithm} \cite{hutchcroft2015interlacements}. In order to perform this analysis, it is important to establish concentration estimates for the \emph{capacity} of 4d loop-erased random walk, which are stated in \cref{sec:capacity,subsec:polygamous_deviants} and may be of independent interest. The proof of \cref{thm:wusf,thm:extrinsic} also yields related estimates for the Abelian sandpile model, which are stated in \cref{subsec:sandpile}.


Here and elsewhere, we write $\asymp$, $\gtrsim$, and $\lesssim$ for equalities and inequalities that hold to within multiplication by a positive constant. 
We also make use of standard Landau asymptotic big-$O$ and little-$o$ notation.
In particular, if $f,g:\N \to [0,\infty)$ then $f(n) \lesssim g(n)$ and $f(n)=O(g(n))$ both mean that there exists a positive constant $C$ such that $f(n) \leq C g(n)$ for every $n\geq 1$, while $f(n)=o(g(n))$ means that $f(n)/g(n) \to 0$ as $n\to\infty$. When the implicit constants depend on an additional parameter we denote this using subscripts so that, for example, ``$f_\lambda(n) \asymp_\lambda g_\lambda(n)$ for every $n\geq 1$ and $\lambda \geq 1$'' means that for every $\lambda \geq 1$ there exist positive constants $c_\lambda$ and $C_\lambda$ such that $c_\lambda g_\lambda(n) \leq f_\lambda(n) \leq C_\lambda g_\lambda(n)$ for every $n \geq 1$. In particular, if $g(n)$ is positive and bounded away from zero then the statements ``$f(n) \lesssim g(n)^{1+o(1)}$ for every $n\geq 1$'' and ``$f(n) \lesssim_\eps g(n)^{1+\eps}$ for every $\eps>0$ and $n\geq 1$'' are equivalent. 






We now consider the \emph{extrinsic} radius. 
 When $d \geq 5$, it is proven in \cite[Theorem 1.4]{1804.04120} that the past of the origin has extrinsic radius at least $r$ with probability of order $r^2$; the same as the probability that a critical, finite-variance branching random walk reaches distance at least $r$. Our next main theorem establishes the order of the logarithmic correction to the same probability in four dimensions.

\begin{theorem}[The extrinsic one-arm]
\label{thm:extrinsic}
Let $\fT$ be the uniform spanning tree of $\Z^4$ and let $\fP=\fP(0)$ be the past of the origin. Then
\begin{equation}
 \frac{(\log n)^{2/3}}{n^2}  \lesssim \pr{ \radext(\fP)\geq n } \lesssim \frac{(\log n)^{2/3+o(1)}}{n^2} 
\label{exponent:extrad}
\end{equation}
for every $n\geq 2$.
\end{theorem}

\begin{remark}
 We believe that it should be possible to remove the $(\log n)^{o(1)}$ error from the upper bound in this theorem using roughly the same methods we develop here but that this may be a highly technical matter. We have chosen not to pursue this further in light of the paper's already formidable length.
\end{remark}

\begin{remark}
It is a fairly straightforward matter to adapt the high-dimensional methods of \cite{1804.04120} to prove that e.g.\ $\P(\radint(\fP) \geq n) \asymp n^{-1} (\log n)^{\pm O(1)}$ in four dimensions, along with similar bounds on the tail of the volume and extrinsic radius. See also \cite{MR3718713} for earlier results of a similar nature. The main contribution of the present paper is to identify the \emph{precise order} of the logarithmic corrections, which is much more challenging technically.
\end{remark}

\begin{remark}
Our methods are not particularly specific to the hypercubic lattice, and we expect that all our results should extend to arbitrary transitive graphs of four-dimensional volume growth. In particular, we believe that the same logarithmic corrections to mean-field scaling should hold universally even in non-Euclidean examples rough-isometric to the Heisenberg group. To establish such a generalization, one would first need to generalize Lawler's work on loop-erased random walk \cite{Law91,Lawlerlog} to this setting. We do not pursue this here.
\end{remark}

\begin{remark}
The results of the present paper play a central role in the computation of the logarithmic corrections to mean-field scaling for the \emph{random walk} on the four-dimensional uniform spanning tree as carried out in subsequent work of Halberstam and the first author \cite{halberstam2022logarithmic}.
\end{remark}

\textbf{Further discussion and relation to previous work.} 
The uniform spanning tree is closely connected to many other interesting models in probability and statistical mechanics, most notably to random walk and loop-erased random walk via the Aldous--Broder algorithm \cite{Aldous90,broder1989generating,hutchcroft2015interlacements} and Wilson's algorithm \cite{Wilson96,USFBenLyPeSc}.
These connections 
  have made it much more amenable to rigorous analysis than essentially any other non-trivial statistical mechanics-type model. 
 This tractability has led the uniform spanning tree to be at the forefront of developments in probability theory over the last thirty years.
Indeed, the study of the uniform spanning tree and loop-erased random walk in two dimensions was instrumental in the development of the theory of Schramm-Loewner evolutions and conformally invariant scaling limits \cite{LaSchWe04,Ken00,SchrammSLE}. More recently, Kozma~\cite{MR2350070} and Angel, Croydon, Hernandez-Torres, and Shiraishi \cite{angel2020scaling} have proven that the scaling limit of the three-dimensional loop-erased random walk and uniform spanning tree are well-defined, a fact that is not known for essentially any other non-trivial three-dimensional model.

Closer to our setting, a very detailed understanding of loop-erased random walk on $\Z^d$ with $d\geq 4$ has been established in the work of Lawler \cite{Law91,Law85,LawlerSAW,Lawlerlog}, who proved in particular that loop-erased random walk on $\Z^d$ has Brownian motion as its scaling limit when $d\geq 4$ \cite{Loop}. Further strong results on 4d loop-erased random walk have recently been obtained in \cite{LawlerSunWu}. Lawler's results for four-dimensional loop-erased random walk play an important role in this paper, and are discussed in detail in \cref{subsec:LERWbackground}. In addition to the work on high-dimensional spanning forests \cite{1804.04120,MR3718713} mentioned above, several related works have also studied the uniform spanning trees of high-dimensional \emph{tori}. Peres and Revelle \cite{peres2004scaling} proved that the uniform spanning tree of the torus $\mathbb{T}^d_n:=(\Z/n\Z)^d$ converges to Aldous's continuum random tree when $d\geq 5$ and $n\to\infty$ after rescaling by a factor of order $n^{d/2}\asymp |\mathbb{T}^d_n|^{1/2}$. Building on this work, Schweinsberg~\cite{schweinsberg2009loop} established a similar scaling limit theorem for four-dimensional tori but where the relevant scaling factor is of order $n^2 (\log n)^{1/6} \asymp |\mathbb{T}^4_n|^{1/2} (\log |\mathbb{T}^4_n|)^{1/6}$. See also \cite{MR2172682} for further related results. 

While the results of Schweinsberg are closely related in spirit to those that we prove here, our results operate at a different scale to his and it does not seem that either set of results can be used to deduce the other. 
Indeed, both aforementioned convergence theorems are stated in terms of the \emph{Gromov-weak} topology, which means that the matrix of distances between $k$ uniform random points converges in distribution to the corresponding distribution in the continuum random tree for each fixed $k\geq 2$. In particular, Schweinsberg's result implies that the intrinsic distance between two \emph{typical} points of $\mathbb{T}^4_n$ is of order $n^2 (\log n)^{1/6}$ but does not establish a similar estimate for the \emph{diameter} of the spanning tree. Recently, Michaeli, Nachmias, and Shalev \cite{michaeli2020diameter} developed a finite-volume version of the methods of \cite{1804.04120} that allowed them to prove that the diameter of the uniform spanning tree of $\mathbb{T}^d_n$ is of order $|\mathbb{T}^d_n|^{1/2}$ with high probability when $d\geq 5$ and $n$ is large, and Archer, Nachmias, and Shalev \cite{archer2021ghp} subsequently sharpened this result further to show Gromov-Hausdorff-Prokhorov convergence of USTs of high-dimensional tori to the continuum random tree. See also \cite{alon2020diameter} for further related results. We are optimistic that a synthesis of the methods of the present paper with those of \cite{michaeli2020diameter,archer2021ghp} may allow one to prove that the diameter of the uniform spanning tree of $\mathbb{T}^4_n$ is of order $n^2 (\log n)^{1/6}$ with high probability when $n$ is large, and hopefully also to strengthen Schweinsberg's results from Gromov-weak to Gromov-Hausdorff or Gromov-Hausdorff-Prokhorov convergence.



\subsection{The $v$-wired uniform spanning forest}

In this section we state analogues of our main theorems for the \emph{$0$-wired uniform spanning forest} ($0$-WUSF) of $\Z^4$. This model is a variant of the uniform spanning tree in which the origin is `wired to infinity' first introduced by J\'arai and Redig \cite{JarRed08} as part of their work on the Abelian sandpile model.
Besides their intrinsic interest, our results on this model also serve an important auxiliary role in the proofs of \cref{thm:wusf,thm:extrinsic} and are used to derive upper bounds on the Abelian sandpile model as discussed in the next subsection.

We now recall the relevant definitions, taking the opportunity to recall the definition of the wired uniform spanning forest of a general graph also.
Let $G=(V,E)$ be an infinite graph, let $v$ be a distinguished vertex of $G$, and let  $( V_n )_{n\geq 1}$ be an exhaustion of $V$ by finite connected sets. For each $n\geq 1$, we define $G_n^*$ to be the graph obtained from identifying (a.k.a.\ wiring) $V\setminus V_n$ into a single point that we denote by $\partial_n$. Recall that the \textbf{wired uniform spanning forest} of $G$ is defined to be the weak limit of the uniform spanning trees of $G_n^*$. This limit was (implictly) proven to exist by Pemantle \cite{Pem91}, who also proved that the wired uniform spanning forest of $\Z^d$ coincides with the uniform spanning forest as we defined it above. Now, for each $n\geq 1$ let $G^{*v}_n$ be the graph obtained from identifying $v$ with~$\partial_n$ in the graph $G_n^*$. We define the $v$-wired uniform spanning forest on $G$ to be the weak limit of the uniform spanning trees on $G_n^{*v}$, which is well-defined and does not depend on the choice of exhaustion \cite[\S3]{LMS08}. Note that this model is \emph{not} automorphism-invariant in general, since the vertex $v$ plays a special role.

Let $G$ be an infinite transient graph, let $\fP(v)$ be the past of $v$ in the wired uniform spanning forest of $G$ and let $\fT_v$ denote the component containing $v$ in the $v$-wired uniform spanning forest of $G$. Lyons, Morris, and Schramm \cite[Proposition 3.1]{LMS08} proved that $\fT_v$ stochastically dominates $\fP(v)$ and moreover that $\fT_v$ is a.s.\ finite if and only if $\fP(v)$ is a.s.\ finite. In our primary setting of $\Z^4$, it can be deduced that the $0$-wired uniform spanning forest has exactly two components almost surely: a finite component containing $0$ and an infinite component that does not contain $0$. In \cite[Lemma 2.1]{1804.04120} (\cref{lem:domination} of this paper), a stronger version of the aforementioned stochastic domination property was derived, in which one may condition on the future of $v$ in advance. This property makes the $0$-WUSF very useful in the study of the usual WUSF, and can often be used in a similar manner to the BK inequality in the theory of Bernoulli percolation. 

As with the past of the WUSF, it is proven in \cite{1804.04120} that the component of the origin in the $0$-WUSF behaves similarly to a critical, finite-variance branching random walk in high dimensions.
Our next theorem establishes the logarithmic-corrections to this scaling in four dimensions. 

\begin{theorem}\label{thm:wusfo}
Let $\fT_0$ be the component of the origin in the $0$-WUSF of $\Z^4$. Then
\begingroup
\addtolength{\jot}{0.5em}
\begin{align}
\frac{(\log n)^{2/3}}{n} &\lesssim\hspace{1.85cm} \pr{\operatorname{rad}_\mathrm{int}(\fT_0) \geq n } &&\lesssim \frac{(\log n)^{2/3+o(1)}}{n}, 
\label{exponent:intrado}
\\
 \frac{\log n}{n^2} &\lesssim\hspace{1.85cm} \pr{\operatorname{rad}_\mathrm{ext}(\fT_0) \geq n} &&\lesssim \frac{(\log n)^{1+o(1)}}{n^2}, 
\label{exponent:extrado}
\\
 \text{and} 
  &\phantom{\lesssim}\hspace{2.075cm} \pr{|\fT_0| \geq n } &&\lesssim \frac{(\log n)^{1/2+o(1)}}{n^{1/2}} 
\label{exponent:volo}
\end{align}
\endgroup
for every $n\geq 2$.
\end{theorem}

Note that the logarithmic corrections for the two models differ by a factor of $(\log n)^{1/3}$ in each case. This is in contrast to the high-dimensional setting, where the two models have the same behaviour up to constants \cite{1804.04120}. This difference between the two models makes many of the arguments of this paper much more delicate than those of \cite{1804.04120}: bounding conditional probabilities using the stochastic domination property (\cref{lem:domination}) is typically not sharp, and indeed will usually lead to an unwanted additional polylogarithmic factor.
As such, we mostly confine the use of the stochastic domination property to show that various `bad events' have negligible probability; for these arguments to go through, it will be important for us to have concentration bounds (on e.g.\ the length and capacity of loop-erased random walks) that are strong enough to counteract the wastefulness of the stochastic domination bounds. This is in contrast to the high-dimensional case, where we could mostly get by with arbitrarily weak concentration estimates.

\begin{remark}
Again, we expect that an elaboration of our methods should be able to remove the $(\log n)^{o(1)}$ errors from the upper bounds and establish a matching lower bound for \eqref{exponent:volo}, but do not pursue this here in view of the paper's length.
\end{remark}

\subsection{Corollaries for the Abelian sandpile model}
\label{subsec:sandpile}

We now discuss applications of \cref{thm:wusf,thm:extrinsic,thm:wusfo} to the \emph{Abelian sandpile model}, a popular example of a system exhibiting \emph{self-organized criticality} that was first introduced by Bak, Tang, and Wiesenfeld \cite{MR949160} and was brought to mathematical maturity in the seminal works of Dhar \cite{Dhar90} and Majumdar and Dhar \cite{MajDhar92}. We will keep our discussion of the model brief, referring the reader to the surveys \cite{MR3857602,dhar2006theoretical} for further background.

Let $d\geq 1$ and let $K \subseteq \Z^d$ be finite.  A \textbf{sandpile} on $K$ is a function $\eta: K \to \{0,1,\ldots\}$, where $\eta(x)$ represents the number of grains of sand at $x$. A sandpile $\eta$ is said to be \textbf{unstable} at $x$ if $\eta(x) \geq 2d$. If $\eta$ is unstable at $x$ we can \textbf{topple} $\eta$ at $x$, 
decreasing the value of $\eta(x)$ by $2d$ and increasing $\eta(y)$ by $1$ for each neighbour $y$ of $x$ in $K$. 
It is a theorem of Dhar~\cite{Dhar90} that if $K$ is finite and $\eta$ is a sandpile on $K$ then carrying out successive topplings of unstable vertices will eventually result in a stable sandpile that does \emph{not} depend on the order in which the topplings are made.
Repeatedly adding a grain of sand to a uniform random vertex of $K$ and stabilizing the resulting configuration defines a Markov chain on the set of stable sandpile configurations on $K$. This Markov chain has a unique closed communicating class, consisting of the \emph{recurrent} configurations, and has a unique stationary distribution equal to the uniform measure on the set of recurrent configurations. We are particularly interested in studying the distribution of the \emph{avalanche} of topplings that is produced when we perform a single step of the Markov chain at equilibrium, which is expected to exhibit many interesting critical-like properties.

Majumdar and Dhar \cite{MajDhar92} discovered a bijection between 
 recurrent sandpiles and spanning trees, known as the \emph{burning bijection},
which allows us to relate many questions about sandpiles to questions about spanning trees. Athreya and J\'arai \cite{MR2077255} used the burning bijection to prove that there is a well-defined infinite volume uniform recurrent sandpile on $\Z^d$ for each $d\geq 2$ 
which can be obtained by applying the burning bijection to the USF.  
Let $\Eta$ (capital $\eta$) be a uniform recurrent sandpile on $\Z^d$, and suppose that we add a single grain of sand to the origin and then attempt to stabilize the resulting configuration. The \textbf{avalanche} is defined to be the \emph{multiset} $\Av_0(\Eta)$ of vertices counted according to the number of times they topple when this stabilization is performed, while the \emph{set} $\AvC_0(\Eta)$ of vertices that topple at least once is called the \textbf{avalanche cluster}.
J\'arai and Redig~\cite{JarRed08} related the distribution of the avalanche cluster $\AvC_0(\Eta)$ to that of the tree $\fT_0$ and used this to deduce that the avalanche on $\Z^d$ is finite almost surely when $d\geq 3$. This methodology was further refined by Bhupatiraju, Hanson, and J\'arai \cite{MR3718713}, who developed general techniques for comparing the critical behaviours of the Abelian sandpile with the WUSF and $0$-WUSF. Roughly speaking, they show that the avalanche cluster is sandwiched between the two random sets $\fP(0)$ and $\fT_0$ in the stochastic ordering; the precise statements they prove are a little more complicated than this but are applicable in much the same way. See \cite[Section 9]{1804.04120} for a summary.

 In high dimensions the two random sets $\fP(0)$ and $\fT_0$ have similar behaviour up to constants, so that these comparison inequalities yield up-to-constants estimates on the distribution of the avalanche \cite[Theorem 1.7]{1804.04120}. Applying the same methods with our main theorems as input yields the following bounds on the four dimensional model: the upper bounds are corollaries of \cref{thm:wusfo} while the lower bounds are corollaries of \cref{thm:wusf,thm:extrinsic}. 

\begin{corollary}
\label{thm:sandpile}
Let $\Eta$ be a uniform recurrent sandpile on $\Z^4$. Then 
\begin{align*}
 \frac{(\log n)^{2/3}}{n^2} &\lesssim&&
 \P\Bigl(\radext\left(\AvC_0(\Eta)\right) \geq n\Bigr) &&\lesssim \frac{(\log n)^{1+o(1)}}{n^2} &&\text{and} \\
\frac{(\log n)^{1/6}}{n^{1/2}} &\lesssim&&
\P\Bigl(|\AvC_0(\Eta)| \geq n\Bigr) &&\leq \P\Bigl(|\Av_0(\Eta)| \geq n\Bigr)  &&\lesssim \frac{(\log n)^{1/2+o(1)}}{n^{1/2}} 
\end{align*}
for every $n\geq 2$.
\end{corollary}

The deduction of \cref{thm:sandpile} from \cref{thm:wusf,thm:extrinsic,thm:wusfo} is identical to the argument of \cite[Section 9]{1804.04120} and is omitted. The previous best bounds on these quantities were due to Bhupatiraju, Hanson, and J\'arai~\cite{MR3718713} who proved that 
\begin{align*}
 \frac{1}{n^2(\log n)^{1/3}} &\lesssim&&
 \P\Bigl(\radext\left(\AvC_0(\Eta)\right) \geq n\Bigr) &&\lesssim \frac{1}{n^{1/4}} &&\text{and} \\
\frac{1}{n^{1/2}(\log n)^{5/6}} &\lesssim&&
\P\Bigl(|\AvC_0(\Eta)| \geq n\Bigr) &&\lesssim \frac{1}{n^{1/16}} 
\end{align*}
for every $n\geq 2$.
Although \cref{thm:sandpile} does not determine the precise order of the logarithmic corrections for the 4d Abelian sandpile model, it does suffice to show that these logarithmic corrections are 
non-trivial, i.e., that the exponents describing these logarithmic corrections are positive if they are well-defined.
  Computation of the precise order of the logarithmic correction to scaling in four dimensions appears to require a substantial new idea, and we are not aware of any conjectured values for the relevant exponents. Further interesting open problems that may be of intermediate difficulty concern the distribution of the total number of times the origin topples and the probability that $x$ topples at all when $x$ is large; see \cite{MR3718713} for partial results and \cite{angel2020tail,MR3395472,MR3962482} for analogous results for branching random walk.



\subsection{About the proof and organization}
\label{subsec:about_the_proof}

We now give a brief overview of our proof, including a heuristic computation of the relevant logarithmic corrections. As in \cite{1804.04120}, our proof utilizes the interplay between two different ways of sampling the uniform spanning tree: \emph{Wilson's algorithm} and the \emph{interlacement Aldous--Broder algorithm}. The latter algorithm, introduced in \cite{hutchcroft2015interlacements}, extends the classical Aldous--Broder algorithm \cite{Aldous90,broder1989generating} to infinite transient graphs by replacing the random walk in the classical algorithm with Sznitman's \emph{random interlacement process} \cite{Sznitman,Teix09}.

Let us now recall the key features of the interlacement Aldous--Broder algorithm; detailed definitions are given in \cref{sec:AldousBroder}. Let $d\geq 3$.
The \emph{random interlacement process} $\sI$  on $\Z^d$ is a Poisson point process on $\cW^* \times \R$, where $\cW^*$ is the space of paths in $\Z^d$ modulo time shift and the $\R$ coordinate of each point is thought of as an arrival time. Intuitively, we think of this process as a Poissonian soup of bi-infinite `random walk excursions from infinity'. The most important property of this process for the present discussion is that the set of times in which any given finite set $K \subseteq \Z^d$ is visited is a Poisson process of intensity given by the \textbf{capacity} (a.k.a.\ conductance to infinity) of $K$, defined by $\cpc{}{K} = 2d\sum_{x\in K} \P_x(\tau_K^+=\infty)$ where $\P_x$ denotes the law of a simple random walk started at $x$ and $\tau^+_K$ denotes the first positive time the walk returns to $K$.  
For each $t\in \R$ and $x\in \Z^d$ we write $\sigma_t(x)$ for the first time after $t$ that $x$ is visited by a trajectory of the random interlacement process and write $e_t(x)$ for the oriented edge that is traversed by this trajectory as it enters $x$ for the first time. For each $t\in \R$, it is proven in \cite{hutchcroft2015interlacements} that
the set of reversals $\AB_t(\sI):=\{e_t(x)^\leftarrow : x\in \Z^d\}$ is distributed as the uniform spanning forest of $\Z^d$ oriented so that each vertex has exactly one oriented edge emanating from it. In this formulation, the future of $x$ can be defined as the unique oriented path emanating from $x$ and the past of $x$ can be defined as the tree spanned by those vertices having $x$ in their future. 

Varying $t$ allows us to think of the uniform spanning forest dynamically, and it turns out that the past of a vertex evolves in a fairly tractable way under these dynamics. Indeed, as we \emph{decrease} $t$, the past of a vertex $x$ becomes monotonically smaller except possibly at those times when $x$ is itself visited by a trajectory. More precisely, on the event that $x$ is not visited by a trajectory between times $s$ and $t$, a point $y$ lies in the past of $x$ in $\AB_s(\sI)$ if and only if it lies in the past of $x$ in $\AB_t(\sI)$ and the path connecting $y$ to $x$ in $\AB_t(\sI)$ is not hit between times $s$ and $t$ (\cref{lem:PastDynamics}). Since the restriction of the interlacement process to $[s,t]$ is independent of $\AB_t(\sI)$, a path $\Gamma$ lying in the past of $x$ in $\AB_t(\sI)$ also belongs to the past in $\AB_s(\sI)$ with conditional probability of order $\exp\left(-|t-s|\cpc{}{\Gamma}\right)$ on the event that $x$ is not hit in $[s,t]$.

Since paths in the uniform spanning tree are distributed as loop-erased random walks by Wilson's algorithm, it will therefore be important for us to understand the capacity of loop-erased random walks. This is the primary subject of \cref{part:LERW} of the paper. In \cref{thm:expectlerwcap}, we establish a very general estimate showing that the \emph{expected} capacity of the simple random walk $X^n:=(X_i)_{i=0}^n$ and its loop-erasure $\LE(X^n)$ are of the same order on any transient graph. It is known that the simple random walk $X^n$ on $\Z^4$ has capacity of order $n/\log n$ with high probability \cite{Law91,ASS}, and we prove that the same is true for the loop-erasure $\LE(X^n)$ in \cref{prop:firstconcentration}. Since $\LE(X^n)$ has length of order $n /(\log n)^{1/3}$ with high probability by the results of Lawler \cite{Law91,Lawlerlog}, this suggests that paths of length $m$ in the uniform spanning tree should typically have capacity of order $m/(\log m)^{2/3}$. It will be important for us to have reasonably good concentration estimates to this effect, which are proven in \cref{sec:capacity,subsec:polygamous_deviants}.

\cref{part:UST} of the paper applies the results of \cref{part:LERW} to study the uniform spanning tree. The first section of \cref{part:UST}, \cref{sec:UST_background}, provides relevant background on the model. The second, \cref{sec:intrinsic}, is primarily devoted to our results on the tail of the intrinsic radius.
Let us now give a brief overview of how these results are proven. We argue heuristically that the considerations of the previous two paragraphs lead to the relation
\begin{align}
\P(\radint(\fP)\geq n) &\approx \frac{\left(\text{non-intersection probability of an $n$-step LERW with an infinite SRW}\right)}{\left(\text{typical capacity of an $n$-step LERW}\right)}
\nonumber\\
& \approx \frac{1}{(\log n)^{1/3}} \cdot \frac{(\log n)^{2/3}}{n},
\label{eq:heuristic}
\end{align}
where the estimate of the non-intersection probability is due to Lawler \cite{Law91,Lawlerlog}. 
We will assume for the sake of this discussion that all length $n$ paths in the uniform spanning tree have capacity between $cn/(\log n)^{2/3}$ and $Cn/(\log n)^{2/3}$ for some positive constants $c$ and $C$; bounding the relevant error terms in order to justify this approximation accounts for much of the technical work in the paper. We first explain the lower bound, which will follow by considering an explicit strategy for generating a large past in the interlacement Aldous--Broder algorithm. Let $\eps=\eps(n)>0$ and let $\sA_n$ be  the event that the origin is visited by a unique trajectory $W$ between times $0$ and $\eps$. When we apply the Aldous--Broder algorithm to the positive part of $W$ (i.e., the part of $W$ after its first visit to $x$) we obtain a tree that contains a unique infinite path starting from the origin which is distributed as a loop-erased random walk (this is not the loop-erasure of the positive part of $W$ but rather a sort of infinite reverse loop-erasure; see \cref{subsec:AB_variant}). If the first $n$ steps of this path are not hit by the negative part of $W$ or by any other trajectory arriving between times $0$ and $\eps$, then this initial segment is contained in the past of the origin in $\AB_0(\sI)$. We expect these two non-intersection events to be approximately independent, suggesting that
\begin{equation*}
\P(\radint(\fP)\geq n \mid \sA_n) \gtrsim 
\left(\text{non-intersection probability}\right)\cdot\exp\left(-\frac{C n}{(\log n)^{2/3}} \cdot \eps\right).
\end{equation*}
Since $\P(\sA_n) \approx \eps$, the lower bound of \eqref{eq:heuristic} follows by taking $\eps =  (\log n)^{2/3} n^{-1}$. A similar heuristic calculation leads to the lower bound on the extrinsic radius of the past in the UST and on the intrinsic and extrinsic radii of the component of the origin in the $0$-WUSF (for which the analogue of the interlacement Aldous--Broder algorithm is discussed in \cref{sec:AldousBroder}). In the later case, the non-intersection probability does not feature in the computations, leading to the $(\log n)^{1/3}$ difference between the two tail probabilities.

The upper bound on the intrinsic radius is more delicate. 
Let $\partial \fP(x,n)$ be the set of vertices belonging to the past of $x$ that have intrinsic distance exactly $n$ from $x$, and let $Q(n)$ be the probability that this set is non-empty. It suffices to prove an inductive inequality of the form
\begin{equation}
\label{eq:inductive_overview}
Q(2n) \leq \frac{C(\log n)^{1/3}}{n} + \frac{1}{4} Q(n).
\end{equation}
(The only important feature of the constant $1/4$ is that it is strictly smaller than $1/2$.)
Proving inductive inequalities such as these is a common approach to arm-exponents in high-dimensional models which we believe first arose in the work of Kozma and Nachmias \cite{Armexpon,AlexOrbach}. (They are also a standard approach to the study of various four-dimensional loop-erased random walk quantities, see e.g.\ \cite[Section 4.4]{Law91}.)
Once again we take $\eps=\eps(n)>0$, but now consider a union bound according to whether the arrival time $\sigma_0(0)$ is smaller or larger than $\eps$. 
We argue in \cref{lem:exceptional_time} that the argument of the previous paragraph admits an approximate converse
\begin{equation}
\label{eq:early_arrival_overview}
\P(\partial \fP(0,n) \neq \emptyset \mid \sigma_0(0) \leq \eps) \lesssim \left(\text{non-intersection probability}\right) \asymp \frac{1}{(\log n)^{1/3}}
\end{equation}
provided that $\eps$ is not too large. This is a rather technical matter and in fact relies upon the computation of the upper bound on the intrinsic radius for the $0$-WUSF and the stochastic domination property; this upper bound on the $0$-WUSF is itself proven using a similar (but simpler) inductive scheme to that discussed here. For the other term in the union bound, we observe that if $\partial \fP(0,2n)$ is non-empty then there must exist $y \in \partial \fP(0,n)$ such that $\partial \fP(y,n)$ is non-empty. Let $\Gamma$ be the future of $0$ and let $\Gamma^n$ denote the first $n$ steps of $\Gamma$. By Markov's inequality and the mass-transport principle (as explained in the proof of Lemma~\ref{lemma:WUSF_ball}), we deduce that
\[
\P(\partial \fP(0,2n) \neq \emptyset, \sigma_0(0) \geq \eps) \leq \P(\Gamma^n \text{ is not hit between times $-\eps$ and $0$} \text{ and } \partial\fP(0,n) \neq \emptyset).
\]
Using the splitting property of Poisson processes, we deduce that
\begin{equation}
\label{eq:late_arrival_overview}
\P(\partial \fP(0,2n) \neq \emptyset, \sigma_0(0) \geq \eps) \leq \exp\left(-\frac{cn}{(\log n)^{2/3}} \cdot \eps\right)Q(n) + \cE(n),
\end{equation}
where $\cE(n)$ is an error term accounting for the possibility of $\Gamma^n$ having small capacity. 
In \cref{pro:o(n)} we apply the results of \cref{part:LERW} together with the upper bound on the tail of the intrinsic radius of the $0$-WUSF to prove that $\cE(n)=o(n^{-1}(\log n)^{1/3})$. The desired inequality \eqref{eq:inductive_overview} follows from \eqref{eq:early_arrival_overview} and \eqref{eq:late_arrival_overview} by taking $\eps$ to be $C' n^{-1} (\log n)^{2/3}$ for appropriately large $C'$.

In \cref{sec:volume} we apply our results on the intrinsic radius to prove our results concerning the volume. The upper bounds on the tail of the volume follow immediately from our results on the intrinsic radius via a standard truncated first moment argument (\cref{prop:volume_upper}). The lower bounds are proven using a second moment argument. While this is also a standard strategy, the required second moment upper bounds are non-trivial to obtain. As a part of this proof we prove various estimates on the geometry of the restriction of the uniform spanning tree to a box which may be of independent interest.

Finally, in \cref{sec:extrinsic} we apply our results on the intrinsic radius to prove the upper bounds on the tail of the extrinsic radius, the lower bounds having already been proven in \cref{sec:intrinsic}.
The heuristic outline of the argument is as follows: When sampling the UST with Wilson's algorithm, we expect that paths of length $n$ in the UST should be generated by walks with length of order $n (\log n)^{1/3}$. As such, it should be very difficult for the past of the origin to have extrinsic radius at least $\sqrt{n (\log n)^{1/3} (\log \log n)^2}$, say, without having intrinsic radius at least $n$, from which the result would follow. Unfortunately, we were not able to prove a sufficiently strong concentration estimate on the length of the loop-erased random walk to push such an argument through. To circumvent this problem, we introduce the notion of the \emph{typical time} $T(\eta)$ of a self-avoiding path $\eta$. We show that a random walk conditioned to have $\eta$ as its loop-erasure has length of order $T(\eta)$ with high probability (\cref{lem:typicaltime}), and prove that if $\eta$ is distributed as a loop-erased random walk then $T(\eta)$ is of order $|\eta| (\log |\eta|)^{1/3+o(1)}$ with very high probability (\cref{lem:geomfixedtime}). Be careful to note the subtle distinction between this claim and a true concentration estimate for the length of a loop-erased random walk! Once this is done, we apply these estimates together with our results on the intrinsic radius to prove that the past contains a path of typical time at least $m$ with probability of order at most $m^{-1}(\log m)^{2/3+o(1)}$ (\cref{prop:timeradius}). In \cref{subsec:extrinsicproof} we complete the proof by 
showing that it is very difficult for 
the past of the origin to have extrinsic radius at least $\sqrt{n (\log n)^{1/3} (\log \log n)^2}$ without containing a path of length at least $n$ or typical time at least $n (\log n)^{1/3}$, and the claimed upper bound on the tail of the extrinsic radius follows.

\begin{remark}
In \cite{1804.04120}, the proof of the upper bound on the tail of the extrinsic radius followed a completely different method to that used here, based on the notion of \emph{pioneer points}. Unfortunately we were not able to prove the required upper bounds on the number of pioneer points needed to implement this strategy in four dimensions. 
\end{remark}

\part{Loop-erased random walk}
\label{part:LERW}

\section{Background}

In this section we start by setting up some notation and recalling the definition and some properties of capacity that will be used later in the proofs. 

\medskip

\textbf{Norms and balls.}
The Euclidean norm of $x\in \Z^4$ is denoted $\|x\|=\|x\|_2$, and the Euclidean ball of center $x$ and radius~$r$ is denoted $B(x,r) \subseteq \Z^4$. We also write $\Lambda_r=[-r,r]^4 \cap \Z^4$ for the box of side length $2r$ centred at the origin, i.e., the $\|\cdot\|_\infty$-ball of radius $r$.

\medskip

\textbf{Hitting times and the Green's function.} We write $\P_x$ for the law of simple random walk $X$ started at $x$ on $\Z^4$. For a set $A\subseteq \Z^4$ we define 
\[
\tau_A=\min\{t\geq 0: X_t\in A\} \ \text{ and } \ \tau_A^+=\min\{t\geq 1: X_t\in A\}. 
\]
If $A=\{x\}$, we simply write $\tau_x=\tau_{\{x\}}$. 
The discrete \textbf{Green's function} $\bG$ is defined by 
$$
\bG(x,y) \ =\ \sum_{k\ge 0} \prstart{X_k=y}{x},\quad\text{ so that }
\quad \frac{\bG(x,y) }{\bG(0,0)}=\ \prstart{\tau_y<\infty}{x}
$$
for every $x,y \in \Z^4$. We also write $\bG(x)=\bG(0,x)$. The function $\bG$ is symmetric, satisfies $\bG(x,y)=\bG(y-x)$ and the estimate
\begin{equation}
\label{Green.bound}
\bG(x) \asymp \frac{1}{\norm{x}^2+1}
\end{equation}
for every $x\in \Z^4$ \cite[Theorem 4.3.1]{LawLim}.
We will make repeated use of the \textbf{last exit decomposition}, which states that
\begin{equation}
\label{eq:last_exit}
\P_x(\tau_A <\infty) = \sum_{a\in A} \bG(x,a) \P_a(\tau_A^+ =\infty)
\end{equation}
for every $x\in \Z^4$ and every finite set $A \subseteq \Z^4$.

\medskip

\textbf{Notation for paths.}
Let $G$ be a graph, which we will usually take to be the hypercubic lattice $\Z^4$. For each $-\infty \leq n \leq m \leq \infty$, let $L(n,m)$  be the graph with vertex set $\{i \in \Z : n \leq i\leq m\}$ and edge set $\{\{i,i+1\}: n\leq i \leq m-1\}$. We define a \textbf{path} in $G$ to be a multigraph homomorphism from $L(n,m)$ to $G$ for some $-\infty \leq n \leq m \leq \infty$. We can consider the random walk $X$ on $G$ to be a random \emph{path} by keeping track of the edges it traverses as well as the vertices it visits. (In our primary setting of $\Z^4$ this distinction is of little consequence as there is at most one edge between any pair of vertices.) Given a path $w : L(n,m)\to G$ we will use  $w(i)$ and $w_i$ interchangeably to denote the  vertex visited by $w$ at time $i$ and use $w(i,i+1)$ and $w_{i,i+1}$ interchangeably to denote the oriented edge crossed by $w$ between times $i$ and $i+1$.
Given $n\leq a \leq b \leq m$, we write $w[a,b]$ for the restriction of $w$ to $L(a,b)$. Similar conventions apply to open and half-open intervals, so that e.g.\ $w(a,b]$ is the restriction of $w$ to $L(a+1,b)$. In the case that $n=a$, we also use the notation $w^b = w[n,b]$ and call $w^b$ the path $w$ \emph{stopped at $b$}. In particular, we will often use the notation $X^T$ for the random walk $X$ stopped at some (possibly random) time $T$.

We will often abuse notation and write simply $w$ for both the path $w$ and the set of vertices it visits. In particular, we will use $X[a,b]$ to denote both the portion of $X$ between times $a$ and $b$ and the set of vertices visited by $X$ between times $a$ and $b$; the precise meaning will be clear from context.




\subsection{Loop-erased random walk}
\label{subsec:LERWbackground}

Let $G$ be a graph. A path in $G$ is said to be \textbf{transient} if it visits each vertex of $G$ at most finitely many times. (In particular, every finite path in $G$ is transient.) Given $0\leq m \leq \infty$ and a transient path $w : L(0,m) \to G$, we define the sequence of times $\ell_n(w)$ recursively by $\ell_0(w)=0$ and 
\[\ell_{n+1}(w) = 1+\max\{ k : w_k = w_{\ell_n}\},\] where we terminate the sequence the first time that $\max\{ k : w_k = w_{\ell_n}\}=m$ when $m<\infty$. The \textbf{loop-erasure}  $\LE(w)$ of $w$ is the path in $G$ defined by
\[\LE(w)_i = w_{\ell_i(w)} \qquad \LE(w)_{i,i+1} = w_{\ell_{i+1}-1,\ell_{i+1}}.\]
Informally, $\LE(w)$ is the path formed from $w$ by erasing cycles chronologically as they appear. The loop-erasure of simple random walk is known as \textbf{loop-erased random walk}. The theory of loop-erased random walk was both introduced and developed extensively by Lawler \cite{LawlerSAW,Loop,Law91}, whose results on four-dimensional loop-erased random walk \cite{Law82,Lawlerlog} shall play a very important role in this paper.

Let $X$ be a simple random walk on $\Z^4$ started from $0$, and let $\ell_n=\ell_n(X)$ for each $n\geq 0$. Be careful to note that $\LE(X^n)$ denotes the loop-erasure of the walk run up to time $n$ (which is of random length), while $\LE(X)^n$ denotes the first $n$ steps of the \emph{infinite} loop-erased random walk $\LE(X)$. We define
\[
 \rho_n=\sum_{k=0}^{n} \1(k=\ell_i \text{ for some $i\geq 0$}) = \max\{m\geq 0: \ell_m \leq n\}\ \  
\]
for each $n\geq 0$, which counts the number of points up to time $n$ that are not erased when computing the loop-erasure of $X$.
The sequences $(\ell_n)_{n\geq 0}$ and $(\rho_n)_{n\geq 0}$ are inverses of each other in the sense that
\begin{align}\label{eq:equivtnrhom}
\ell_n\leq m \quad \text{ if and only if } \quad \rho_m\geq n
\end{align}
for each $n,m\geq 0$. 
 For each $n\geq 0$ we also define $\eta_n = \eta_n(X) = \max\{\ell_k : k\geq 0, \ell_k \leq n\}$ to be the maximal time before $n$ contributing to the infinite loop-erasure and define 
\begin{align}\label{eq:definflerw}
\LE_\infty(X^n):=
\LE(X^{\eta_n}) = \LE(X^n)^{\rho_n} = \LE(X)^{\rho_n},
\end{align}
where the equality of these expressions follows from the definitions. (This is a slight abuse of notation since $\LE_\infty(X^n)$ is not a function of $X^n$ but depends on the entire walk $X$.) We think of $\LE_\infty(X^n)$ as the part of the infinite loop-erasure $\LE(X)$ that is contributed by the first $n$ steps of the walk: it is an initial segment of both $\LE(X)$ and $\LE(X^n)$.

The following theorem of Lawler tells us that, in four dimensions, $\rho_n$ and $\ell_n$ are weakly concentrated around $n(\log n)^{-1/3}$ and $n (\log n)^{1/3}$ respectively. (In contrast, $\rho_n$ and $\ell_n$ are approximately linear in $n$ when $d>4$ \cite[Theorem 7.7.2]{Law91} and differ from $n$ by a power when $d<4$ \cite{Ken00,Daisuke,MR1703133}.)



\begin{theorem}[\cite{Law91}, Theorem 7.7.5]
\label{lem:amounterased}
	Let $X$ be a simple random walk on $\Z^4$. Then  
	\begin{align*}
	\pr{\left| \frac{\rho_n}{n(\log n)^{-1/3}} -1\right|>\epsilon} &\lesssim_\eps \frac{\log \log n}{(\log n)^{2/3}} \qquad \text{and hence}\\ \pr{\left| \frac{\ell_n}{n(\log n)^{1/3}} -1\right|>\epsilon} &\lesssim_\eps \frac{\log \log n}{(\log n)^{2/3}} 
	\end{align*}
	for every $\eps>0$ and $n\geq 3$.
\end{theorem}

\begin{proof}
This theorem is \emph{almost} proven in~\cite[Theorem~7.7.5]{Law91}. The quantitative bound on the probability is not stated there, but can easily be inferred by combining the proof given there  with the fact that $\P(n = \ell_i$ for some $i\geq 0)\sim(\log n)^{-1/3}$ as $n\to\infty$, which was proven by Lawler in his later work \cite{Lawlerlog}. (The details are similar to, but simpler than, the proof of our \cref{prop:firstconcentration}.)
\end{proof}


We will also use the following closely related result of Lawler \cite{Lawlerlog} on avoidance probabilities for simple random walk and loop-erased random walk. See also \cref{lem:displacement} for a useful minor variation on the same estimate and \cite{LawlerSunWu} for more precise asymptotic estimates.

\begin{theorem}[\cite{Lawlerlog}]\label{thm:logpaper}
	Let $X$ and $Y$ be independent random walks on $\Z^4$ both started at the origin. Then 
	\begin{align*}
		\pr{X(0,\infty)\cap \LE(Y^n) =\emptyset} \asymp \pr{X(0,n)\cap \LE(Y) =\emptyset} \asymp \frac{1}{(\log n)^{1/3}} 
	\end{align*}
	for every $n\geq 2$.
\end{theorem}


\subsection{Capacity of random walk}
In this section we recall the definition of the \emph{capacity} of a set in $\Z^4$ and state several results on the capacity of simple random walk due to the second author, Asselah, and Schapira \cite{ASS,Asselah:2016hccc} which will play an important role in our analysis of the uniform spanning tree. Further background on capacities can be found in e.g.\ \cite{LP:book}.

Let $G=(V,E)$ be an infinite transient network with positive edge conductances $(c_e)_{e\in E}$. We will usually take $G$ to be the hypercubic lattice $\Z^4$ with unit edge conductances. The \textbf{capacity} of a finite set $A \subseteq V$ is defined by
\[
\cpc{}{A} = \sum_{x\in A}c(x)\prstart{\tau_A^+=\infty}{x},
\]
where we define $c(x)$ to be the total conductance of all oriented edges emanating from $x$ for each vertex $x\in V$.
The capacity is also known as the \textbf{conductance to infinity}.  
We will make repeated use of the fact that the capacity is both \emph{increasing} and \emph{subadditive} \cite[Proposition 2.3.4]{Law91}: If $A$ and $B$ are finite sets of vertices in an infinite transient network then
\begin{equation}
\cpc{}{A} \leq \cpc{}{A \cup B} \leq \cpc{}{A} + \cpc{}{B}.
\end{equation}
The following proposition provides a partial converse to this subadditivity estimate.

\begin{prop}[\cite{ASS}, Proposition~1.6]
\label{pro:deccap}
Let $A$ and $B$ be two finite subsets of $\Z^d$. We have that
\begin{equation}\label{dec.intro}
\cc{A\cup B} = \cc{A} + \cc{B} - \chi(A,B) -\chi(B,A)+\epsilon(A,B),
\end{equation}
where
\[
\chi(A,B) = 2d\sum_{y\in A} \sum_{z\in B} \prstart{\tau^+_{A\cup B}=
\infty}{y} \bG(y,z)\prstart{\tau_B^+=\infty}{z}
\]
and $0\leq\epsilon(A,B) \leq\cc{A\cap B}$.
\end{prop}

(Note that the statement given here differs from that given in~\cite{ASS} since 
 our definition of the capacity differs from the one given there by a factor of $2d$.)

It will be useful for us to have a coarser version of this result involving quantities that are easier to work with.
For $A,B$ finite subsets of $\Z^d$ we define
\[
\til{\chi}(A,B)= 2d\sum_{y\in A} \sum_{z\in B} \prstart{\tau^+_{A}=
\infty}{y} \bG(y,z)\prstart{\tau_B^+=\infty}{z}.
\]
We obviously have that $\chi(A,B)\leq \til{\chi}(A,B)$, and consequently by \cref{pro:deccap} that
\begin{align}\label{eq:obvious}
\cc{A\cup B} \geq \cc{A} + \cc{B} - \tilde\chi(A,B) -\tilde\chi(B,A)
\end{align}
for every two finite subsets $A$ and $B$ of $\Z^4$. This inequality is more helpful to us than \eqref{dec.intro} due to the following useful property.

\begin{lemma}
\label{lem:monotonicityoftilchi}
	Let $A\subseteq A'$ and $B\subseteq B'$ be four finite subsets of $\Z^d$. Then 
	$\til{\chi}(A,B) \leq \til{\chi}(A',B')$.
\end{lemma}

In particular, when $A$ and $B$ are two segments of a loop-erased random walk, we will be able to bound the associated cross-terms $\tilde \chi(A,B)$ and $\tilde \chi(B,A)$ in terms of the cross-terms associated to the \emph{simple} random walk.

\begin{proof}[Proof of \cref{lem:monotonicityoftilchi}]
Using the last exit decomposition formula we obtain
\begin{align*}
	\til{\chi}(A,B)=  2d\sum_{x\in A} \prstart{\tau_A^+=\infty}{x}\prstart{\tau_B<\infty}{x}  \leq  2d\sum_{x\in A} \prstart{\tau_A^+=\infty}{x}\prstart{\tau_{B'}<\infty}{x}. 
\end{align*}
	Applying the last exit decomposition formula two more times and using the symmetry of the Green's function $\bG(x,y)$ we obtain that
		\begin{align*}
	&	\sum_{x\in A} \prstart{\tau_A^+=\infty}{x}\prstart{\tau_{B'}<\infty}{x} = \sum_{x\in A}\prstart{\tau_A^+=\infty}{x} \sum_{y\in B'}\bG(x,y) \prstart{\tau_{B'}^+=\infty}{y}\\
		&= \sum_{y\in B'} \prstart{\tau_{B'}=\infty}{y} \prstart{\tau_A<\infty}{y}
		\leq \sum_{y\in B'} \prstart{\tau_{B'}=\infty}{y} \prstart{\tau_{A'}<\infty}{y} = \frac{1}{2d}\til{\chi}(A',B')
	\end{align*}
as required.
\end{proof}

The capacity of the simple random walk is understood very well in every dimension \cite{Asselah:2016hccc,ASS}. We will use the following estimates giving control of the mean capacity and concentration around this mean in four dimensions. 

\begin{lemma}\label{lem:known}
	Let $X$ be a simple random walk in $\Z^4$. Then
	\[
	\E{\cpc{}{X^n} }\asymp \frac{n}{\log n} \quad \text{ and }\quad  \E{\left|\cc{X^n} - \E{\cc{X^n}}\right|^k } \lesssim_k \frac{n^k}{(\log n)^{2k}}
	\]
	for every $k\geq 1$ and $n\geq 2$.
\end{lemma}

\begin{proof}[Proof of \cref{lem:known}]
The first statement follows from~\cite[Corollary~1.4]{Asselah:2016hccc} which is a direct consequence of Lawler's two sided non-intersection estimate \cite{Law91}. (See also \cite[Lemmas 5.5 and 5.6]{1804.04120} for a simple argument giving a lower bound of the correct order in every dimension.)
The second statement follows from~\cite[Lemma~7.1]{ASS}.
\end{proof}

We will also use the following bound on the cross-terms. (Note that $n\geq 3 > e$ ensures that $\log n$ and $\log \log n$ are both positive.)
\begin{lemma}\label{lem:secondmoment}
	Let $X$ be a simple random walk in $\Z^4$. Then there exists a positive constant $C$ such that
	\begin{align*}
		\E{\left(\til{\chi}\left(X[0,n], X[n,2n]\right)\right)^2} \leq Cn^2 \cdot \frac{(\log\log n)^2}{(\log n)^4}
	\end{align*}
	for every $n\geq 3$.
\end{lemma}

The proof of this lemma can be found in the second part of the proof of Lemma~3.1 of~\cite{ASS}, where a bound on $\til{\chi}$ is implicitly established (although the notation $\tilde \chi$ is not used).

\section{The capacity of loop-erased random walk}
\label{sec:capacity}

In this section we prove several estimates concerning the capacity of the four-dimensional loop-erased walk that will be used when analysing the UST via the interlacement Aldous--Broder algorithm.
Both upper and lower bounds on the capacity of the LERW will be useful for our analysis of the UST: Generally speaking, we will want lower bounds on the capacity when proving upper bounds on arm events for the UST, and upper bounds on the capacity for proving lower bounds on arm events for the UST. We compute the order of $\E{\cpc{}{\ler{}{X^n}}}$ in \cref{subsec:expected_capacity} and establish lower and upper tail estimates in \cref{subsec:capacity_lower_tail,subsec:capacity_upper_tail} respectively.

\subsection{The expected capacity}
\label{subsec:expected_capacity}


We begin with the following theorem, which states very generally that the \emph{expected} capacity of the simple random walk and loop-erased random walk are of the same order. (In the other direction, we trivially have that $\cpc{}{\ler{}{X^n}} \leq \cpc{}{X^n}$ by monotonicity of capacity.)
	
\begin{thm}
\label{thm:expectlerwcap}
Let $G=(V,E)$ be a transient, locally finite network, let $x\in V$, and let $X$ be a random walk on $G$ started from $x$. Then
\[
\E{\cpc{}{\ler{}{X^n}}} \geq \frac{1}{256}\, \E{\cpc{}{X^n}}
\]
for every $n\geq 0$.
\end{thm}

Applying this result together with \cref{lem:known} yields the following immediate corollary.

\begin{corollary}
\label{cor:expectedlerwcapZ4}
Let $X$ be simple random walk on $\Z^4$. Then
\[
\E{\cpc{}{\ler{}{X^n}}} \asymp \frac{n}{\log n}
\]
for every $n\geq 2$.
\end{corollary}

\cref{thm:expectlerwcap} is closely related to the results of Lyons, Peres, and Schramm \cite{LyonsPeresScrhamm}, and will be deduced using techniques similar to that paper. In particular, we will apply Theorem 3.1 of that paper, which we now state in full.

\begin{theorem}[\cite{LyonsPeresScrhamm}, Theorem~3.1]
\label{thm:LPS}
Let $X$ and $Y$ be independent Markov chains on countable state spaces\footnote{In \cite{LyonsPeresScrhamm} the authors require the two chains to have the same state space. This condition is redundant since they do not require the two chains to have the same transition probabilities: one can always take the two state spaces to be $\N$ without loss of generality.} $\mathcal{S}_1$ and $\mathcal{S}_2$ and with initial states $x_0$ and $y_0$ respectively. Let $A \subseteq \N \times \mathcal{S}_1 \times \N \times \mathcal{S}_2$, and let $\mathrm{hit}(A)$ be the event that $(n,X_n,m,Y_m) \in A$ for some $n,m \geq 0$. Given any weight function $w:\N \times \mathcal{S}_1 \times \N \times \mathcal{S}_2 \to [0,\infty)$ that vanishes outside of $A$, consider the random variable
\[
S_w = \sum_{n,m=0}^\infty w(n,X_n,m,Y_m).
\]
If $\P(\mathrm{hit}(A))>0$ then there exists such a weight function  $w:\N \times \mathcal{S}_1 \times \N \times \mathcal{S}_2 \to [0,\infty)$  vanishing outside of $A$ such that $0<\E{S_w^2} <\infty$ and
\begin{equation}
\label{eq:reverse_Cauchy_Schwarz}
\frac{\E{S_w}^2}{\E{S_w^2}} \leq \P\left(\mathrm{hit}(A)\right) \leq 64\frac{\E{S_w}^2}{\E{S_w^2}}.
\end{equation}
\end{theorem}

\medskip

This theorem is based on earlier work of Salisbury \cite{MR1395618} and Fitzsimmons and Salisbury~\cite{MR1023955}.
Note that the lower bound of \eqref{eq:reverse_Cauchy_Schwarz} is a trivial consequence of Cauchy-Schwarz and holds for \emph{any} choice of weight function; the content of the theorem is that, in this context, there always exists a weight function such that the reverse inequality holds up to multiplication by a universal constant. Similar but more explicit results for a \emph{single} Markov chain can be proven using the notion of \emph{Martin capacity} \cite{MR1349175}.


\medskip
The proof of \cref{thm:expectlerwcap} will also use the fact that the capacity of a set can be expressed as a limit of effective conductances in finite networks. Recall that if $G=(V,E)$ is a finite network and $A$ and $B$ are disjoint sets of vertices in $G$ then the \textbf{effective conductance} is defined to be
\[
\Ceff(A \leftrightarrow B;G)=\sum_{a\in A}c(a)\P_a(\tau_B <\tau^+_A) = \sum_{b\in B}c(b)\P_b(\tau_A <\tau^+_B).
\]
See \cite[Chapter 2]{LP:book} for background. Now suppose that $G$ is an infinite network and that $(V_n)_{n\geq 1}$ is an exhaustion of $V$ by finite sets. For each $n\geq 1$, we define $G_n^*$ to be the finite network obtained from $G$ by contracting every vertex in $V\setminus V_n$ down to a single vertex, denoted $\partial_n$, and deleting all the self-loops from $\partial_n$ to itself. Then 
\[
\cpc{}{A}=\Ceff(A \leftrightarrow \infty;G) =\lim_{n\to\infty}\Ceff(A \leftrightarrow \partial_n;G_n^*)
\]
for every finite set $A \subseteq V$.

\begin{proof}[Proof of \cref{thm:expectlerwcap}]
This proof will be a small variation on the proof of \cite[Lemma 4.1]{LyonsPeresScrhamm}. The main difference is that the walk $X$ is stopped at a deterministic time $n$ rather than at the hitting time of some set, which means that the exact symmetry used to prove eq.\ (4.3) of \cite{LyonsPeresScrhamm} does not hold. Luckily for us, this issue is fairly straightforward to resolve.


Fix $x\in V$. Let $(V_r)_{r\geq 1}$ be an exhaustion of $V$ by finite connected sets each of which contains $x$, and for each $r \geq 1$ let $G_r^*$ be defined as above.  Let $r\geq 1$ and let $X$ be a random walk started at $x$ and let $Y$ be an independent random walk started at $\partial_r$. Let $\tau$ be the first time $X$ visits $\partial_r$ and let $\kappa$ be the first positive time $Y$ visits $\partial_r$. It suffices to prove that
\begin{equation}
\label{eq:256finitevolume}
\P\left(\ler{}{X^{(\tau-1)\wedge n}} \cap Y^{\kappa-1} \neq \emptyset\right)
\geq \frac{1}{256}\P\left(X^{(\tau-1)\wedge n} \cap Y^{\kappa-1} \neq \emptyset\right)
\end{equation}
for each $r,n \geq 1$. Indeed, we will then have by the definitions that 
\begin{align*}
\E{\Ceff\left(\ler{}{X^{(\tau-1)\wedge n}} \leftrightarrow \partial_r ;G_r^*\right)} &= c(\partial_r)\P\left(\ler{}{X^{(\tau-1)\wedge n}} \cap Y^{\kappa-1} \neq \emptyset\right)
\\
& \geq \frac{c(\partial_r)}{256}\P\left(X^{(\tau-1)\wedge n} \cap Y^{\kappa-1} \neq \emptyset\right)
\\&
= \frac{1}{256}\E{\Ceff\left(X^{(\tau-1)\wedge n} \leftrightarrow \partial_r ;G_r^*\right)}
\end{align*}
for every $r,n\geq 1$, so that the claim will follow by taking $r\to\infty$ and applying the bounded convergence theorem.

We now begin the proof of \eqref{eq:256finitevolume}. Fix $r,n\geq 1$ and let 
\[A=\{(a,v,b,v) : 0 \leq a \leq n, b \geq 0, v \in V_r \}.\] Applying \cref{thm:LPS} to the stopped walks $X^\tau$ and $Y^\kappa$ we obtain that there exists $w : \N \times V_r \times \N \times V_r \to [0,\infty)$ that is supported on $A$ and satisfies
\begin{equation}
\label{eq:Swinitial}
\P\left(X^{(\tau-1)\wedge n} \cap Y^{\kappa-1} \neq \emptyset\right) = \P(\operatorname{hit}(A)) \leq 64 \frac{\E{S_w}^2}{\E{S_w^2}}
\end{equation}
where 
\[S_w = \sum_{a=0}^\infty \sum_{b=0}^\infty w(a,X_a^\tau,b,Y_b^\kappa)=\sum_{a=0}^{(\tau-1)\wedge n} \sum_{b=1}^{\kappa-1} w(a,X_a,b,Y_b).\]
(The second equality here follows from the fact that $w$ is supported on $A$.) 
Let $0 \leq a \leq n$ and $b\geq 0$. On the event $\sA_{a,b}:=\{X_a=Y_b, a \leq \tau-1, b \leq \kappa-1\}$, define
\begin{align*}
j(a,b) &:= \min\left\{j \geq 0 : \ler{}{X^a}_j \in \{X_a,\ldots,X_{(\tau-1) \wedge n}\}\right\}
\intertext{and}
i(a,b) &:= \min\left\{i \geq 0 : \ler{}{X^a}_i \in \{Y_b,\ldots,Y_{(\kappa-1)}\}\right\},
\end{align*}
noting that both sets being minimized over are never empty. Define $i(a,b)=j(a,b)=0$ on the complement of $\sA_{a,b}$.
Observe that if $X^{\tau \wedge n}_a=Y^{\kappa}_b \in V_r$ for some $0 \leq a \leq n$ and $b\geq 0$ and $i(a,b) \leq j(a,b)$ then $\ler{}{X^a}_{i(a,b)}$ belongs to both $\ler{}{X^{(\tau-1)\wedge n}}$ and $Y^{\kappa-1}$, so that if we define
\[
I_w := \sum_{a=0}^{\infty}\sum_{b=0}^{\infty} w(a,X_a^{\tau \wedge n},b,Y_b^{\kappa}) \mathbbm{1}(i(a,b) \leq j(a,b)) = \sum_{a=0}^{(\tau-1)\wedge n}\sum_{b=0}^{\kappa-1} w(a,X_a,b,Y_b) \mathbbm{1}(i(a,b) \leq j(a,b))
\]
then 
\begin{equation}
\label{eq:Iwinitial}
\P\left(\ler{}{X^{(\tau-1)\wedge n}} \cap Y^{\kappa-1} \neq \emptyset\right) \geq \P(I_w >0).
\end{equation}
To complete the proof, we will apply a second moment analysis to lower bound the right hand side of \eqref{eq:Iwinitial}.


 Let $v\in V_r$. Given $\sA_{a,b} \cap \{X_a=Y_b=v\}$, the continuations $(X_a,X_{a+1},\ldots,X_{\tau-1})$ and $(Y_b,Y_{b+1},\ldots,Y_{\kappa-1})$ are conditionally independent  and have the same conditional distribution. It follows in particular that the conditional distribution of $\{X_a,\ldots,X_{(\tau-1) \wedge n}\}$ is stochastically dominated by the conditional distribution of $\{Y_b,\ldots,Y_{(\kappa-1)}\}$, so that
\[
\P\left(i(a,b) \leq j(a,b) \mid \sA_{a,b}\cap\{X_a=Y_b=v\}\right) \geq \P\left(j(a,b) \leq i(a,b) \mid \sA_{a,b}\cap\{X_a=Y_b=v\}\right)
\]
and hence that
\[
\P\left(i(a,b) \leq j(a,b) \mid \sA_{a,b}\cap\{X_a=Y_b=v\}\right) \geq \frac{1}{2}
\]
for every $0 \leq a \leq n$, $b\geq 0$, and $v\in V_r$. Since $i(a,b)=j(a,b)=0$ whenever the event $\sA_{a,b}$ does not hold, it follows that
\begin{equation}
\label{eq:ijconditioned}
\P\left(i(a,b) \leq j(a,b) \mid X_a^{\tau\wedge n},Y_b^{\kappa}\right) \geq \frac{1}{2}
\end{equation}
almost surely.
%
%
%
%
Conditioning on $X^{\tau \wedge n}_a$ and $Y^\kappa_b$ and applying \eqref{eq:ijconditioned} gives that
\begin{align*}
\E{I_w} &= \sum_{a=0}^{\infty}\sum_{b=0}^{\infty} \E{w(a,X_a^{\tau\wedge n},b,Y_b^\kappa) \P\left(i(a,b) \leq j(a,b) \mid X_a^{\tau \wedge n},Y_b^{\kappa}\right)}
\\
&\geq
\frac{1}{2}\sum_{a=0}^{\infty}\sum_{b=0}^{\infty} \E{w(a,X_a^{\tau\wedge n},b,Y_b^\kappa) } = \frac{1}{2}\E{S_w}.
\end{align*}
On the other hand, we trivially have that $I_w \leq S_w$ almost surely, and hence by Cauchy-Schwarz that
\begin{equation}
\label{eq:Iwfinal}
\P(I_w >0)\geq \frac{\E{I_w}^2}{\E{I_w^2}} \geq \frac{1}{4} \frac{\E{S_w}^2}{\E{S_w^2}} \geq \frac{1}{256}  \P\left(X^{(\tau-1)\wedge n} \cap Y^{\kappa-1} \neq \emptyset\right),
\end{equation}
where we used \eqref{eq:Swinitial} in the final inequality.
The claimed inequality \eqref{eq:256finitevolume} now follows from \eqref{eq:Iwinitial} and \eqref{eq:Iwfinal}.
\end{proof}


\subsection{A lower tail estimate}
\label{subsec:capacity_lower_tail}

The goal of this section is to bound from above the probability that the capacity of the 4d loop-erased random walk is much smaller than its expectation. Note that, since $\cpc{}{\ler{}{X^n}} \leq \cpc{}{X^n}$ for every $n\geq 0$, we have by \cref{cor:expectedlerwcapZ4}, \cref{lem:known}, and the Paley-Zygmund inequality that there exists a positive constant $c$ such that
\[
\pr{\cc{\ler{}{X^n}}\geq \frac{cn}{\log n
}} \geq \frac{\E{\cc{\ler{}{X^n}}}^2}{4\E{\cc{\ler{}{X^n}}^2}} \geq \frac{\E{\cc{\ler{}{X^n}}}^2}{4\E{\cc{X^n}^2}} \geq c
\]
for every $n\geq 2$. Recall the definition of $\ler{\infty}{X^n}$ from~\eqref{eq:definflerw}. The following proposition substantially improves upon the estimate above, and shows that the capacity of $\ler{\infty}{X^n}$ (and hence that of $\ler{}{X^n}$) is polylogarithmically unlikely to be much smaller than its mean.

\begin{prop}\label{prop:firstconcentration}
	Let $X$ be a random walk on $\Z^4$ started at the origin. There exists a positive constant $c$ such that
	\[
	\pr{\cc{\ler{\infty}{X^n}} \leq  \frac{cn}{\log n}} \lesssim_\eps \frac{1}{(\log n)^{1-\epsilon}}.
	\]
	for every $\eps>0$ and $n\geq 2$.
\end{prop}

The proof will apply estimates due to Lawler concerning the cut times (a.k.a.\ loop-free times) of the random walk. Recall that a time $t\geq 0$ is said to be a \textbf{cut time} of the random walk $X$ if $X[0,t]$ and $X(t,\infty)$ are disjoint. Observe that if $0 \leq s \leq t$ are cut times of $X$ then the loop-erasure of $X$ is equal to the concatenation of the loop-erasures of the portions of $X$ before $s$, between $s$ and $t$, and after $t$; this property is very useful to us as it allows us to decorrelate different parts of the loop-erased random walk.

Although a doubly-infinite random walk in $\Z^4$ does not have any cut times almost surely~\cite{ErdosTaylor}, a \emph{singly}-infinite random walk will nevertheless have a reasonably good supply of cut times.
To make use of this, we will use the following estimate of Lawler \cite[Lemma 7.7.4]{Law91} concerning the prevalence of cut times in the four dimensional simple random walk.

\begin{lemma}\!{\rm{{\cite[Lemma~7.7.4]{Law91}}}}
\label{lem:Lawlercuttimes}
\!Let $X$ be simple random walk on $\Z^4$. Then 
\[
\P(\text{there are no cut times between times $n$ and $m$}) \lesssim \frac{\log \log m}{\log m}
\]
for every $3\leq n \leq m$ such that $|n-m| \geq m/(\log m)^6$.
\end{lemma}

We now apply \cref{lem:secondmoment,lem:known,cor:expectedlerwcapZ4,lem:Lawlercuttimes} to prove \cref{prop:firstconcentration}.

\begin{proof}[Proof of \cref{prop:firstconcentration}]
Following the idea of Lawler~\cite[Theorem~7.7.5]{Law91}, we will compare the capacity of $\ler{\infty}{X^n}$ to the capacity of the union of the loop-erased paths on suitably defined disjoint subintervals of $[0,n]$.
 Since $\cpc{}{\ler{\infty}{X^n}}$ is increasing in $n$, it suffices to prove the claim for $n$ of the form $n=r^4 2^r$ for some integer $r\geq 2$. Fix such an $n$ and $r$.
Set $\ell = r^2 2^r \asymp n/(\log n)^2$ and divide the interval $[0,n]$ into $m= r^2 = n/\ell \asymp (\log n)^2$ subintervals $(I_j)_{j=1}^m$ each of length $\ell$, so that $I_j=[(j-1)\ell,j\ell]$ for each $j=1,\ldots,m$.
We clearly have that
\begin{align}
	&\pr{\cc{\ler{\infty}{X^n}}\leq \frac{n}{C \log n}} \leq  \pr{\cc{\cup_{j=1}^{m}\ler{}{X[I_j]}}\leq \frac{2n}{C \log n}} 
	\nonumber
	\\ 
	&\hspace{4.5cm}+ \pr{\cc{\cup_{j=1}^{m}\ler{}{X[I_j]}}-\cc{\ler{\infty}{X^n}} > \frac{n}{C\log n}}
	\label{eq:firstconcentration_union}
\end{align}
for each $C\geq 1$. We will show that if $C$ is taken to be appropriately large then both terms on the right hand side admit upper bounds of the desired order. The first term will be bounded using \cref{cor:expectedlerwcapZ4} together with ideas taken from \cite{ASS}, while the second will be bounded by a similar method to the proof of \cite[Theorem~7.7.5]{Law91}.

\medskip

We begin by bounding the second term on the right hand side of \eqref{eq:firstconcentration_union}. 
 For each  $1 \leq j \leq m$ we define the \emph{left and right buffer intervals}
\[
L_j = \left[(j-1)\ell, (j-1)\ell + 2^r   \right]  \ \text{ and }\ R_j=\left[ j\ell - 2^r, j \ell    \right],
\]
so that $L_j,R_j \subseteq I_j$ for every $0\leq j \leq m$, noting that $2^r \asymp n/(\log n)^4$.
For each $1\leq j \leq m$ we also define
\[
A_j = \1(\exists \ \text{cut times times in both } L_{j} \text{ and } R_{j} ).
\]
The goal is to compare $\cc{\cup_{j=1}^{m}\ler{}{X[I_j]}}$ to $\cc{\ler{\infty}{X^n}}$. 
We first claim that  
\begin{align}\label{eq:inclusionnotobvious}
	\bigcup_{j=1}^{m} \ler{}{X[I_j]} \subseteq \ler{\infty}{X[0,n]}\cup \left(\bigcup_{j=1}^{m} X[L_j]\cup X[R_j]\right) \cup \left( \cup_{j: A_j=0} X[I_j]  \right).
\end{align}
Indeed, suppose that $X_k\in \ler{}{X[I_j]}$ with $k\in I_j$ and suppose that $A_j=1$, i.e.\ that there exist cut times $s$ and $t$ in $L_j$ and $R_j$ respectively. Then this means that $X[0,s]\cap X(s,\infty)=\emptyset$ and $X[0,t]\cap X(t,\infty)=\emptyset$.
It follows that if $k\in [s,t]$ is such that $X_k \in \LE(X[s,t])$ then $X_k\in \ler{\infty}{X^n}$, while otherwise we have that $X_k\in X[L_j]\cup X[R_j]$ and hence the claim is proved. 
Taking capacities of both sides of~\eqref{eq:inclusionnotobvious} and using subadditivity of capacity we obtain that
\begin{align}\label{eq:leinfty}
	\cc{\cup_{j=1}^{m}\ler{}{X[I_j]}}-\cc{\ler{\infty}{X^n}} \leq 16m2^r + \sum_{j=1}^{m}(1-A_j) \cc{X[I_j]}.
\end{align}
The first term satisfies $16m2^r \lesssim n/(\log n)^2$. For the second term, \cref{lem:Lawlercuttimes} gives that
\[
\pr{A_j=0}\lesssim \frac{\log \log n}{\log n},
\]
and applying H\"older's inequality and \cref{lem:known} we obtain that
\begin{align*}
	\E{\sum_{j=1}^{m}(1-A_j) \cc{X[I_j]}} 
&\leq \sum_{j=1}^m \P(A_j=0)^{1-\frac{1}{2k}} \E{\cpc{}{X[I_j]}^{2k}}^{\frac{1}{2k}}
\\
	&\lesssim_k m \left(\frac{\log \log n}{\log n}\right)^{1-\frac{1}{2k}}\cdot  \frac{\ell}{\log \ell}  \lesssim_k \frac{n}{(\log n)^{2-\frac{3}{k}}}
\end{align*}
for every $k \geq 1$, where we took expectations over the bound
\begin{multline*}\mathrm{Cap}(X[I_j])^{2k} \leq \left(\mathbb{E}\left[\mathrm{Cap}(X[I_j])\right]+\bigl|\mathrm{Cap}(X[I_j])-\mathbb{E} \left[\mathrm{Cap}(X[I_j])\right]\bigr|\right)^{2k} \\\lesssim_k \mathbb{E}\left[\mathrm{Cap}(X[I_j])\right]^{2k} + \bigl|\mathrm{Cap}(X[I_j])-\mathbb{E} \left[\mathrm{Cap}(X[I_j])\right]\bigr|^{2k} \end{multline*} in the first inequality and used that $m \ell = n$ in the final inequality.
Applying Markov's inequality, we get that
\begin{align*}
	\pr{\sum_{j=1}^{m}(1-A_j) \cc{X[I_j]} \geq \frac{n}{\log n}} \lesssim_\eps \frac{1}{(\log n)^{1-\epsilon}}
\end{align*}
for every $\eps > 0$.
Using this together with~\eqref{eq:leinfty} we get 
\begin{align}\label{eq:difference}
	\pr{\cc{\cup_{j=1}^{m}\ler{}{X[I_j]}}-\cc{\ler{\infty}{X[0,n]}} > \frac{n}{C\log n}} \lesssim_\eps \frac{C}{(\log n)^{1-\epsilon}}
\end{align}
for every $\eps>0$ and $C \geq 1$. This completes the analysis of the second term of \eqref{eq:firstconcentration_union}.

\medskip

We now turn to the first term on the right hand side of \eqref{eq:firstconcentration_union}. We adapt the methods of the proof of \cite[Theorem 1.1]{ASS}. It suffices to prove that there exists $C \geq 1$ such that
\begin{align}\label{eq:goalinproof}
\pr{\cc{\cup_{j=1}^{m}\ler{}{X[I_j]}}\leq \frac{2n}{C \log n}}\lesssim \frac{\log \log n}{(\log n)^{2}},
\end{align}
the right hand side being smaller than required by the statement of the lemma.
We will suppose for ease of notation that there exists an integer $L \asymp \log \log n$ such that $2^L=m=n/\ell$, the general case being similar. Applying Proposition~\ref{pro:deccap} we get that
\begin{align*}
\cc{\cup_{j=1}^{m}\ler{}{X[I_j]}}  \geq  \cc{A_1}  + \cc{A_2} -\chi(1,1),
\end{align*}
where   $A_1=\cup_{j=1}^{m/2}\ler{}{X[I_j]}$, $A_2=\cup_{j=\tfrac{m}{2}+1}^{m}\ler{}{X[I_j]}$ and $\chi(1,1)= \chi(A_1,A_2) + \chi(A_2,A_1)$.  Continuing in the same way, by subdividing every union appearing above into two sets and applying Proposition~\ref{pro:deccap} repeatedly in a dyadic fashion, we obtain that
\begin{align}\label{eq:decompole}
	\cc{\cup_{j=1}^{m}\ler{}{X[I_j]}} \geq  \sum_{j=1}^{m} \cc{\ler{}{X[I_j]}} - \sum_{i=1}^{L}\sum_{j=1}^{2^{i-1}} \chi(i,j ),
\end{align}
where $\chi(i,j) = \chi(A_{i,j}^{(1)},A_{i,j}^{(2)})+ \chi(A_{i,j}^{(2)},A_{i,j}^{(1)})$ with
\begin{align*}
A_{i,j}^{(1)}=\bigcup_{k=(2j-2)\tfrac{m}{2^i}+1}^{(2j-1)\tfrac{m}{2^i}}  \ler{}{X[I_k]} \quad \text{ and } \quad  A_{i,j}^{(2)}=\bigcup_{k=(2j-1)\tfrac{m}{2^i}+1}^{2j\tfrac{m}{2^i}}  \ler{}{X[I_k]}.
\end{align*}
 Using that $\ler{}{X[I_k]}\subseteq X[I_k]$ together with Lemmas~\ref{lem:monotonicityoftilchi} and \ref{lem:secondmoment} and the inequality $\chi \leq \tilde \chi$, we obtain that
\begin{align}\label{eq:boundsmoments}
	\E{\chi(i,j)^2} \lesssim n^2 \cdot \frac{(\log \log n)^2}{2^{2i}(\log n)^4} \quad \text{ and } \quad \E{\chi(i,j)} \lesssim n \cdot \frac{\log \log n}{2^{i}(\log n)^2} 
\end{align}
for every $1 \leq i \leq L$. 
For a fixed $i$ the variables $(\chi(i,j))_j$ are independent and identically distributed, and we deduce by Cauchy--Schwarz that
\begin{align}\label{eq:boundonvariance}
	\vr{\sum_{i=1}^{L}\sum_{j=1}^{2^{i-1}} \chi(i,j ) } \leq L\cdot \sum_{i=1}^{L}2^{i-1}\vr{\chi(i,1)} \lesssim n^2 \cdot \frac{(\log \log n)^3}{(\log n)^4},
\end{align}
where we also used that $L\asymp \log \log n$.
Applying~\eqref{eq:decompole} we get that
\begin{multline}\label{eq:proofofgoal}
	\pr{\cc{\cup_{j=1}^{m}\ler{}{X[I_j]}}\leq \frac{2n}{C \log n}} \leq\pr{\sum_{j=1}^{m}\cc{\ler{}{X[I_j]}}\leq \frac{4n}{C \log n}} 
	\\+ \pr{\sum_{i=1}^{L}\sum_{j=1}^{2^{i-1}}\chi(i,j) \geq \frac{2n}{C\log n}} .
\end{multline}
Using the trivial bound $\vr{\cpc{}{\ler{}{X[I_k]}}}\leq |I_k|^2$, \cref{cor:expectedlerwcapZ4}, and the fact that the capacities $\cpc{}{\ler{}{X[I_j]}}$ and $\cpc{}{\ler{}{X[I_k]}}$ are independent when $j\neq k$, Chebyshev's inequality implies that there exists a constant $C\geq 1$ such that
\begin{align*}
	\pr{\sum_{j=1}^{m}\cc{\ler{}{X[I_j]}}\leq \frac{4n}{C \log n}} \lesssim \frac{1}{(\log n)^3}.
\end{align*}
Similarly, using~\eqref{eq:boundsmoments} and~\eqref{eq:boundonvariance} and Chebyshev's inequality again gives that
\begin{align*}
	\pr{\sum_{i=1}^{L}\sum_{j=1}^{2^{i-1}}\chi(i,j) \geq \frac{2n}{C\log n}} \lesssim \frac{(\log \log n)^3}{(\log n)^2}.
\end{align*}
Plugging these two bounds into~\eqref{eq:proofofgoal} proves~\eqref{eq:goalinproof} and this finishes the proof of the lemma. 
\end{proof}


\begin{remark}
We expect that the upper bound of \cref{prop:firstconcentration} is essentially optimal in the sense that
\begin{equation}
\label{eq:heuristic1}
\pr{\cc{\ler{\infty}{X^n}}\leq \frac{n}{\lambda \log n } }\gtrsim_{\lambda,\eps} \frac{1}{(\log n)^{1+\eps}}.
\end{equation}
for every $\lambda \geq 1,\eps>0$ and $n \geq 3$. (It may not even be necessary to include the additional $\eps$ in the power of log, but this is harder to be sure of without writing a detailed proof.) Indeed, writing $\ler{n}{X^{\lfloor \theta n\rfloor}}$ for the part of $\ler{}{X^n}$ that is contributed by the first $\lfloor \theta n\rfloor$ steps of the walk, we expect that 
\begin{equation}
\label{eq:heuristic2}
\P\left(X[n,\infty] \cap \LE_n\left(X^{\lfloor\theta n\rfloor}\right) \neq \emptyset\right) \gtrsim_{\theta,\eps} \frac{1}{(\log n)^{1+\eps}}
\end{equation}
for every $\theta \in [0,1/2]$, $\eps>0$ and $n\geq 2$, as is consistent with the fact that if $X$ and $Y$ are two independent random walks started at  distance roughly $\sqrt{n}$ apart then $Y$ intersects $\LE(X^n)$ with probability of order $1/\log n$ \cite[Theorem 4.3.3]{Law91}.
  (Making this rigorous would require one to control the difference between $\LE_n(X^{\lfloor \theta n\rfloor})$ and $\LE(X^{\lfloor\theta n\rfloor})$ and to argue that the path $\LE(X^{\lfloor \theta n \rfloor})$ is approximately independent of $X_n$ in some sense.) Again, it may be that the additional $\eps$ in the exponent is not necessary, but we have included it to be cautious. If \eqref{eq:heuristic2} were to be proven, one would easily deduce \eqref{eq:heuristic1} from it together with \cref{lem:known}, since if the event on the left hand side of \eqref{eq:heuristic2} holds then $\LE_\infty(X^n)$ is contained in $X^{\lfloor \theta n \rfloor}$, and the latter set is unlikely to have capacity much larger than its mean by \cref{lem:known}. 

 We do not investigate this matter further since we do not require a lower bound of the form \eqref{eq:heuristic1} for the proofs of our main theorems. We shall however prove in \cref{subsec:polygamous_deviants} that stronger concentration holds for $\cpc{}{\ler{}{X^n}}$ if we condition on the event that $\|X_n\|$ is not too small.

\end{remark}

\subsection{An upper tail estimate}
\label{subsec:capacity_upper_tail}

The last result of this section applies \cref{lem:known} together with \cref{lem:amounterased} to estimate the upper tail of the capacity of the first $n$ steps of the loop-erased random walk. (Recall the distinction between $\ler{}{X^n}$ and $\ler{}{X}^n$ as defined in Section~\ref{subsec:LERWbackground} and 
note that an analogous bound for the capacity of $\ler{}{X^n}$ follows trivially from \cref{lem:known}.) 

\begin{lemma}\label{cor:lercap}
Let $X$ be a random walk in $\Z^4$ starting from $0$. There exists a positive constant $C$ such that
\begin{align*}
	\pr{\cc{\ler{}{X}^n}\geq \frac{Cn}{(\log n)^{2/3}}} \lesssim \frac{\log\log n}{(\log n)^{2/3}}
\end{align*}
for every $n \geq 3$.
\end{lemma}

\begin{proof}[Proof]
Recall that $\ell_n=\ell_n(X)$ denotes the $n$th time contributing to the loop-erasure of $X$.
  Using~\eqref{eq:equivtnrhom} together with Theorem~\ref{lem:amounterased} give that for a suitable constant $C$ we have 
\begin{align}\label{eq:conctn}
\pr{\ell_n \geq C n (\log n)^{1/3}}  \lesssim \frac{\log \log n}{(\log n)^{2/3}}.
\end{align}
On the other hand, Lemma~\ref{lem:known} and Chebyshev's inequality yield that 
	\begin{align*}
		\pr{\cc{X^n} \geq 2 \E{\cc{X^n}}} \leq \frac{\vr{\cc{X^n}}}{\left(\E{\cc{X^n}}\right)^2} \asymp \frac{1}{(\log n)^2}
	\end{align*}
	for every $n\geq 2$.
Putting these estimates together, we deduce
 that there exists a constant $C'$ such that
\begin{align*}
	&\pr{\cc{\ler{}{X}^n}\geq \frac{C'n}{(\log n)^{2/3}}} \\ &\hspace{1.5cm}\leq \pr{\cc{\ler{}{X}^n}\geq \frac{C'n}{(\log n)^{2/3}}, \ell_n\leq Cn(\log n)^{1/3}} + \pr{\ell_n\geq Cn(\log n)^{1/3}} \\&\hspace{1.5cm}\leq 
	\pr{\cc{X[0,Cn(\log n)^{1/3}]} \geq \frac{C'n}{(\log n)^{2/3}}} + \pr{\ell_n\geq Cn(\log n)^{1/3}} \lesssim 
	\frac{\log \log n}{(\log n)^{2/3}},
\end{align*}
	where we also applied \cref{lem:known} to bound the first term on the last line. \end{proof}

\section{Polylogarithmic deviations for the length and capacity}
\label{subsec:polygamous_deviants}

Unfortunately, for several applications later in the paper, the lower tail estimates of \cref{lem:amounterased,prop:firstconcentration} are not quite strong enough for our arguments to work. 
In this section, we address this shortcoming by studying the probability that the length or the capacity of the loop-erased random walk $\LE(X^t)$ are smaller than their typical value by a polylogarithmic factor. The basic idea, which is inspired by the decorrelation techniques of Masson \cite[Section 4.1]{MR2506124}, is that if $\|X_t\|$ is not too small then the lengths and capacities of the segments of $\LE(X^t)$ near $0$ and near $X_t$ are approximately independent. Thus, we should be able to square
the concentration estimates of \cref{prop:firstconcentration,lem:amounterased} in this case, giving a polylogarithmic improvement to the relevant bounds. To implement these arguments rigorously, it will be convenient to use a geometric random time $T$ of mean $t$ rather than a fixed time $t$. We write $a \wedge b = \min\{a,b\}$.

%
%


\begin{prop}\label{lem:capgeom}
 Let $t\geq 1$, let $T$ be a geometric random variable of mean $t$, and let~$X$ be an independent simple random walk on $\Z^4$. Then there exists a constant $C$ such that 
		\begin{align}
	\prstart{\cpp{\LE\Bigl(X^T\Bigr)}\leq \frac{t}{\lambda \log t} \text{ and } \norm{X_{T}}\geq \sqrt{C \lambda^{-1} t \log \log t}}{0} &\lesssim_{\eps} \frac{1}{(\log t)^{2-\eps}}
\qquad	\text{and}\\
	\prstart{\Bigl|\LE\Bigl(X^T\Bigr)\Bigr|\leq \frac{t}{\lambda (\log t)^{1/3}} \text{ and } \norm{X_{T}}\geq \sqrt{C \lambda^{-1} t \log \log t}}{0} &\lesssim_{\eps} \frac{1}{(\log t)^{4/3-\eps}}.
	\end{align}
	for every $t \geq 3$ and $\lambda \geq \log \log t$.
\end{prop}


\begin{remark}
Proving upper tail estimates for the capacity of $\ler{}{X^t}$ is a much easier task than proving lower tail estimates since it suffices to prove an analogous estimate on the capacity of the \emph{simple} random walk. Indeed, adapting~\cite[Proposition~4.1]{ASS16+} to the discrete setting would immediately yield an upper tail estimate of the form $\exp(-c t^\kappa)$ for some constants $c,\kappa>0$.
\end{remark}

This proposition has the following corollary. Although this corollary is only a very modest improvement of \cref{prop:firstconcentration,lem:amounterased}, and indeed is only an improvement at all when $\alpha>1/2$ or $\alpha > 1/3$ as appropriate, it will nevertheless suffice for all our later applications, \cref{pro:o(n)} being chief among these.

\begin{corollary}
\label{cor:capgeom}
 Let~$X$ be a simple random walk on $\Z^4$. Then for every $0< \alpha \leq 1$, $C<\infty$, and $\eps>0$ we have that
		\[
\frac{1}{n}\sum_{m=2}^{n}	\prstart{\cpp{\LE\bigl(X^m\bigr)}\leq \frac{Cm}{(\log m)^{1+\alpha}} }{0} \lesssim_{\alpha,C,\eps} \frac{1}{(\log n)^{2\alpha-\eps}} 
	\]
	for every $n \geq 2$. Similarly, for every $0< \alpha \leq 2/3$, $C<\infty$, and $\eps>0$ we have that
			\[
	\frac{1}{n}\sum_{m=2}^{n}\prstart{\bigl|\LE\bigl(X^m\bigr)\bigr|\leq \frac{Cm}{(\log m)^{1/3+\alpha}}}{0} \lesssim_{\alpha,C,\eps} \frac{1}{(\log n)^{2\alpha-\eps}} 
	\]
	for every $n \geq 2$.
\end{corollary}




\begin{proof}[Proof of \cref{cor:capgeom} given \cref{lem:capgeom}]
We will prove the claim concerning the capacity, the proof of the claim concerning the length being similar. 
It suffices to prove the claim in the case $C=1$, the general case following by decreasing $\alpha$ by some arbitrarily small amount. 
Fix $0<\alpha \leq 1$ and $\eps>0$. Proposition~\ref{lem:capgeom} yields that
\begin{equation*}
	\sum_{m=0}^{\infty} \left( \frac{t}{t+1} \right)^m \frac{1}{t+1}\, \prstart{\cpp{\lr{X^m}}\leq \frac{t}{(\log t)^{1+\alpha}},\ \norm{X_m}\geq \sqrt{\frac{t \log \log t}{(\log t)^\alpha}} }{0} \\
	\lesssim_{\alpha,\eps} \frac{1}{(\log t)^{2-\eps}}
\end{equation*}
for every $t \geq 3$. On the other hand, the local limit theorem implies that $\P(\|X_m\| \leq r) \lesssim m^{-2} r^4$ for every $m,r \geq 1$. We deduce by a union bound that
\begin{equation*}
	\sum_{m=t}^{2t}\left( \frac{t}{t+1} \right)^m \frac{1}{t+1} \, 
	\prstart{
	\cpp{\lr{X^m}}\leq \frac{t}{(\log t)^{1+\alpha}}
	}{0}
	\lesssim_{\alpha,\eps} \frac{1}{(\log t)^{2 - \eps}} + \frac{(\log \log t)^2}{(\log t)^{2\alpha}}
	 \lesssim_{\alpha,\eps} \frac{1}{(\log t)^{2\alpha - \eps}}
\end{equation*}
for every $t \geq 3$.
Using that $t/(\log t)^{1+\alpha}$ is increasing for large $t$ and that the term $(t/(t+1))^m$ is of order $1$ for $m\in [t,2t]$, we deduce that
\begin{equation*}
	\sum_{m=t}^{2t}
	\prstart{
	\cpp{\lr{X^m}}\leq \frac{m}{(\log m)^{1+\alpha}}
	}{0}
	\lesssim_{\alpha,\eps} \frac{t}{(\log t)^{2\alpha-\eps}}
\end{equation*}
for every $t\geq 2$.
%
%
 Summing this estimate over all $t$ of the form $2^k$ between $1$ and $n$ yields that
\begin{equation*}
\sum_{m=2}^{n}  \prstart{\cpp{\ler{}{X^m}}\leq \frac{m}{(\log m)^{1+\alpha}}}{0}  \,\lesssim_{\alpha,\eps}\, \sum_{k=1}^{\lfloor \log_2(n)\rfloor} \frac{2^k}{k^{2\alpha-\eps}} \,\lesssim_{\alpha,\eps}\, \frac{n}{(\log n)^{2\alpha-\eps}} 
\end{equation*}
as required. 
\end{proof}

We begin working towards the proof of \cref{lem:capgeom} by proving the following lemma, which is a consequence of the domain Markov property of the loop-erased random walk and describes how the initial and final segments of the loop-erased random walk are correlated.
Recall that we write $B(x,r)$ for the $\|\cdot\|_2$-ball of radius $r$ around $x$ in $\Z^d$, and write $\partial B(x,r)$ for the internal vertex boundary of $B(x,r)$, i.e., the set of $y\in B(x,r)$ that have a neighbour outside of $B(x,r)$. For each path $\gamma=[\gamma_0,\ldots, \gamma_k]$ we write $\gamma^\leftarrow=[\gamma_k,\ldots, \gamma_0]$ for the time reversal of $\gamma$. (Here it is convenient to follow different indexing conventions for time-reversals than in \cref{sec:AldousBroder}.)

\begin{lemma}
\label{lem:end_segments_decorrelation}
Let $d\geq 1$, let $T$ be a geometric random variable with mean $t>0 $, and let $X$ be an independent random walk on $\Z^d$. Fix $r\geq 1$, let $\kappa=\kappa_r$ be the first time $\LE(X^{T})$ visits $\partial B(X_0,r)$, let $\sigma=\sigma_r$ be the last time $\LE(X^{T})$ visits $\partial B(X_{T},r)$, and let $\rho=|\LE(X^T)|$. (We define $\kappa$ and $\sigma$ to be equal to $\rho$ and $0$ respectively on the events that the relevant sets are not hit.) Then for each $x \in \Z^d$ with $\|x\|_2> 2r$ we have that
\begin{multline*}
\P_0\Bigl(\LE(X^T)[0,\kappa]=\omega \text{ and } \LE(X^T)[\sigma,\rho] = \eta 
\;\Big |\; X_T = x \Bigr)
\\=
\P_0\Bigl(\LE(X^T)[0,\kappa]=\omega\Bigr) \P_x\Bigl(\LE(X^T)[0,\kappa] = \eta^\leftarrow
\;\Big |\; X_T = 0 \Bigr) \frac{\prcond{X_{T}=\eta_0, \tau_{\eta} = T}{\tau_{\omega}^+>T}{\omega_k}}{\prstart{X_{T}=\eta_0, \tau_{\eta} =T}{0}}
\end{multline*}
for every path $\omega=[\omega_0,\ldots,\omega_k]$ from $0$ to $\partial B(0,r)$ and every path $\eta=[\eta_0,\ldots,\eta_m]$ from $\partial B(x,r)$ to $x$, where $\tau_\eta$ denotes the first time the walk visits $\{\eta_0,\ldots,\eta_m\}$ and $\tau_\omega^+$ denotes the first positive time the walk visits $\{\omega_0,\ldots,\omega_k\}$.
\end{lemma}

We now recall the \emph{domain Markov property} of loop erased random walk, which was first used implicitly by
  Lawler \cite[Proposition 7.4.1]{Law91}. (Note that the domain Markov property holds more generally for the loop erasure of any Markov chain \cite[Lemma 2.4]{MR2506124} and hence for the  walk killed at a geometric time as in our setup below.)

\begin{theorem}[Domain Markov property]\label{thm:domain}
	Let $d\geq 1$, let $T$ be a geometric random variable with mean $t>0$ and let $X$ be an independent  simple random walk on $\Z^d$. Then for every two self-avoiding paths $\omega=[\omega_0,\ldots, \omega_k]$ and $\eta=[\eta_0,\ldots,\eta_m]$ with $\omega_k=\eta_0$ and $\omega\cap \eta=\{\omega_k\}$ we have 
	\[
	\prcond{\LE(X^T)[0,m+k]=\omega\oplus \eta}{\LE(X^T)[0,k]=\omega}{0} = \prcond{\LE(X^T)[0,m]=\eta}{\tau_{\omega}^+>T}{\omega_k},
	\]
where $\omega\oplus \eta$ denotes the concatenation of $\omega$ and $\eta$. 
\end{theorem}

We next state the reversibility property of LERW. This is a classical result but we include the proof (which follows using Lawler's bijection~\cite[Lemma 7.2.1]{Law91}) since our setup is slightly different to that appearing in the literature.  

\begin{theorem}[Reversibility of LERW]\label{thm:reversibility} 
Let $X$ be a simple random walk on $\Z^d$.
 	Let $n\geq 0$, let $\eta=[\eta_0,\ldots, \eta_m]$ be a self-avoiding path in $\Z^d$, and let $A\subseteq \Z^d$. Then 
\[
\prstart{\ler{}{X^n}=\eta,\, \tau_A^+>n}{\eta_0} = 
 \prstart{\ler{}{X^n}=\eta^\leftarrow,\, \tau_A\geq n}{\eta_m}.
\]
\end{theorem}

If $T$ is a geometric random time of mean $t\geq 0$ independent of $X$, it follows by summing over the possible values of $T$ that
\begin{equation}
\label{eq:reversibility_geometric}
\prstart{\ler{}{X^T}=\eta,\, \tau_A^+>T}{\eta_0} = 
 \prstart{\ler{}{X^T}=\eta^\leftarrow,\, \tau_A\geq T}{\eta_m}
\end{equation}
for every self-avoiding path $\eta=[\eta_0,\ldots, \eta_m]$ and $A\subseteq \Z^d$.

\begin{proof}[Proof of \cref{thm:reversibility}]
Fix a self-avoiding path $\eta =[\eta_0,\ldots,\eta_m]$ in $\Z^d$.
For each finite path $\lambda$ in $\Z^d$ we define the reverse loop-erasure $\LE^R(\lambda)$ of $\lambda$ to be
\[
\LE^R(\lambda)=(\LE(\lambda^\leftarrow))^\leftarrow.
\]
For each $n\geq 0$, let $\Gamma_n$ be the set of paths in $\Z^d$ of length $n$ that start from $\eta_0$. Lawler in~\cite[Lemma~7.2.1]{Law91} shows that for each $n$ there exists a bijection $f^n:\Omega_n\to \Omega_n$ such that
\[
\LE(\lambda) = \LE^R(f^n\lambda)
\]
  and that $\lambda$ and $f^n\lambda$ cross the same edges with the same multiplicities for each $\lambda\in \Omega_n$. In particular, the bijection $f^n$ is measure-preserving in the sense that
  $\prstart{X^n=\lambda}{\eta_0} = \prstart{X^n=f^n\lambda}{\eta_0}$ for every $\lambda \in \Omega_n$.
Since we also have that $X^n$ visits $A$ if and only if $f^n X^n$ does, it follows that
\begin{multline*}
	\prstart{\LE(X^n)=\eta, \tau_A^+>n}{\eta_0} = \prstart{\LE^R(f^nX^n)=\eta, \tau_{A}^+>n}{\eta_0}
  = \prstart{\LE^R(X^n)=\eta, \tau_A^+>n}{\eta_0}\\
   = \prstart{\LE((X^n)^\leftarrow)=\eta^\leftarrow, \tau_A^+>n}{\eta_0}
  =\prstart{\LE(X^n)=\eta^\leftarrow ,\tau_A\geq n, X_n=\eta_0}{\eta_m}
\end{multline*}
as required.
\end{proof}

\begin{proof}[Proof of \cref{lem:end_segments_decorrelation}]
Fix $d\geq 1$, $t>0$, $r\geq 1$, and $x\in \Z^d$ with $\|x\|_2 > 2r$.
 The domain Markov property of the LERW~(\cref{thm:domain}) implies that
\begin{equation}
	\prcond{\LE(X^T)[\sigma, \rho] = \eta}{\LE(X^T)[0,\kappa] = \omega}{0} =
	\prcond{\LE(X^T)[\sigma, \rho] = \eta}{\tau^+_\omega > T}{\omega_k}.
\end{equation}
(Note that if $\LE(X^T)[\sigma, \rho] = \eta$ then $X_T=x$.)
Using the reversibility property of LERW as in \eqref{eq:reversibility_geometric} together with the observation that the first hitting time of a set for a forward path becomes the last hitting time of the same set for the reversed path we deduce that 
\begin{multline}
	\prcond{\LE(X^T)[\sigma, \rho] = \eta}{\LE(X^T)[0,\kappa] = \omega}{0} \\= \frac{1}{\P_{\omega_k}(\tau^+_\omega>T)}
	\prstart{X_T= \omega_k, \tau_\omega = T, \text{ and } \LE(X^T)[0, \kappa] = \eta^\leftarrow}{x}.
\end{multline}
Applying the domain Markov property (\cref{thm:domain}) a second time, we deduce that
\begin{multline}
	\prcond{\LE(X^T)[\sigma, \rho] = \eta}{\LE(X^T)[0,\kappa] = \omega}{0} \\= \frac{1}{\P_{\omega_k}(\tau^+_\omega>T)}
	\prstart{\LE(X^T)[0, \kappa] = \eta^\leftarrow}{x} \prcond{X_T= \omega_k, \tau_\omega = T }{\tau_\eta^+ > T}{\eta_0}
\end{multline}
and hence that
\begin{multline}
\P_0\Bigl(\LE(X^T)[0,\kappa]=\omega \text{ and } \LE(X^T)[\sigma,\rho] = \eta 
\Bigr)
\\=
\P_0\Bigl(\LE(X^T)[0,\kappa]=\omega\Bigr) \P_x\Bigl(\LE(X^T)[0,\kappa] = \eta^\leftarrow 
 \Bigr) 
\frac{\prstart{X_T= \omega_k, \tau_\omega = T, \tau_\eta^+>T }{\eta_0}}
{\prstart{\tau_\omega^+>T}{\omega_k}\prstart{\tau_\eta^+>T}{\eta_0}}.
\label{eq:symmetric_decorrelation}
\end{multline}
This equation can be thought of as a more symmetric form of the claimed equality. To deduce the claim from it, we use the domain Markov property again to write
\[
\prstart{\LE(X^T)[0,\kappa]=\eta^\leftarrow, X_T = 0}{x} = \prstart{\LE(X^T)[0,\kappa]=\eta^\leftarrow}{x} \prcond{X_T = 0}{\tau^+_\eta>T}{\eta_0}.
\]
Substituting this equality into \eqref{eq:symmetric_decorrelation}, rearranging, and using that $\P_0(X_T=x)=\P_x(X_T=0)$ yields that
\begin{multline}
\prcond{\LE(X^T)[0,\kappa]=\omega \text{ and } \LE(X^T)[\sigma,\rho] = \eta }{X_T=x}{0}= \P_0\Bigl(\LE(X^T)[0,\kappa]=\omega\Bigr)
\\
 \cdot \prcond{\LE(X^T)[0,\kappa] = \eta^\leftarrow 
 }{X_T = 0}{x} 
\frac{\prstart{X_T= \omega_k, \tau_\omega = T, \tau_\eta^+>T }{\eta_0}}
{\prstart{\tau_\omega^+>T}{\omega_k}\prstart{X_T=0, \tau^+_\eta>T}{\eta_0}}.
\label{eq:unsymmetrized}
\end{multline}
Applying a further time-reversal (for the simple random walk rather than the loop erasure) to both terms involving walks started at $\eta_0$, we deduce that
\begin{multline*}
\frac{\prstart{X_T= \omega_k, \tau_\omega = T, \tau_\eta^+>T }{\eta_0}}
{\prstart{\tau_\omega^+>T}{\omega_k}\prstart{X_T=0, \tau^+_\eta>T}{\eta_0}}= \frac{\prstart{X_T= \eta_0, \tau_\omega^+ > T, \tau_\eta = T }{\omega_k}}
{\prstart{\tau_\omega^+>T}{\omega_k}\prstart{X_T=\eta_0, \tau_\eta = T}{0}}
\\=
\frac{\prcond{X_T= \eta_0, \tau_\eta = T }{\tau_\omega^+ > T}{\omega_k}}
{\prstart{X_T=\eta_0, \tau_\eta = T}{0}},
\end{multline*}
and substituting this equality into \eqref{eq:unsymmetrized} yields the claim.
\end{proof}

In order to apply this lemma, we will need to bound the ratio of probabilities that appears on the right hand side in cases of interest.
We begin by proving the following basic decorrelation estimate for random walk killed at a geometric random time.

\begin{lemma}\label{lem:forget}
	Let $d\geq 3$, let $T$ be a geometric random variable of mean $t \geq 1$, and let $X$ be an independent random walk on $\Z^d$. Let $r\geq 1$, and let $x,y \in \Z^d$ and $A,D \subseteq \Z^d$ be such that
	$x \in \partial B(0,r)$, $x \in A \subseteq B(0,r)$, $\|y\|_2 \geq 10 r$, and $d(0,D) \geq 10r$.
	 Then we have 
	\[
	\prcond{X_{T}=y, \tau_D>T}{\tau_A^+>T}{x} \lesssim_d \left(\frac{t}{t+1} \right)^{-4r^2} \prstart{X_{T}=y, \tau_D>T}{0}.
	\]
\end{lemma}

We will only apply this estimate in the regime $r^2=O(t)$, where $\left(t/(t+1) \right)^{-4r^2}=O(1)$.

\medskip


The proof of this lemma will rely upon the \emph{parabolic Harnack inequality}  \cite[Theorem 3.3.5]{KumagaiBook}, which was originally formulated in the graphical context by Delmotte \cite{MR1666463}. Recall that a space-time function $u:\Z^d \times \Z \to \R$ is said to be \textbf{parabolic} on a space-time region $A \subseteq \Z^d \times \Z$ if 
\[
u(x,n+1) = \frac{1}{2d}\sum_{y \sim x}u(y,n)
\]
for every $(x,n) \in A$. Parabolicity is a space-time analogue of harmonicity;  every harmonic function can be thought of as a parabolic function that does not depend on its time coordinate. The parabolic Harnack inequality for $\Z^d$ states that there exists a constant $C=C(d)$ such that if $R \geq 1$ and $u:\Z^d \times \Z \to \R$ is parabolic on $B(0,2R) \times [0,4 R^2]$, then
\begin{equation}
\label{eq:PHI}
\max_{\substack{x\in B(0,R)\\ R^2 \leq n \leq 2R^2}} u(x,n) \leq 
C \min_{\substack{x\in B(0,R)\\ 3R^2 \leq n \leq 4R^2}} (u(x,n)+u(x,n+1)).
\end{equation}
The parabolic Harnack inequality (PHI) can be thought of as a space-time analogue of the elliptic Harnack inequality (EHI), with which the reader may be more familiar and which is used to prove decorrelation inequalities for LERW in the work of Masson \cite{MR2506124}. 

Before using the parabolic Harnack inequality to prove \cref{lem:forget}, let us first give a simple and illustrative application of the PHI that we shall use several times below. Let $T$ be a geometric random variable of mean $t$ and let $X$ be an independent random walk on $\Z^d$. Let $z\in \R^d$ be such that $z\notin B(0,2r)$. Then the function $u: \Z^d \times \Z \to \R$ defined by
\[
u(x,n) = \left( \frac{t}{t+1}\right)^{-n} \P_x(X_{T}=z) = \left( \frac{t}{t+1}\right)^{-n} \sum_{m=0}^\infty \frac{1}{t+1}\left(\frac{t}{t+1} \right)^m \P_x(X_{m}=z)
\]
is parabolic on $B(0,2r) \times \Z$, and it follows that
\begin{multline}
\left(\frac{t}{t+1}\right)^{-2r^2}\max_{x\in B(0,r)}\P_x\left(X_{T}=z \right) = \max_{\substack{x\in B(0,r)\\ r^2 \leq n \leq 2r^2}} u(x,n) \\\lesssim \min_{\substack{x\in B(0,r)\\ 3r^2 \leq n \leq 4r^2}} (u(x,n)+u(x,n+1) )
\lesssim_d \left(\frac{t}{t+1}\right)^{-3r^2}\min_{x\in B(0,r)}\P_x\left(X_{T}=z \right) 
\label{eq:simplePHI}
\end{multline}
for every $r\geq 1$.  The proof of \cref{lem:forget} will ultimately rely on a similar idea.

\begin{proof}[Proof of \cref{lem:forget}]
Fix $r,t\geq 1$.
Let $Q=B(0,2r)$ and write $\tau_{A,Q}$ for the first hitting time of $A$ after $\tau_{\partial Q}$. By the strong Markov property we get 
\begin{align}\label{eq:ineqden}
	\prstart{\tau_A^+>T}{x} \geq \prstart{\tau_{\partial Q}<\tau_A^+, \tau_{\partial Q}<T, \tau_{A,Q}>T }{x} \gtrsim_d \prstart{\tau_{\partial Q}<\tau_A^+, \tau_{\partial Q}<T}{x},
\end{align}
where the final inequality follows from the fact that, since $d\geq 3$, the probability that a simple random walk started at $\partial Q=\partial B(0,2r)$ never hits $B(0,r)$ is at least a positive constant.

Since $d(0,y)\geq 10r$, it follows that on the event $X_{T}=y$ we have $\tau_{\partial Q}<T$. By the strong Markov property applied to the time $\tau_{\partial Q}$ and the memoryless property of $T$ we obtain that
\begin{align}
&\prstart{X_{T}=y, \tau_D>T,\tau_A^+>T}{x} = \sum_{z\in \partial Q}\prstart{X_{T}=y, \tau_D>T, \tau_A^+>T, X_{\tau_{\partial Q}}=z, \tau_{\partial Q}<T}{x}	
\nonumber\\
&\hspace{3cm}= \sum_{z\in \partial Q} \prstart{X_{T}=y,\tau_D>T, \tau_A^+>T}{z}\prstart{X_{\tau_{\partial Q}}=z, \tau_{\partial Q}<T, \tau_{\partial Q}<\tau_A^+}{x} 
\nonumber\\
&\hspace{3cm}\leq \sum_{z\in \partial Q} \prstart{X_{T}=y,\tau_D>T}{z}\prstart{X_{\tau_{\partial Q}}=z, \tau_{\partial Q}<T, \tau_{\partial Q}<\tau_A^+}{x}.
\label{eq:prePHI}
\end{align}
%
Applying the parabolic Harnack inequality~\eqref{eq:PHI} to the function
\begin{multline*}
u(n,x) = \left(\frac{t}{t+1} \right)^{-n} \prstart{X_{T}=y, \tau_D>T}{x} \\= \left(\frac{t}{t+1} \right)^{-n} \sum_{m = 0}^\infty  \frac{1}{t+1}\left(\frac{t}{t+1} \right)^{m}\prstart{X_{m}=y, \tau_D>m}{x}
\end{multline*}
which is parabolic on $B(0,10r-1) \times \Z$, we obtain that
\begin{multline*}
\left(\frac{t}{t+1}\right)^{-8 r^2} \max_{z\in B(0,2r)}\prstart{X_{T}=y,\tau_D>T}{z} 
\\
= \max_{\substack{z\in B(0,2r)\\ 4r^2 \leq n \leq 8r^2}} u(z,n)
\lesssim \min_{\substack{z\in B(0,2r)\\ 12 r^2 \leq n \leq 16r^2}} (u(z,n)+u(z,n+1))
\\ \lesssim_d \left(\frac{t}{t+1}\right)^{-12 r^2} \min_{z\in B(0,2r)}\prstart{X_{T}=y,\tau_D>T}{z}
\end{multline*}
for every $r\geq 1$ and $t\geq 2$. Substituting this estimate back into \eqref{eq:prePHI} we obtain that
\begin{align*}
\prstart{X_{T}=y, \tau_D>T,\tau_A^+>T}{x} \lesssim_d \left(\frac{t}{t+1}\right)^{-4 r^2}\prstart{X_{T}=y,\tau_D>T}{0}\prstart{\tau_{\partial Q}<T, \tau_{\partial Q}<\tau_A^+}{x}.
\end{align*}
This together with~\eqref{eq:ineqden} completes the proof of the claim. 
\end{proof}

\begin{proof}[Proof of \cref{lem:capgeom}] We will prove the claim concerning the capacity; the proof of the claim concerning the length is very similar but uses \cref{lem:amounterased} in place of \cref{prop:firstconcentration}. 
We will prove the stronger claim that there exists a constant $C$ such that
\begin{equation}
\prcond{\cpp{\LE\Bigl(X^T\Bigr)}\leq \frac{t}{\lambda \log t}}{X_T=x}{0} \lesssim_{\eps} \frac{1}{(\log t)^{2-\eps}}
\end{equation}
for every $\eps>0$, $\lambda \geq \log \log t$, and $x \in \Z^4$ with $\|x\| \geq \sqrt{C \lambda^{-1} t \log \log t}$.
We first set up some relevant notation.
For each $r \geq 1$, let $\tau_r$ be the first time that the random walk visits $\partial B(X_0,r)$. It is a standard consequence of Azuma's inequality \cite[Proposition 2.1.2(b)]{LawLim} that there exists a positive constant $c$ such that
\begin{equation}
\label{eq:Azuma}
\P(\tau_r \leq n) \leq \exp\left[-\frac{cr^2}{n}\right]
\end{equation}
for every $r,n \geq 1$. Let $C_1$ be the constant from \cref{prop:firstconcentration} and define $C_2 = 4C_1/c$.
 Fix $t\geq 2$ and $\lambda \geq \log \log t$. To lighten notation, we let $L =\ler{}{X[0,T]}$ and write $\rho = |L|$ so that $L_\rho=X_{T}$.
We set 
\[r= \left \lceil \sqrt{C_2 \lambda^{-1} t \log \log t} \right \rceil \qquad \text{ and } \qquad R= 2r,   \]
so that $r \leq R \lesssim \sqrt{t}$. 
As in \cref{lem:end_segments_decorrelation}, we define
  $\kappa=\kappa_r$ to be the first hitting time of $\partial B(X_0,r)$ by the loop-erased random walk $L$ and define
   $\sigma$ to be the last hitting time of $\partial B(X_{T},r)$ by $L$. We define these times to be equal to $\rho$ and $0$ respectively on  the events that the relevant sets are not hit (this case will be irrelevant to us).
   Fix $x$ with $\|x\|_2 \geq 100 R$ and
    write $K=t/(\lambda \log t)$.
For each $y\in \Z^4$, let $\Omega_y$ be the set of paths starting at $y$ and ending in $\partial B(y,r)$ that have capacity at most $K$. Then we have by \cref{lem:end_segments_decorrelation} that
\begin{align*}\label{eq:boundbyb1b2}
&\prcond{\cpp{L}\leq K}{X_T=x}{0}\\
&\hspace{1.8cm}\leq 
	\prcond{\cpp{L[0,\kappa]} \leq K \text{ and } \cpp{L[\sigma,\rho]} \leq K}{X_T=x}{0}\nonumber
	\\
	&\hspace{1.8cm}=\sum_{\omega \in \Omega_0} \sum_{\eta \in \Omega_x} \P_0(L[0,\kappa]=\omega)\P_x(L[0,\kappa] = \eta \mid X_T=0) \frac{\prcond{X_{T}=\eta_m, \tau_{\eta} = T}{\tau_{\omega}^+>T}{\omega_k}}{\prstart{X_{T}=\eta_m, \tau_{\eta} =T}{0}}
	\\
	&\hspace{1.8cm}\leq  \prstart{\cpp{L[0,\kappa]} \leq K, L_\kappa \in \partial B(0,r)}{0} \prcond{\cpp{L[0,\kappa]} \leq K}{X_T=0}{x}
	\\
	&\hspace{1.8cm}\hspace{6.6cm} \cdot \sup_{\omega \in \Omega_0} \sup_{\eta \in \Omega_x}\frac{\prcond{X_{T}=\eta_m, \tau_{\eta} = T}{\tau_{\omega}^+>T}{\omega_k}}{\prstart{X_{T}=\eta_m, \tau_{\eta} =T}{0}},
\end{align*}
where we take $k$ and $m$ to be the lengths of $\omega$ and $\eta$ respectively.
Since $r \lesssim \sqrt{t}$, \cref{lem:forget} yields that
\[
\sup_{\omega \in \Omega_0} \sup_{\eta \in \Omega_x}\frac{\prcond{X_{T}=\eta_m, \tau_{\eta} = T}{\tau_{\omega}^+>T}{\omega_k}}{\prstart{X_{T}=\eta_m, \tau_{\eta} =T}{0}} \lesssim \left(\frac{t}{t+1}\right)^{-4r^2} \lesssim 1
\]
 and hence that
\begin{multline}\label{eq:boundbyb1b2}
\prcond{\cpp{L}\leq K}{X_T=x}{0} 
	\\\lesssim  \prstart{\cpp{L[0,\kappa]} \leq K,\, L_\kappa \in \partial B(0,r)}{0} \prcond{\cpp{L[0,\kappa]} \leq K}{X_T=0}{x}.
\end{multline}
To complete the proof, it therefore suffices to prove that
\begin{align}
\label{eq:secondprobgoal}
	\prstart{\cpp{L[0,\kappa] } \leq K,\, L_\kappa \in \partial B(0,r) }{0} &\lesssim_{\eps} \frac{1}{(\log t)^{1-\eps}}\quad \text{ and }\\
\label{eq:firstprobgoal}
	\prcond{\cpp{L[0,\kappa]}\leq K}{X_{T}=0}{x} &\lesssim_{\eps} \frac{1}{(\log t)^{1-\eps}}
\end{align}
for every $\eps>0$.

We begin by proving \eqref{eq:secondprobgoal}, which is easier. Recall that $\tau_r$ is the first time that $X$ visits $\partial B(X_0,r)$. (Note the distinction between $\tau_r$ and $\kappa$, the latter of which was defined in terms of the loop-erasure of $X$.) On the event that $L_\kappa \in \partial B(0,r)$ we have that $\tau_r \leq T$ and hence that $\LE_\infty(X^{n})$ is an initial segment of $L[0,\kappa]$ for every $n \leq \tau_r$.  
By \cref{prop:firstconcentration} and \eqref{eq:Azuma}, if we define $n= \lfloor 2 C_1 \lambda^{-1} t \rfloor$ then we have that
	\begin{align*}
\prstart{\cpp{L[0,\kappa] } \leq K, L_\kappa\in \partial \B(X_0,r)}{x}&\leq \prstart{\cpp{\ler{\infty}{X^n}} \leq K}{x} + \P_x(\tau_r \leq n)\\
& \lesssim_\eps \frac{1}{(\log t)^{1-\eps}} + \exp\left[-c \frac{r^2}{n}\right] \leq \frac{1}{(\log t)^{1-\eps}} + \frac{1}{\log t}
	\end{align*}
	for every $\eps>0$ by definition of $n,r$ and the constant $C_2$, completing the proof of \eqref{eq:secondprobgoal}.

We  now prove \eqref{eq:firstprobgoal}. Let $\tau_r$ and $\tau_R$ be the first times that $X$ visits $\partial B(x,r)$ and $\partial B(x,R)$ respectively. 
On the event $X_T=0$ we must have that $\tau_r \leq \tau_R \leq T$ and hence that $\LE_\infty(X^{\tau_r})$ is an initial segment of $L[0,\kappa]=\LE(X^T)[0,\kappa]$. 
Let  $\til{\tau}$ be the first time that $X$ hits $X[0,\tau_r]$ after reaching $\partial B(X_0,R)$, setting $\til\tau=\infty$ if this never happens. 
We use the union bound
\begin{multline}\label{eq:capKler}
\prstart{\cpp{\ler{\infty}{X^{\tau_r}}}\leq K, X_{T}=0}{x} \leq \prstart{\cpp{\ler{\infty}{X^{\tau_r}}}\leq K, X_{T}=0, \widetilde \tau > T}{x} \\
+ \prstart{X_{T}=0, \widetilde \tau \leq T}{x}.
\end{multline}
To bound the second term, we use the strong Markov property at time $\til{\tau}$, the memoryless property of the geometric distribution, and the parabolic Harnack inequality as in~\eqref{eq:simplePHI} to obtain that
\begin{align}\label{eq:strongmemoryharnach}
	\prstart{\widetilde \tau \leq T, X_{T}=0}{x}  = 
  \estart{
  \prstart{X_{T}=0}{X_{\til{\tau}_r}}
\mathbbm{1}\left(\widetilde \tau \leq T\right)
  }{x}  \asymp \prstart{X_{T}=0}{x} \prstart{\widetilde \tau \leq T}{x}.
\end{align}
 Let $S$ be an independent simple random walk on $\Z^4$ started at some vertex $y$ and let $\P_{x,y}$ denote the joint law of $X$ and $S$. Letting $m:= \lceil r^2\log \log r /2 c\rceil$, we have by \eqref{eq:Azuma} that
 \begin{align}\label{eq:boundfromtaur}
 	\begin{split}
 	\prstart{\widetilde{\tau} \leq T}{x} &\leq \max_{y\in \partial B(x,R)}\prstart{X[0,\tau_r]\cap S[0,\infty)\neq \emptyset}{x,y} \\ 
 	&\lesssim 
  \frac{1}{(\log t)^2}
   + \max_{y\in \partial B(x,R)}\prstart{X[0,m]\cap S[0,\infty)\neq \emptyset}{x,y}.
 	\end{split}
 \end{align}
Applying the non-intersection estimate \cite[Theorem~4.3.3]{Law91} yields that 
\[
\max_{y\in \partial B(x,R)}\prstart{X[0,m]\cap S[0,\infty)\neq \emptyset}{x,y} \lesssim \frac{\log \log t}{\log t},
\]
and it follows from this together with \eqref{eq:boundfromtaur} and \eqref{eq:strongmemoryharnach} that
\begin{align}
\label{eq:seconprobleinfty2}
	\prstart{\widetilde \tau \leq T,X_{T}=0}{x} \lesssim  \frac{\log \log t}{\log t}\cdot  \prstart{X_{T}=0}{x}
\end{align}
as required.
We now bound the first term on the right of~\eqref{eq:capKler}. We write $\ler{\tau_R}{X^{\tau_r}}$ for the part of $\ler{}{X^{\tau_R}}$ that is contributed by the first $\tau_r$ steps of the walk. Note that on the event $\{\widetilde \tau>T\}$, we have that $\ler{\infty}{X^{\tau_r}} =  \ler{\tau_R}{X^{\tau_r}}$, and hence we obtain
  \begin{align*}
  	&\prstart{\cpp{\ler{\infty}{X^{\tau_r}}}\leq K, X_{T}=0, \widetilde \tau > T}{x}
  	\leq \prstart{\cpp{\ler{\tau_R}{X^{\tau_r}}}\leq K, X_{T}=0, \tau_R<T}{x} \\
  	&= 
    \estart{\prstart{X_T=0}{X_{\tau_R}} 
    \mathbbm{1}\left(
    \cpp{\ler{\tau_R}{X^{\tau_r}}}
    \leq K
    \right)
    }{x}\lesssim \prstart{X_T=0}{x} \prstart{\cpp{\ler{\infty}{X^{\tau_r}}}\leq K}{x},
  \end{align*}
  where for the equality we used the strong Markov property applied to $\tau_R$ and the memoryless property of $T$ and for the last inequality we used the parabolic Harnack inequality again as in~\eqref{eq:simplePHI}.
Defining $n= \lceil 2 C_1 t/\lambda\rceil$ as above, we can once again apply \cref{prop:firstconcentration} and \eqref{eq:Azuma} to deduce that
\begin{equation}\label{eq:seconprobleinfty}
	\prstart{\cpp{\ler{\infty}{X^{\tau_r}}}\leq K}{x}  \leq \prstart{\cpp{\ler{\infty}{X^n}} \leq K}{x} + \prstart{\tau_{r}\leq n}{x}
	\lesssim_\eps \frac{1}{(\log t)^{1-\eps}}
\end{equation}
for every $\eps>0$. Substituting \eqref{eq:seconprobleinfty2} and \eqref{eq:seconprobleinfty} into \eqref{eq:capKler} completes the proof. \qedhere
\end{proof}

\part{The uniform spanning tree}
\label{part:UST}

\section{Background}
\label{sec:UST_background}

In this section we recall background on uniform spanning trees that will be applied in the remainder of the paper.

\medskip
\textbf{Orientations, the future, and the past.} Let $d\geq 2$. It is often convenient to think of both the uniform spanning forest of $\Z^d$ and the $0$-wired uniform spanning forest of $\Z^d$ as \emph{oriented} forests: Since every infinite tree is one-ended in both models \cite{USFBenLyPeSc,LMS08}, there is a unique orientation of each infinite tree such that every vertex has exactly one oriented edge emanating from it. Similarly, we orient the tree containing the origin in the $0$-wired uniform spanning forest towards the origin, so that every vertex other than the origin has a unique oriented edge emanating from it. (Note that it is also possible to sensibly define the oriented wired uniform spanning forest without knowing that every tree is one-ended \cite[Section 5]{USFBenLyPeSc}. This is important for several of the proofs of one-endedness, but will not concern us here.) This orientation leads to the following very natural equivalent definition of the past of a vertex: A vertex $u$ is in the \textbf{future} of a vertex $v$ if the unique oriented path emanating from $v$ passes through $u$; $u$ is in the \textbf{past} of $v$ if $v$ is in the future of $u$.

When considering the uniform spanning tree $\fT$ of $\Z^4$, we write $\fP(x,n)$ for the set of vertices in the past $\fP(x)$ of $x$ in $\fT$ that have intrinsic distance at most $n\geq 0$ from $x$, and write $\partial \fP(x,n)$ for the set of vertices in $\fP(x)$ that have intrinsic distance exactly $n$ from $x$. Similarly, when considering the $0$-wired uniform spanning forest $\fF_0$ of $\Z^4$, we write $\fP_0(x,n)$ for the set of vertices in the past $\fP_0(x)$ of $x$ in $\fF_0$ that have intrinsic distance at most $n\geq 0$ from $x$, and write $\partial \fP_0(x,n)$ for the set of vertices in $\fP_0(x)$ that have intrinsic distance exactly $n$ from $x$. Note that the past $\fP_0(0)$ of the origin in $\fF_0$ is equal to the entire component $\fT_0$ of $0$; it is useful to use the notation $\fP_0(0)$ anyway to strengthen the analogy between the two models.

\medskip

\textbf{The stochastic domination property.} 
The $v$-WUSF and WUSF are related by the following very useful stochastic domination property, which is proven in \cite[Lemma 2.1]{1804.04120}.
%
 We denote by $\operatorname{past}_F(v)$ the past of $v$ in the oriented forest $F$, which need not be spanning. We write $\fT_v$ for the tree containing $v$ in the $v$-WUSF $\F_v$, and write $\Gamma(u,\infty)$ and $\Gamma_v(u,\infty)$ for the future of $u$ in the WUSF $\F$ and $v$-WUSF $\F_v$ respectively.

\begin{lemma}[Stochastic Domination] 
\label{lem:domination}
Let $G$ be an infinite transient network,  let $\F$ be an oriented wired uniform spanning forest of $G$, and for each vertex $v$ of $G$ let $\F_v$ be an oriented $v$-wired uniform spanning forest of $G$.
Let $K$ be a finite set of vertices of $G$, and define $F(K) =\bigcup_{u\in K} \Gamma(u,\infty)$ and $F_v(K) =\bigcup_{u\in K} \Gamma_v(u,\infty)$. Then for every $u\in K$ and every increasing event $\sA \subseteq \{0,1\}^E$ we have that 
\begin{align}\P\Bigl( \operatorname{past}_{\F \setminus F(K)}(u) \in \sA \mid F(K) \Bigr) &\leq \P\bigl(\fT_u \in \sA \bigr)\qquad \text{and}
\label{eq:dom1}\\
\P\Bigl( \operatorname{past}_{\F_v\setminus F_v(K)}(u) \in \sA \mid F_v(K) \Bigr) &\leq \P\bigl(\fT_u \in \sA \bigr).
\label{eq:dom2}
\end{align}
\end{lemma}

This lemma often plays a similar role to that played by the BK inequality in percolation, enabling us to ignore the interactions between different parts of the tree.

\medskip

\textbf{Wilson's algorithm}
 \cite{Wilson96} establishes a powerful and direct probabilistic relationship between the uniform spanning tree and loop-erased random walk that is of central importance to most modern work on uniform spanning trees. This algorithm was generalized to infinite transient graphs by Benjamini, Lyons, Peres, and Schramm \cite{USFBenLyPeSc}; this generalization is known as \emph{Wilson's algorithm rooted at infinity}. 


We now briefly recall Wilson's algorithm rooted at infinity, referring the reader to  \cite[Chapters 4 and 10]{LP:book} for further background.
Suppose that $G$ is an infinite transient graph and let $v_1,v_2,\ldots$ be an enumeration of its vertices. We set $\F^{0}$ to be the empty forest, which contains no vertices or edges. Having defined $\F^n$ for some $n\geq 0$, we next define $\F^{n+1}$ as follows: run a random walk started from $v_{n+1}$ until the first time it hits $\F^n$, running it indefinitely if it never hits $\F^n$. Loop erase the path obtained and add the edges of the loop erasure to $\F^n$ in order to obtain $\F^{n+1}$. Finally define $\F=\bigcup_{n\geq 0} \F^n$. Then the random forest $\F$ has the law of the WUSF of $G$ \cite[Theorem 5.1]{USFBenLyPeSc}. If in addition we orient each edge of $\F$ in the direction it was crossed by the loop erasure used in the algorithm, then we obtain a sample of the oriented WUSF. It follows in particular that the future of the origin is distributed as a loop-erased random walk.   
Alternatively, if we set $\F^0$ to be the forest with vertex set $\{v\}$ and no edges (rather than the empty forest as before), we obtain a sample of the oriented $v$-wired uniform spanning forest \cite{JarRed08}.

\subsection{The interlacement Aldous--Broder algorithm}
\label{sec:AldousBroder}

The \textbf{random interlacement process} was introduced by Sznitman \cite{Sznitman} as a means to describe the ``local picture'' left by the trace of a simple random walk on a torus, and was generalized to arbitrary transient graphs by Teixeira \cite{Teix09}. It was shown in \cite{hutchcroft2015interlacements} that the random interlacement process can be used to generalize the \emph{Aldous--Broder} algorithm to infinite transient graphs. The resulting interlacement Aldous--Broder algorithm is of pivotal importance both to the computation of the scaling exponents for the USF in high dimensions \cite{1804.04120} and to our computation of the logarithmic corrections in four dimensions.

In order to formally define the random interlacement process on a transient graph $G$, first recall that for each $-\infty \leq n \leq m \leq \infty$ we set $L(n,m)$ to be the graph with vertex set $\{i \in \Z : n \leq i\leq m\}$ and edge set $\{\{i,i+1\}: n\leq i \leq m-1\}$. We now define $\cW(n,m)$ to be the set of multigraph homomorphisms from $L(n,m)$ to $G$ that are transient, i.e.\ for which the preimage of each vertex of $G$ is finite. We define the set $\cW$ to be
\[ \cW:= \bigcup\left\{\cW(n,m): -\infty \leq n \leq m \leq \infty\right\}.\]
The set $\cW$ is equipped with a natural Polish topology that makes the local times at vertices and first and last hitting times of finite sets continuous, see \cite[Section 3.2]{hutchcroft2015interlacements} for definitions. (Note that this topology is not the product topology.)
For each $k\in \N$, the \textbf{time shift}  $\theta_k:\cW\to \cW$ is defined so that for each $w\in \cW(n,m)$, the shifted path $\theta_k(w)\in \cW(n-k,m-k)$ is given by
\[ \theta_k(w)(i)=w(i+k) \qquad \text{ and } \qquad  \theta_k(w)(i,i+1)= w(i+k,i+k+1).\]
 Next define the space~$\cW^*$ to be the quotient 
\[\cW^* = \cW / \sim\,, \text{ where } w_1\sim w_2 \text{ if and only if } w_1 = \theta_k (w_2) \text{ for some k}.\]
Let $\pi : \cW \to \cW^*$ be the associated quotient function. We equip the set $\cW^*$ with the quotient topology and associated Borel $\sigma$-algebra, and call elements of $\cW^*$ \textbf{trajectories}. 

In order to define the intensity measure of the interlacement process, we first introduce some more notation. For $w \in \cW(n,m)$ we write $w^\leftarrow \in \cW(-m,-n)$ for the reversal of $w$, defined by 
$w^\leftarrow(i)=w(-i)$ for all $-m \leq i \leq -n$ and $w^\leftarrow(i,i+1)=w(-i,-i-1)$ for all $-m \leq i \leq -n-1$. For a subset $\sA \subseteq \cW$ we write $\sA^\leftarrow$ for the set
\[\sA^\leftarrow:= \{w \in \cW: w^\leftarrow \in \sA \}.\]
For each set $K \subseteq V$, we write $\cW_K(n,m)$ for the set of $w\in \cW(n,m)$ that hit $K$, i.e.\ those $w$ for which there exists $n\leq i\leq m$ such that $w(i)\in K$. We also define $\cW_K=\bigcup \{\cW_K(n,m): -\infty \leq n \leq m \leq \infty\}$ and 
define a measure $Q_K$ on $\cW$, supported on $\{w\in \cW: w(0)\in K\} \subseteq \cW_K$, by setting 
\begin{multline*}
Q_K\left( \{w \in \cW: w|_{(-\infty,0]} \in \sA,\, w(0) = u \text{ and } w|_{[0,\infty)} \in \sB \}\right)\\= 
 c(u)\prstart{X \in \sA^\leftarrow \text{ and } \tau_K^+ =\infty}{u} \prstart{X \in \sB}{u}
\end{multline*}
for each $u\in K$ and each two Borel subsets $\sA,\sB\subseteq \cW$, 
where $X$ is a simple random walk and $\tau_K^+ = \inf\{t\geq 1: X_t\in K\}$.
 For each set $K\subseteq V$ we denote by $\cW^*_K:=\pi(\cW_K)$ the set of trajectories that visit $K$.
It follows from the work of Sznitman~\cite{Sznitman} and Teixeira~\cite{Teix09} that there exists a unique $\sigma$-finite measure $Q^\ast$ on $\cW^*$ such that for every Borel set $\sA \subseteq \cW^*$ and every finite $K\subset V$,
\begin{equation}\label{eq:ildefn} Q^*(\sA \cap \cW^*_K) =  Q_K\left(\pi^{-1}(\sA)\right). \end{equation}
The measure $Q^*$ is called the \textbf{interlacement intensity measure}. We can now define the  \textbf{random interlacement process} $\sI$ as the Poisson point process on $\cW^*\times \R$ with intensity measure $Q^* \otimes \Lambda$, where $\Lambda$ is the Lebesgue measure on $\R$. (Equivalently, it is also possible to construct the interlacement process by taking a limit of the random walk on an exhaustion with wired boundary conditions, see~\cite[Proposition 3.3]{hutchcroft2015interlacements} for details.) For each $t\in\R$ and $A \subseteq \R$, we write $\sI_t$ for the set of $w\in \cW^*$ such that $(w,t)\in\sI$, and  write $\sI_{A}$ for the intersection of $\sI$ with $\cW^* \times A$. 

\medskip

We now describe the interlacement Aldous--Broder algorithm. Let $G=(V,E)$ be an infinite transient network and let $\sI$ be the interlacement process on $G$.
 For every~$t\in \R$ and each vertex $v$ of $G$, let $\sigma_t(v)=\sigma_t(v,\sI)$ be the smallest time greater than or equal to $t$ such that there exists a trajectory $W_{\sigma_t(v)} \in \sI_{\sigma_t(v)}$ hitting $v$, and note that the trajectory $W_{\sigma_t(v)}$ is unique for every $t\in \R$ and $v\in V$ almost surely. We define $e_t(v)=e_t(v,\sI)$ to be the oriented edge of $G$ that is traversed by the trajectory $W_{\sigma_t(v)}$ as it enters $v$ for the first time, and define 
\[
\AB_t(\sI):=\Bigl\{ e_t(v,\sI)^\leftarrow: v \in V \Bigr\}.
\]
Theorem 1.1 of \cite{hutchcroft2015interlacements} states that $\AB_t(\sI)$ has the law of the oriented wired uniform spanning forest of $G$ for every $t\in \R$. Moreover, \cite[Proposition 4.2]{hutchcroft2015interlacements} states that $\langle \AB_t(\sI)\rangle_{t\in \R}$ is a stationary, ergodic, mixing, stochastically continuous  Markov process. 

\medskip

\textbf{$v$-wired variants.}
In \cite[Section 3.1]{1804.04120}, a $v$-wired version of the random interlacement process and Aldous--Broder algorithm was developed, which we now describe.
Let $G$ be a (not necessarily transient) network and let $v$ be a vertex of $G$. Let $X$ be a simple random walk on $G$ that is not necessarily started from $v$. Recall that we write $\tau_v$ for the first time that $X$ visits $v$ and  $\tau^+_K$ for the first positive time that $X$ visits $K$. 
Recall also that we write $X^T$ for the random walk stopped at the (possibly random and/or infinite) time $T$, which is considered to be an element of $\cW(0,T)$. In particular, if $X$ is started at $v$ then $X^{\tau_v}$ is the path of length zero at $v$. Analogously to before, for every finite set $K \subset V$ we define a measure $Q_{v,K}$ on $\cW$ via $Q_{v,K}(\{w \in \cW : w(0)\notin K\})=0$,
\begin{multline*}
Q_{v,K}\left( \{w \in \cW: w|_{(-\infty,0]} \in \sA,\, w(0) = u \text{ and } w|_{[0,\infty)} \in \sB \}\right)\\= 
 c(u)\bP_u\big( X^{\tau_v} \in \sA^\leftarrow \text{ and } \tau_K^+ >\tau_v \big) \bP_u\big(X^{\tau_v} \in \sB \big) 
\end{multline*}
for every $u\in K\setminus \{v\}$ and every two Borel sets $\sA,\sB \subseteq \cW$, and 
\begin{multline*}
Q_{v,K}\left( \{w \in \cW: w|_{(-\infty,0]} \in \sA,\, w(0) = v \text{ and } w|_{[0,\infty)} \in \sB \}\right)\\=
c(v)\mathbbm{1}(w_0 \in \sA^\leftarrow) \bP_v\big(X^{\tau_v} \in \sB \big)
+c(v)\bP_v\big( X^{\tau_v} \in \sA^\leftarrow \text{ and } \tau_K^+ = \infty \big) \mathbbm{1}(w_0 \in \sB)
\end{multline*}
for every two Borel sets $\sA,\sB \subseteq \cW$ if $v\in K$, 
where we write $w_0\in \cW(0,0)$ for the path of length zero at $v$. It is proven in \cite[Corollary 3.2]{1804.04120} that there exists a unique locally finite measure $Q^*_v$ on $\cW^*$, known as the \textbf{$v$-wired interlacement intensity measure}, satisfying the consistency condition
 \begin{equation}
 \label{eq:rootedintensitydef}
 Q^*_v(\sA \cap \cW^*_K) = Q_{v,K}(\pi^{-1}(\sA))
 \end{equation}
for every finite set $K \subset V$ and every Borel set $\sA \subseteq \cW^*$. Analogously to before, we define the \textbf{$v$-wired interlacement process} $\sI_v$ to be the Poisson point process on $\cW^*\times \R$ with intensity measure $Q_v^* \otimes \Lambda$, where $\Lambda$ is the Lebesgue measure on $\R$. 
The process $\sI_v$ may contain trajectories that are either doubly infinite, singly infinite and ending at $v$, singly infinite and starting at $v$, or finite and both starting and ending at $v$. The analogue of the Aldous--Broder algorithm in this context is that if $\sI_v$ is a $v$-wired interlacement process on an infinite network $G$ then
\[
\AB_{v,t}(\sI_v):=\left\{ e_t\left(u,\sI_v\right)^\leftarrow : u \in V\setminus \{v\}\right\}
\]
is distributed as the oriented $v$-wired uniform spanning forest of $G$ for each $t\in \R$ \cite[Proposition 3.3]{1804.04120}; this proposition also states that $\langle \AB_{v,t}(\sI_v) \rangle_{t\in \R}$ is a stationary, ergodic, mixing, stochastically continuous Markov process.

\subsection{Capacity and evolution of the past}

 Recall that if $K$ is a finite set of vertices of $G$, then 
\[
\cpc{}{K}= Q^*(\cW^*_K)=\sum_{v\in K} c(v)\bP_v(\tau^+_K=\infty).
\]
Similarly, we define the \textbf{$v$-wired capacity} of a finite set $K$ via
\[
\cpc{v}{K} = Q_v^*(\cW^*_K) = \sum_{u\in K \setminus v} c(u)\bP_u(\tau^+_K>\tau_v) + c(v)\mathbbm{1}(v\in K)\left(1+\prstart{\tau^+_K=\infty}{v}\right),
\]
where we use the convention that $\tau^+_K>\tau_v$ in the case where they are both equal to $\infty$. Note that if $v\notin K$ then $\cpc{v}{K}$ is the effective conductance between $K$ and $\{\infty,v\}$.
It follows from the definition of the interlacement process that for every $a<b$, the number of trajectories
in $\sI_{[a,b]}$ that hit $K$ is a Poisson random variable with parameter $|a-b| \cpc{}{K}$, while the number of trajectories in $\sI_{v,[a,b]}$ that hit $K$ is a Poisson random variable with parameter $|a-b| \cpc{v}{K}$.
It is proven in \cite[Eq. 3.3]{1804.04120} that 
\begin{align*}
\cpc{}{K} \leq \cpc{v}{K} &\leq \cpc{}{K} + 3c(v)
\end{align*}
for every $K$ and $v$, so that the capacity and $v$-wired capacity coincide up to an additive constant in networks with bounded vertex conductances.

We now recall~\cite[Lemmas 3.4 and 3.5]{1804.04120}, 
which describe how the past of a vertex evolves in time under the dynamics induced by the interlacement Aldous--Broder algorithm. The idea is that, when running time backwards, the past $\fP_{-t}(v)$ of $v$ in $\F_{-t}$ can only become larger at time $-t$ if $v$ is hit by a trajectory at time $-t$. At all other times, it decreases as the past of $v$ is hit by other trajectories. For a set $A\subseteq \R$, we denote by $\cI_{A}$ the set of vertices that are hit by some trajectory in $\sI_{A}$ and we also write $\fP_t(v)$ for the past of $v$ in the forest $\F_t$.

\begin{lemma}
\label{lem:PastDynamics}
Let $G=(V,E)$ be a transient network, let $\sI$ be the interlacement process on $G$, and let $\langle \F_t\rangle_{t\in \R}=\langle \AB_t(\sI) \rangle_{t\in \R}$. Let $v$ be a vertex of $G$, and let $s<t$. If $v \notin \cI_{[s,t)}$, then $\fP_s(v)$ is equal to the component of $v$ in the subgraph of $\fP_t(v)$ induced by $V \setminus \cI_{[s,t)}$.
\end{lemma}

Similarly, for the $v$-wired case, for $A\subseteq \R$ we denote by $\cI_{v,A}$ the set of vertices that are hit by some trajectory in $\sI_{v,A}$, and by $\fP_{v,t}(u)$ the past of $u$ in the forest $\F_{v,t}$. 

\begin{lemma}
\label{lem:vPastDynamics}
Let $G=(V,E)$ be a network, let $v$ be a vertex of $G$, let $\sI_v$ be the $v$-wired interlacement process on $G$, and let $\langle \F_{v,t}\rangle_{t\in \R}=\langle \AB_{v,t}(\sI_v) \rangle_{t\in \R}$. Let $u$ be a vertex of $G$, and let $s<t$. If $u \notin \cI_{v,[s,t)}$, then $\fP_{v,s}(u)$ is equal to the component of $u$ in the subgraph of $\fP_{v,t}(u)$ induced by $V \setminus \cI_{v,[s,t)}$.
\end{lemma}

\subsection{A further Aldous--Broder variant}
\label{subsec:AB_variant}

It will be convenient for us to introduce a further variant of the Aldous--Broder algorithm first considered in the proof of \cite[Proposition 7.1]{hutchcroft2015interlacements}. Let $G$ be a transient graph, let $\sI$ be the interlacement process on $G$, and let $X$ be a simple random walk started from some vertex $v$ of $G$. The set $\sI\cup X:= \sI \cup \{(X,0)\}$ may be thought of as a point process on $\cW^* \times \R$; we will now briefly argue that 
\[\AB_0(\sI \cup X)
:= \{e_0\left(u,\sI \cup X\right)^\leftarrow : u \neq v \}
\] is distributed as the wired uniform spanning forest of $G$ but where the component of $v$ is oriented towards $v$ rather than towards infinity. (In particular, if we flip the orientations in the unique infinite path starting at $v$, then the resulting oriented forest is distributed as the oriented WUSF.) Here, as usual, $e_0\left(u,\sI \cup X\right)$ denotes the oriented edge that is crossed as $u$ is visited either by $X$ or by a trajectory of $\sI$ for the first time after time zero; Since $X$ is considered to arrive at time zero, it has priority over all the trajectories of $\sI$ when computing the forest. 
 In an exhaustion by finite graphs, this construction corresponds to first running a random walk from $v$ until hitting the distinguished boundary vertex, and then decomposing the rest of the walk into excursions from the boundary vertex. The details of the proof are very similar to those of \cite[Theorem 1.1]{hutchcroft2015interlacements}.

This equality of distributions implies various related equalities describing the tree obtained by applying the Aldous--Broder algorithm to a single walk. (These relations may also be proven directly.) Let $d\geq 3$. Given an infinite path $\gamma$ starting at the origin in $\Z^d$ and a vertex $u\in \gamma \setminus \{0\}$, define $e(u,\gamma)$ to be the edge that is crossed by $\gamma$ as it enters $u$ for the first time.
We define $\AB(\gamma)=\{e(u,\gamma)^\leftarrow :u\in \gamma \setminus \{0\}\}$, which is an infinite tree oriented towards the origin. If $X$ is a simple random walk started at the origin and $\sI$ is an independent random interlacement process on $\Z^d$ then $\AB(X)$ is contained in $\AB_0(\sI \cup X)$ since $X$ is given priority when computing $\AB_0(\sI \cup X)$. Since every component of the uniform spanning forest of $\Z^d$ is one-ended almost surely, $\AB_0(\sI \cup X)$ contains a unique infinite path starting at the origin, which is oriented towards the origin and distributed as a loop-erased random walk by Wilson's algorithm. Since $\AB(X)$ is infinite and contained in $\AB_0(\sI \cup X)$, it follows that $\AB(X)$ is one-ended and contains the same infinite path; indeed, this path is the unique infinite path in $\AB(X)$ starting from the origin. Note that this path is not the loop-erasure of $X$, but is rather constructed via an infinite-volume version of the reverse loop-erasure operation $\LE^R$ discussed in the proof of \cref{thm:reversibility}. Indeed, if we define $\LE^R(X)$ to be this path, then for each $n\geq 0$ the path connecting $0$ and $X_n$ is equal to $\LE^R(X^n):=\LE((X^n)^\leftarrow)^\leftarrow$ and $\LE^R(X)$ is equal to the limit of the paths $\LE^R(X^n)$ as $n\to\infty$.

Now suppose that $Y$ is a random walk on $\Z^d$ that is conditioned not to return to the origin after time zero. Since the law of $Y$ is absolutely continuous with respect to that of $X$, $\AB(Y)$ also contains a unique infinite path starting at $0$ almost surely, which we denote by $\LE^R(Y)$, and the law of $\LE^R(Y)$ is absolutely continuous with respect to that of $\LE^R(X)$. We claim that in fact $\LE^R(Y)$ is distributed exactly as a loop-erased random walk, so that these two paths have the same distribution. Indeed, for each $r\geq 1$ let $G_r^*$ be obtained from $\Z^d$ by contracting everything outside of $\Lambda_r$ into a single vertex $\partial_r$, and let $Y^r$ be a random walk on $G_r^*$ started at the origin, stopped when it first hits $\partial_r$, and conditioned to hit $\partial_r$ before returning to the origin. Let $W^r$ be a random walk on $G_n^*$ started at $\partial_r$, stopped when it first hits the origin, and conditioned to hit the origin before returning to $\partial_r$. By reversibility, $Y^r$ and the time-reversal of $W^r$ have the same distribution. Moreover, if we define $\widetilde W_r$ to be an unconditioned random walk on $G_n^*$ started at $\partial_r$ and stopped when it first hits the origin, then $\LE(W^r)$ and $\LE(\widetilde W_r)$ clearly have the same distribution, and by Wilson's algorithm both paths are distributed as the path connecting $0$ to $\partial_r$ in the uniform spanning tree of $\Z^d$. Since $\LE^R(Y^r)$ is equal in distribution to the time-reversal of $\LE(W^r)$, the claim follows by taking the limit as $r\to\infty$.

\section{The intrinsic radius}
\label{sec:intrinsic}

The main goal of this section is to prove the parts of \cref{thm:wusf,thm:wusfo} concerning the tail of the intrinsic radius for both the UST and $v$-WUSF. These results will then be applied in the computation of the logarithmic corrections for the volume and extrinsic radius in the following sections. We begin by establishing the relevant lower bounds for both the intrinsic and extrinsic radius in \cref{sec:lower}. 
In \cref{subsec:weakl1} we study the expected volume of an intrinsic ball in the $0$-wired uniform spanning forest, introducing the \emph{weak $L^1$ technique} which we will make repeated use of throughout the rest of the paper.
Finally, in \cref{subsec:intrinsicupper} we apply the results of \cref{part:LERW,subsec:weakl1} to prove the upper bounds on the intrinsic radius.

\subsection{Lower bounds on the intrinsic and extrinsic radius}\label{sec:lower}

In this section we prove the lower bounds on the tails of the intrinsic and extrinsic radii that are stated in \cref{thm:wusf,thm:extrinsic,thm:wusfo}. In fact, we will prove the following stronger statements concerning the joint probability of a large intrinsic and extrinsic radius. 

\begin{prop}\label{pro:lowerbounds}
Let $\fP=\fP(0)$ be the past of the origin in the uniform spanning tree of $\Z^4$. Then
\begin{align*}
 \pr{\radint(\fP) \geq n \text{ and } \radext(\fP) \geq \sqrt{n}(\log n)^{1/6} } \gtrsim \frac{(\log n)^{1/3}}{n}
\end{align*}
 for every $n\geq 2$.
\end{prop}

\begin{prop}\label{pro:lowerbounds0}
Let $\fT_0$ be the component of the origin in the $0$-wired uniform spanning tree of $\Z^4$. Then
\begin{align*}
 \pr{\radint(\fT_0) \geq n \text{ and } \radext(\fT_0) \geq \sqrt{n}(\log n)^{1/6} } \gtrsim \frac{(\log n)^{2/3}}{n}
\end{align*}
 for every $n\geq 2$.
\end{prop}

Before proving the above propositions, we state and prove a lemma on the correlation between the displacement of the loop-erased random walk at time $n$ and the non-intersection event with an independent simple random walk. The upper bound of \eqref{eq:nonintersection_n_outside} is not needed for our present purposes but is included for later use in the proof of \cref{lem:exceptional_time}.

\begin{lemma}
\label{lem:displacement}
Let $X$ and $Y$ be independent random walks on $\Z^4$, both started from the origin. Then for every $\lambda >0$ there exists $N_\lambda < \infty$ such that 
\begin{equation}
\label{eq:nonintersectiondisplacement}
\pr{X(0,\infty)\cap \LE(Y)^n=\emptyset,\, \|\LE(Y)_n\|\geq \lambda n^{1/2}(\log n)^{1/6}} \gtrsim_\lambda \frac{1}{(\log n)^{1/3}}
\end{equation}
for every $n\geq N_\lambda$. Moreover, we have that
\begin{equation}
\label{eq:nonintersection_n_outside}
\pr{
\|\LE(Y)_n\|\geq \lambda n^{1/2}(\log n)^{1/6}
} \gtrsim_\lambda 1 
\quad\text{and}\quad
\pr{X(0,\infty)\cap \LE(Y)^n=\emptyset} \asymp \frac{1}{(\log n)^{1/3}}
\end{equation}
for every $n\geq 2$.
\end{lemma}



The proof of this lemma will use the following result on intersections of \emph{simple} random walks due to Lawler. (See also Albeverio and Zhou~\cite[Theorem~1.1]{AZ} for more general results.) We write $\sim$ to mean that the ratio of the two sides tends to $1$ in the appropriate limit.

\begin{theorem}[\cite{Law91}, Theorem 4.3.6]
\label{thm:AlbeverioZhou}
Let $X$ and $Y$ be independent random walks on $\Z^4$ both started at the origin. Let 
$0<a<b<\infty$. Then
\[
\pr{X(0,\infty)\cap Y[an,bn]\neq \emptyset}\sim \frac{\log (b/a)}{2\log n}
\]
 as $n\to\infty$.
\end{theorem}


%
Let us now use this estimate to prove \cref{lem:displacement}.

\begin{proof}[Proof of Lemma~\ref{lem:displacement}]
We start by proving the first estimate, \eqref{eq:nonintersectiondisplacement}.
Define $r=r(n)=n^{1/2}(\log n)^{1/6}$ and let
\[m_0 =\left\lfloor\frac{n(\log n)^{1/3}}{4}\right\rfloor, \quad m_1 =\left\lfloor\frac{n(\log n)^{1/3}}{2}\right\rfloor, \quad \text{ and } \quad m_2 =\left\lfloor 2n(\log n)^{1/3}\right\rfloor.\]
For each $m\geq 0$, let $\rho_m=\rho_m(Y)$ be maximal such that $\ell_{\rho_m} \leq m$ as defined in \cref{subsec:LERWbackground}. If $n \leq \rho_{m_2}$ then $\LE(Y)^n$ is an initial segment of $\LE(Y^{m_2})$, so that
\begin{multline}
\label{eq:displacement1}
\pr{X(0,\infty)\cap \LE(Y)^n=\emptyset,\, \|\LE(Y)_n\|\geq \lambda r}
\\\geq 
\pr{X(0,\infty)\cap \LE(Y^{m_2})=\emptyset,\, \rho_{m_1} \leq n \leq \rho_{m_2}, \text{ and }  \inf_{m \geq m_1} \|Y_m\|\geq \lambda r}
\end{multline}
and hence by \cref{lem:amounterased} there exists a constant $C_1$ such that
\begin{multline}
\label{eq:displacement2}
\pr{X(0,\infty)\cap \LE(Y)^n=\emptyset,\, \|\LE(Y)_n\|\geq \lambda r}
\\\geq 
\pr{X(0,\infty)\cap \LE(Y^{m_2})=\emptyset \text{ and }  \inf_{m \geq m_1} \|Y_m\|\geq \lambda r} - C_1\frac{\log \log n}{(\log n)^{2/3}}.
\end{multline}
For the first term, we have by \cref{thm:AlbeverioZhou} that there exists a constant $C_2$ such that
\begin{align}
\nonumber
&\pr{X(0,\infty)\cap \LE(Y^{m_2})=\emptyset \text{ and }  \inf_{m \geq m_1} \|Y_m\|\geq \lambda r} 
\\
\nonumber
&\hspace{1.15cm}\geq 
\pr{X(0,\infty)\cap \LE(Y^{m_0})=\emptyset,\, X(0,\infty) \cap Y[m_0,m_2] = \emptyset,\, \text{ and }  \inf_{m \geq m_1} \|Y_m\|\geq \lambda r}
\\
&\hspace{1.15cm}\geq
\pr{X(0,\infty)\cap \LE(Y^{m_0})=\emptyset\, \text{ and }  \inf_{m \geq m_1} \|Y_m\|\geq \lambda r} - C_2 \frac{1}{\log n}.
\label{eq:displacement3}
\end{align}
 Since $|m_1-m_0|\asymp r^2$ 
and the random variables $(Y_m-Y_{m_0})_{m \geq m_0}$ are independent of the event $\{X(0,\infty)\cap \LE(Y^{m_0})=\emptyset\}$, it is a consequence of the central limit theorem that there exists $N_\lambda < \infty$ such that
\begin{multline}
\label{eq:displacement4}
\pr{X(0,\infty)\cap \LE(Y^{m_0})=\emptyset\, \text{ and }  \inf_{m \geq m_1} \|Y_m\|\geq \lambda r} \\
\gtrsim \pr{X(0,\infty)\cap \LE(Y^{m_0})=\emptyset\, \text{ and }   \|Y_{m_1}\|\geq 2\lambda r} \\\gtrsim_\lambda\pr{X(0,\infty)\cap \LE(Y^{m_0})=\emptyset} 
\end{multline}
for every $n\geq N_\lambda$.
Finally, it follows from \cref{thm:logpaper} that
\begin{align}
\pr{X(0,\infty)\cap \LE(Y^{m_0})=\emptyset} 
&\gtrsim \frac{1}{(\log n)^{1/3}}
\label{eq:displacement5}
\end{align}
 The claim follows by putting together the estimates of \cref{eq:displacement1,eq:displacement2,eq:displacement3,eq:displacement4,eq:displacement5}. 

The first inequality of \eqref{eq:nonintersection_n_outside} follows by a similar proof to that of \eqref{eq:nonintersectiondisplacement}, while the lower bound of the second inequality is an immediate corollary of \eqref{eq:nonintersectiondisplacement}.
 It remains to prove the upper bound of~\eqref{eq:nonintersection_n_outside}. Let $k=\lfloor n^{1/8} \rfloor$, 
  we have by a union bound that
	\begin{align*}
		\pr{X(0,\infty)\cap \LE(Y)^n=\emptyset}&\leq \pr{X(0,k)\cap \LE(Y)=\emptyset} + \pr{X(0,k)\cap \LE(Y)[n,\infty)\neq \emptyset}.
	\end{align*}
	Using Theorem~\ref{thm:logpaper} we get that the first term on the right hand side is of order $(\log n)^{-1/3}$.
To bound the second term we observe that $X(0,k)$ is trivially contained in the ball $B(0,k)$ and hence that if we define $m=\lfloor n^{1/2} \rfloor$ then
	\begin{align*}
		\pr{X(0,k)\cap \LE(Y)[n,\infty)\neq \emptyset}&\lesssim  \pr{\ell_n(Y)<m} + \pr{Y[m,\infty)\cap B(0,k)\neq \emptyset}\\
		&\lesssim  \frac{\log\log n}{(\log n)^{2/3}} + \frac{k^2}{m} = o\left(\frac{1}{(\log n)^{1/3}} \right),
	\end{align*}
where we used \cref{lem:amounterased} to bound the first term and a standard random walk calculation (see \cite[Lemma 4.4]{1804.04120}) to bound the second.
\end{proof}

\begin{proof}[Proof of Proposition~\ref{pro:lowerbounds}]
To help the exposition, we first show the lower bound for the intrinsic radius, which is simpler, and then show how the proof can be modified to yield the lower bound for the extrinsic radius.
Fix $n\geq 3$.
Let $\eps>0$ and let $A_\epsilon$ be the event that $0$ is hit by a unique trajectory $W$ of the interlacement process in the interval $[0,\epsilon]$. We parameterize $W$ in such a way that it visits the origin for the first time at time zero, and set $X=W[0,\infty)$. 
As discussed in \cref{subsec:AB_variant}, applying the Aldous--Broder algorithm to $X$ yields a unique infinite path $\eta=\LE^R(X)$ starting at $0$ that is distributed as a loop-erased random walk.
 Moreover, the path $W(-\infty,0]$ in reverse time has the law of a simple random walk started from $0$ conditioned to never return to~$0$. Let $B$ be the event that $W(-\infty,0)$ does not intersect~$\eta^n=\eta[0,n]$ and that no other trajectory of the interlacement process arriving in $[0,\epsilon]$ hits $\eta^n$.
 Let $C_1$ be the constant from \cref{cor:lercap}. By the splitting property for Poisson processes we have that
\begin{align}\label{eq:lowerbound}\begin{split}
\pr{\radint(\fP)\geq n}&\geq \pr{A_\epsilon \cap B} 
	\\&\geq \epsilon\ \cc{\{0\}} e^{-\epsilon \cc{\{0\}}} \E{\mathbbm{1}(W(-\infty,0)\cap \eta^n=\emptyset) e^{-\epsilon \cc{\eta^n}}}\\
	&\gtrsim \epsilon  e^{-2C_1\epsilon n/(\log n)^{2/3} } \pr{W(-\infty,0)\cap \eta^n=\emptyset, \cc{\eta^n} \leq \frac{C_1 n}{(\log n)^{2/3}}}.
\end{split}
\end{align}
Let $S$ be a simple random walk in $\Z^4$ started from $0$ independent of $\eta$ and let $\tau_0^+$ be its first return time to $0$. It follows from the strong Markov property and~\eqref{eq:nonintersection_n_outside} of \cref{lem:displacement} that
\begin{multline}\label{eq:nonintseta}
	\pr{W(-\infty,0)\cap \eta^n=\emptyset} = \prcond{S(0,\infty)\cap \eta^n=\emptyset}{\tau_0^+=\infty}{0} \\
  \geq \prstart{S(0,\infty)\cap \eta^n=\emptyset}{0} \asymp\frac{1}{(\log n)^{1/3}}.
\end{multline}
 Combining this with \cref{cor:lercap} we obtain that there exist positive constants $C_2$ and $c_1$ such that
\begin{align*}
	\pr{W(-\infty,0)\cap \eta^n=\emptyset, \cc{\eta^n} \leq \frac{C_1 n}{(\log n)^{2/3}}} \geq \frac{C_2}{(\log n)^{1/3}} - \frac{c_1\log \log n}{(\log n)^{2/3}}
\end{align*}
and hence that
\begin{align}\label{eq:conccap}
	\pr{W(-\infty,0)\cap \eta^n=\emptyset, \cc{\eta^n} \leq \frac{C_1 n}{(\log n)^{2/3}}} \gtrsim \frac{1}{(\log n)^{1/3}}.
\end{align}
Taking $\epsilon=(\log n)^{2/3}/n$ and substituting \eqref{eq:conccap}  into~\eqref{eq:lowerbound}  yields the desired lower bound.

We now prove the joint lower bound on the intrinsic and extrinsic radii. Fix $n\geq 3$, let $\eps = (\log n)^{2/3}/n$, and let $A_\eps$, $B$, $W$, $\eta$, and $C_1$ be as above. Let $r=r(n) = n^{1/2} (\log n)^{1/6}$ and let $D$ be the event that $\|\eta_n\|_1 \geq  r$. Then we have by similar reasoning to above that
\begin{align*}
 &\pr{\radint(\fP)\geq  n \text{ and } \radext(\fP) \geq r} \geq \pr{A_\epsilon \cap B \cap D} 
	\nonumber
	\\&\hspace{3cm}\geq \epsilon\ \cc{\{0\}} e^{-\epsilon \cc{\{0\}}} \E{\1(W(-\infty,0)\cap \eta^n=\emptyset, \|\eta_n\|_1\geq  r) e^{-\epsilon \cc{\eta^n}}}
	\nonumber
	\\
	&\hspace{3cm}\gtrsim \frac{(\log n)^{2/3}}{n} \pr{W(-\infty,0)\cap \eta^n=\emptyset,\, \|\eta_n\|_1\geq  r,\, \cc{\eta^n} \leq \frac{C n}{(\log n)^{2/3}}}.
	\label{eq:lowerbounddisplacement}
\end{align*}
Putting this estimate together with  \cref{cor:lercap} and \cref{lem:displacement} and arguing as above yields that there exist positive constants~$c_2$, $C_3$, and $N$ such that
\begin{align*}
 &\pr{\radint(\fP)\geq  n \text{ and } \radext(\fP) \geq r}
\\&\hspace{1cm}\geq c_2\frac{(\log n)^{2/3}}{n} \pr{W(-\infty,0)\cap \eta^n=\emptyset,\, \|\eta_n\|_1\geq  r} - C_3 \frac{(\log n)^{2/3}}{n} \cdot \frac{\log \log n}{(\log n)^{2/3}}
	\\
	&\hspace{1cm}\geq c_2\frac{(\log n)^{2/3}}{n} \pr{S(0,\infty)\cap \eta^n=\emptyset,\, \|\eta_n\|_1\geq  r} \hspace{1.16em}- C_3 \frac{(\log n)^{2/3}}{n} \cdot \frac{\log \log n}{(\log n)^{2/3}}  
	   \gtrsim \frac{(\log n)^{1/3}}{n}
\end{align*}
for every $n \geq N$, where $S$ is a simple random walk in $\Z^4$ started from $0$ independent of $\eta$. 
\end{proof}

\begin{proof}[Proof of \cref{pro:lowerbounds0}]
The proof is similar to but easier than that of \cref{pro:lowerbounds} and we omit some details. Let $\eps = n^{-1} (\log n)^{2/3}$ and let $A_\epsilon$ be the event that $0$ is hit by a unique trajectory starting at $0$ of the $0$-wired interlacement process in the interval $[0,\epsilon]$ and that this trajectory is infinite. Let $W$ be this trajectory, which is distributed as a random walk started at $0$ and conditioned never to return. As discussed in \cref{subsec:AB_variant}, applying the Aldous--Broder algorithm to $W$ yields a unique infinite path $\eta=\LE^R(W)$ starting at $0$ that has the law of a loop-erased random walk in $\Z^4$.  Let $B$ be the event that the first $n$ steps of $\eta$ are not hit by any other trajectory of the interlacement process during $[0,\epsilon]$, let $r=n^{1/2}(\log n)^{1/6}$, and let $D$ be the event that $\|\eta_n\|_1 \geq r$. Since $\cpc{}{\eta^n}$ and $\cpc{0}{\eta^n}$ coincide up to an additive constant, we deduce by a similar argument to above that
\begin{align*}
	\pr{\radint(\fT_0)\geq n \text{ and } \radext(\fT_0) \geq r} &\geq \pr{A_\epsilon\cap B \cap D} 
  \gtrsim \epsilon \E{\mathbbm{1}(\|\eta_n\|_1 \geq r)e^{-\epsilon\cpc{}{\eta^n}}}
	\\&\gtrsim   \P(\|\eta_n\|_1 \geq r) \frac{(\log n)^{2/3}}{n} - o\left(\frac{(\log n)^{2/3}}{n}\right) \gtrsim \frac{(\log n)^{2/3}}{n}
\end{align*}
where the second term in the last line accounts for the event that $\eta^n$ has small capacity and the final inequality follows from \eqref{eq:nonintersection_n_outside}. 
\end{proof}

\subsection{Intrinsic balls and the weak $L^1$ method}
\label{subsec:weakl1}


In this section we estimate the expected number of points in the past of the origin in the UST and $0$-WUSF that lie inside an intrinsic ball. Similar estimates in the high-dimensional case are given in \cite[Section 6]{1804.04120}. Let us first note the following easy fact.

\begin{lemma}
\label{lemma:WUSF_ball}
Let $d\geq 2$ and let $\mathfrak{P}(0,n)$ be the intrinsic ball of radius $n$ around $0$ in the past of $0$ in the uniform spanning tree of $\Z^d$. Then
\[\E{|\partial\fP(0,n)|} = 1\] for every $n\geq 0$.
\end{lemma}

\begin{proof}
This is an immediate consequence of the \textbf{mass-transport principle} for $\Z^d$, which states that if $d\geq 1$ and $F:\Z^d \times \Z^d \to [0,\infty)$ is \textbf{diagonally invariant} in the sense that $F(x-z,y-z)=F(x,y)$ for every $x,y,z\in \Z^d$ then
\begin{equation}
\label{eq:MTP}
\sum_{x\in \Z^d}F(0,x)=\sum_{x\in \Z^d} F(x,0).
\end{equation}
The claim follows by considering the diagonally-invariant function $F(x,y)=\P(y$ is the $n$th vertex in the future of $x)$ and using the fact that every vertex has exactly $1$ vertex $n$ steps in its future. See \cite[Chapter 8]{LP:book} for further background on the mass-transport principle.
\end{proof}

Estimating the analogous quantity for the $0$-WUSF is a more subtle problem. We will use Lawler's concentration estimates on the number of points erased (\cref{lem:amounterased}) as well as our related polylogarithmic deviation bound of \cref{cor:capgeom}.

\begin{prop}
\label{prop:WUSFo_ball}
Let $\fP_0(0,n)$ be the intrinsic ball of radius $n$ around $0$ in the $0$--wired uniform spanning forest of $\Z^4$. Then
\[
 \E{|\fP_0(0,n)|}\asymp n (\log n)^{1/3} 
\] for every $n\geq 2$.
\end{prop}

Note in particular that  $\E{|\fP_0(0,n)|}$ and $\E{|\fP(0,n)|}$ differ in order by a factor of $(\log n)^{1/3}$. This is in contrast to the high-dimensional case, where these two quantities are always of the same order \cite[Lemma 6.6]{1804.04120}.

\medskip

\textbf{The weak $L^1$ method.}
The proof of \cref{prop:WUSFo_ball} will use the following bound on the conditional distribution of the length of a random walk given its loop-erasure, which is a 4d analogue of \cite[Lemma 4.3]{1804.04120}. 
We apply this estimate many times throughout the paper.

\begin{lemma}
\label{lem:weakl1}
Let $x\in \Z^4 \setminus \{0\}$, let $X$ be a random walk on $\Z^4$ started at $x$, and let $\gamma$ be a simple path connecting $x$ to the origin.  Then 
\[
\P_x\left(\tau_0 \geq n \mid \tau_0 < \infty \text{ and } \LE(X^{\tau_0}) = \gamma\right) \lesssim \frac{|\gamma| \log (|\gamma|+1)}{n}
\]
for every $n\geq 1$.
\end{lemma}

Note that while this lemma does \emph{not} give the correct order of the logarithmic correction for typical values of the loop-erasure, the fact that it holds no matter what we condition the loop-erasure to be will be very useful. (Moreover, the estimate is of the correct order when we condition the loop-erasure to be a straight line.) We explore the conditional distribution of the random walk given its loop-erasure in more detail in \cref{sec:typicaltime}.

The proof of \cref{lem:weakl1} will use some simple facts about \emph{weak $L^1$ norms} of random variables. Recall that if $Z$ is a real-valued random variable then the weak $L^1$ norm $\|Z\|_{1,w}$ of $Z$ is defined by
\[
\|Z\|_{1,w} = \sup_{t>0} t\P(|Z| \geq t) = \inf\left\{ \lambda \geq 0 : \P(|Z| \geq t) \leq \frac{\lambda}{t} \text{ for all $t>0$}\right\} \in [0,\infty].
\]
Despite its name, the weak $L^1$ norm is not a norm and does not satisfy the triangle inequality. Vershynin \cite{VershyninweakL1} observed that it does however satisfy the following weakened form of the triangle inequality: If $n\geq 2$ and $Z_1,\ldots,Z_n$ are real-valued random variables then
\begin{equation}
\label{eq:weakl1triangle}
\Bigl\| \sum_{i=1}^n Z_i\Bigr\|_{1,w} \leq 2e \log n \sum_{i=1}^n \left\|  Z_i\right\|_{1,w}.
\end{equation}

\begin{proof}[Proof of \cref{lem:weakl1}] 
As in the proof of~\cite[Lemmas 5.1 and 5.3]{1804.04120}, the definition of the loop-erasure implies that the  random variables $(\ell_{k+1}(X^{\tau_0})-\ell_{k}(X^{\tau_0}) )_{k=0}^{|\gamma|-1}$ are independent conditional on the event that $\tau_0<\infty$ and $\LE(X^{\tau_0})=\gamma$. 
The proof of the same lemma also implies that
			\vspace{0.2em}
			\begin{align}
			\vspace{0.2em}
				\P_u\left(\ell_{k+1}(X^{\tau_0})-\ell_k(X^{\tau_0}) =m+1 \mid \tau_0 < \infty,\, \LE(X^{\tau_0})=\gamma \right) &\leq 
\P_{\gamma_k}(X_m=\gamma_k, X^m \cap \gamma^{m-1} = \emptyset)
\nonumber\\
        &\leq p_m(0,0)
			\end{align}
			for every $0 \leq k \leq |\gamma|-1$ and $m\geq 0$, where $p_m(0,0)$ denotes the $m$-step return probability of simple random walk on $\Z^4$. It follows that there exists a constant $C$ such that
\begin{equation}
			\vspace{0.2em}
				\P_x\left(\ell_{k+1}(X^{\tau_0})-\ell_k(X^{\tau_0}) \geq m \mid \tau_0 < \infty,\, \LE(X^{\tau_0})=\gamma \right) \leq \frac{C}{m}
			\end{equation}
			for every $0\leq k \leq |\gamma|-1$ and $m\geq 1$. 
			 Thus, we may apply the weak triangle inequality for weak $L^1$ norms \eqref{eq:weakl1triangle} to obtain that
			\begin{multline*}
\P_x\left(\tau_0 \geq m \mid \tau_0 < \infty,\, \LE(X^{\tau_0})=\gamma \right) \\= \P_x\left( \sum_{k=0}^{|\gamma|-1} \left[\ell_{k+1}(X^{\tau_0})-\ell_k(X^{\tau_0}) \right]\geq m \mid \tau_0 < \infty,\, \LE(X^{\tau_0})=\gamma \right) \leq \frac{2e C |\gamma|\log(|\gamma|+1)}{m}
			\end{multline*}
			for every $m\geq 1$ as claimed. (We have written $\log(|\gamma|+1)$ rather than $\log |\gamma|$ here to deal with the case $|\gamma|=1$, in which case this inequality holds since $2e \log 2 \geq 1$.)
\end{proof}


\begin{proof}[Proof of \cref{prop:WUSFo_ball}]
Consider generating the $0$-WUSF using Wilson's algorithm, starting with a random walk $X$ started at the vertex $v$. This vertex $v$ will belong to $\fP_0(0,n)$ if and only if $\tau_0<\infty$ and $|\LE(X^{\tau_0})| \leq n$. Summing over $v$, we obtain that
\[
\E{|\fP_0(0,n)|} = \sum_{v\in \Z^4}  \P_v(\tau_0 <\infty \text{ and } |\LE(X^{\tau_0})|\leq n).
\]
We begin with the upper bound, which is the only direction we will actually use in the proofs of our main theorems. First note that \cref{lem:weakl1} implies that there exists a constant $C$ such that 
\[\P_v(\tau_0 \leq C n \log n \text{ and } |\LE(X^{\tau_0})|\leq n) \geq \frac{1}{2} \P_v(\tau_0 <\infty \text{ and } |\LE(X^{\tau_0})|\leq n)\]
for every $v\in \Z^4$ and $n\geq 2$, and hence that
\[
\E{|\fP_0(0,n)|} \leq 2 \sum_{v\in \Z^4} \sum_{m=0}^{\lceil C n \log n \rceil} \P_v(X_m =0 \text{ and } |\LE(X^m)|\leq n) 
\]
for every $n\geq 2$. Applying the mass-transport principle we obtain that
\begin{align*}
\E{|\fP_0(0,n)|} \leq 2 \sum_{v\in \Z^4} \sum_{m=0 }^{\lceil C n \log n \rceil} \P_0(X_m =v \text{ and } |\LE(X^m)|\leq n)
&=2 \sum_{m=0 }^{\lceil C n \log n \rceil} \P_0(|\LE(X^m)|\leq n)
\end{align*}
for every $n\geq 2$.
 To proceed define $m_1 = \lceil 2 n (\log n)^{1/3} \rceil$, $m_2 = \lceil n (\log n)^{8/9} \rceil$, and $m_3 = \lceil C n \log n \rceil$. (The number $8/9$ could be replaced by any number strictly between $1$ and $2/3$.) We break the above bound into pieces
\begin{align*}
\E{|\fP_0(0,n)|} &\lesssim \sum_{m=0 }^{m_1} \P_0(|\LE(X^m)|\leq n)
&&+
\sum_{m=m_1}^{m_2} \P_0(|\LE(X^m)|\leq n)
&&+ \sum_{m=m_2 }^{m_3} \P_0(|\LE(X^m)|\leq n).
\\
\intertext{We can use \cref{lem:amounterased} and \cref{cor:capgeom} with $\alpha=5/9$ and $\eps=1/9$ to bound the second and third terms respectively and obtain that}
\E{|\fP_0(0,n)|} &\lesssim
\hspace{1cm} n (\log n)^{1/3} &&+ \hspace{0.2cm} n (\log n)^{8/9} \cdot \frac{\log \log n}{(\log n)^{2/3}} 
&&+ \hspace{0.4cm} n \log n \cdot (\log n)^{-1}
\\&\lesssim\hspace{1cm} n (\log n)^{1/3}
\end{align*}
for every $n\geq 2$ as claimed.

We now turn to the lower bound. Since this bound is not actually needed for the proofs of our main theorems, we will omit some details. Let $M= \lceil n (\log n)^{1/3}/2 \rceil$. We can write
\begin{align*}
\E{|\fP_0(0,n)|} 
&\geq 
\sum_{v\in \Z^4} \sum_{m=0}^{M}  \P_v(\tau_0 = m \text{ and } |\LE(X^{m})|\leq n)
\\
&\geq 
\sum_{v\in \Z^4} \sum_{m=0}^{M}  \P_v(\tau_0 = m) - \sum_{v\in \Z^4} \sum_{m=0}^{M}\P_v(X_m=0 \text{ and } |\LE(X^m)| \geq n).
\end{align*}
For the first term, we have by time-reversal that
\begin{multline}
\label{eq:wusf0balllower1}
\sum_{v\in \Z^4} \sum_{m=0}^{M}  \P_v(\tau_0 = m) = 
\sum_{v\in \Z^4} \sum_{m=0}^{M}  \P_0(\tau_0^+ > m, X_m = v) \\\geq  \P_0(\tau^+_0 = \infty) n (\log n)^{1/3} \gtrsim n (\log n)^{1/3}
\end{multline}
as required.  On the other hand, for the second term, we have by the mass-transport principle that
\begin{align}
\sum_{v\in \Z^4} \sum_{m=0}^{M} 
\P_v(X_m=0 \text{ and } |\LE(X^m)| \geq n) &=
\sum_{v\in \Z^4} \sum_{m=0}^{M} 
\nonumber
\P_0(X_m=v \text{ and } |\LE(X^m)| \geq n) \\
&= \sum_{m=0}^{M} 
\P_0(|\LE(X^m)| \geq n).
\label{eq:wusf0balllower2}
 \end{align}
\cref{lem:amounterased} and \cref{lem:Lawlercuttimes} imply that this sum is $o(M)=o(n (\log n)^{1/3})$ as $n\to\infty$, and the claimed lower bound follows. Here is a brief sketch of how this estimate can be proven: for each $\eps>0$ and each $\eps M \leq m \leq M$, there is a cut time for the random walk between $m-m/(\log m)^6$ and $m$ with high probability. On this event we have that $|\LE(X^m)| \leq |\LE_\infty(X^m)|+m/(\log m)^6$, which is strictly less than $n$ with high probability by \cref{lem:amounterased}. The claim follows since $\eps>0$ was arbitrary. Once this estimate is in place, the claim follows by comparing the estimates \eqref{eq:wusf0balllower1} and \eqref{eq:wusf0balllower2}.
\end{proof}

\subsection{Upper bounds on the intrinsic radius}
\label{subsec:intrinsicupper}

In this section we prove the upper bound on the tail of the intrinsic diameter of the past for both the UST and $0$-WUSF.

\begin{prop}
\label{prop:unrootedupper}
 Let $\fP=\fP(0)$ be the past of the origin in the uniform spanning tree of $\Z^4$. Then 
\[
\P(\radint(\fP)\geq n) \lesssim \frac{(\log n)^{1/3}}{n}
\]
for every $n\geq 2$.
\end{prop}

\begin{prop}
\label{prop:rootedupper}
Let $\fP_0=\fP_0(0)$ be the past of the origin in the $0$-wired uniform spanning forest of $\Z^4$. Then 
\[
\P(\radint(\fP_0) \geq n) \lesssim \frac{(\log n)^{2/3} \log \log n}{n}
\]
for every $n\geq 3$.
\end{prop}

As discussed in the introduction, the proofs of both upper bounds follow the same basic strategy as the proofs of \cite[Theorems 1.2 and 7.1]{1804.04120}, but with many of the steps becoming much more delicate to implement. 
We begin by proving \cref{prop:rootedupper}, which is the 4d analogue of \cite[Lemma 7.7]{1804.04120} and will be used in the proof of \cref{prop:unrootedupper}. The proof of this estimate will rely on the results of \cref{sec:capacity} via the following technical lemma, whose proof is deferred until after we have completed the proof of \cref{prop:rootedupper}.

\begin{lemma}[Bad points]
\label{pro:o(n)}
Let $X$ be simple random walk on $\Z^4$ started from the origin.
 There exists a positive constant $C$ such that 
	\[
	\lim_{n\to\infty}\frac{1}{n}
	\E{\#\left\{t\leq A n\log n: |\lr{X^t}|\in [n,2n] \, \text{ and } \, \cpp{\lr{X^t}}\leq \frac{1}{C}\cdot \frac{n}{(\log n)^{2/3}}\right\}} =0
		\] 
		for every constant $A>0$.
\end{lemma}

\begin{proof}[Proof of \cref{prop:rootedupper} given \cref{pro:o(n)}]
	Let $Q_0(n)  = \pr{\radint(\fP_0) \geq n} = \pr{\partial\fP_0(0,n) \neq \emptyset}$ for each $n\geq 1$. It suffices to prove that
	\begin{align}
	\label{eq:Q00}
		Q_0(3n) \lesssim \frac{(\log n)^{2/3}\log \log n}{n} +  o(1) Q_0(n)
	\end{align}
	for every $n\geq 2$.
Indeed, this inequality implies by induction on $k$ that \[Q_0(3^k)\lesssim \frac{k^{2/3}\log k}{3^k}\]
for every $k\geq 2$, and the claim follows since $Q_0(n)$ is decreasing in $n$.

Fix $n\geq 2$, let $C_1$ be the constant from \cref{pro:o(n)} and let 
\[
\delta = \frac{2C_1 (\log n)^{2/3} \log \log n}{n}.
\]
	Let $\sI_0$ be the $0$-rooted interlacement process on $\Z^4$ and let $\F_{0,t}=\AB_t(\sI_0)$ for each $t\in \R$. For each $t\in \R$ and $u\in \Z^d$, we write $\fP_{0,t}(u,n)$ for the set of vertices that lie in the past of $u$ in $\F_{0,t}$ and have intrinsic distance at most $n$ from $u$, and write $\partial \fP_{0,t}(u,n)= \fP_{0,t}(u,n) \setminus \fP_{0,t}(u,n-1)$.
	Let $\sigma_0$ be the first time that $0$ is hit by a trajectory of the interlacement process, so that 
	\begin{align*}
		Q_0(3n) = \pr{ \partial \fP_{0,0}(0,3n) \neq \emptyset,\, \sigma_0 \leq \delta} + \pr{ \partial \fP_{0,0}(0,3n) \neq \emptyset,\,\sigma_0 > \delta}.
	\end{align*}
	The first of these terms is trivially bounded from above by $\pr{\sigma_0 \leq \delta}=O(\delta)$. Thus, to prove \eqref{eq:Q00}, it suffices to prove that
\begin{equation}
\label{eq:Q01}
\pr{ \partial \fP_{0,0}(0,3n) \neq \emptyset,\,\sigma_0 > \delta} \lesssim o(1) Q_0(n).
\end{equation}
Observe that on the event whose probability we wish to estimate there must be at least $n$ points in $\fP_{0,0}(0,2n)\setminus \fP_{0,0}(0,n)$ that lie on a geodesic from $0$ to $\partial \fP_{0,0}(0,3n)$,  so that Markov's inequality implies that
	\begin{multline}\label{eq:markovdiamint}
		\pr{ \partial \fP_{0,0}(0,3n) \neq \emptyset,\,\sigma_0 > \delta} \\
		 \leq \frac{1}{n} \sum_{u \in \Z^4} \pr{u\in \fP_{0,0}(0,2n)\setminus \fP_{0,0}(0,n), \partial\fP_{0,0}(u,n) \neq \emptyset, \sigma_0>\delta}.
	\end{multline}
	By time stationarity, we can shift time to $-\delta$ and ask that no trajectory hit $0$ during~$[-\delta,0]$. Moreover, on this event, the past of $0$ at time $-\delta$ is equal to the component of the past of $0$ at time $0$ induced by $\Z^4\setminus \mathcal{I}_{[-\delta,0]}$ by \cref{lem:vPastDynamics}. 
  Writing $\Gamma_{0,0}(x,y)$ for the unique path between $x$ and $y$ in $\fF_{0,0}$, we therefore have that
	\begin{align*}
		&\pr{ \partial \fP_{0,0}(0,3n)\neq \emptyset,\,\sigma_0 > \delta}
		\\&\hspace{1cm}\leq \frac{1}{n}\sum_{u\in \Z^4}\pr{u\in \fP_{0,0}(0,2n)\setminus \fP_{0,0}(0,n), \partial \fP_{0,0}(u,n)\neq \emptyset, \Gamma_{0,0}(u,0)\cap \mathcal{I}_{[-\delta,0]}=\emptyset}
		\\
		&\hspace{1cm}\leq 
		\frac{Q_0(n)}{n}\sum_{u\in \Z^4}\pr{u\in \fP_{0,0}(0,2n)\setminus \fP_{0,0}(0,n),\Gamma_{0,0}(u,0)\cap \mathcal{I}_{[-\delta,0]}=\emptyset},
	\end{align*}
	where we used the stochastic domination property (\cref{lem:domination}) in the last line.

	To prove \eqref{eq:Q01} it therefore suffices to prove that
	\begin{equation}
	\label{eq:badpoints1}
\lim_{n\to\infty}\frac{1}{n}\sum_{u\in \Z^4}\pr{u\in \fP_{0,0}(0,2n)\setminus \fP_{0,0}(0,n),\Gamma_{0,0}(u,0)\cap \mathcal{I}_{[-\delta,0]}=\emptyset} =0.
	\end{equation}
	Since this estimate will not use any dynamic arguments, we write $\fP_0=\fP_{0,0}$ and $\Gamma_0(x,y)=\Gamma_{0,0}(x,y)$ to lighten notation. We decompose each term into two terms according to whether $\cc{\Gamma_0(u,0)}$ is larger or smaller than~$\tfrac{n}{C_1(\log n)^{2/3}}$. Since $\cpc{}{\Gamma_0(u,0)}$ and $\cpc{0}{\Gamma_0(u,0)}$ coincide up to an additive constant and $\cI_{[-\delta,0]}$ is independent of $\fF_{0,0}$, we obtain in the first case that
	\begin{multline*}
		\pr{u\in \fP_0(0,2n)\setminus \fP_0(0,n),  \Gamma_0(u,0)\cap \mathcal{I}_{[-\delta,0]}=\emptyset,\cc{\Gamma_{0}(u,0)}\geq  \frac{n}{C_1(\log n)^{2/3}}} \\
\lesssim e^{-\delta n/C_1(\log n)^{2/3}}\pr{u\in \fP_0(0,2n)\setminus \fP_0(0,n) } = \frac{1}{(\log n)^2}\pr{u\in \fP_0(0,2n)\setminus \fP_0(0,n) },
	\end{multline*}
 so that summing over $u$ and applying \cref{prop:WUSFo_ball} we deduce that
	\begin{multline}
		\frac{1}{n}\sum_{u\in \Z^4}\pr{u\in \fP_0(0,2n)\setminus \fP_0(0,n), \Gamma_0(u,0)\cap \mathcal{I}_{[-\delta,0]}=\emptyset,\cc{\Gamma_0(u,0)}\geq  \frac{n}{C_1(\log n)^{2/3}}} \\
    \lesssim \frac{1}{n(\log n)^2} \E{|\fP_0(0,2n)|} \lesssim \frac{1}{(\log n)^{5/3}} =o(1)
\label{eq:Q02}
	\end{multline}
	as required.
 We now treat the second term, in which the capacity of the path connecting $u$ to the origin is small. That is, we bound
\begin{align*}
	\frac{1}{n}  \sum_{u\in \Z^4}  \pr{u\in \fP_0(0,2n)\setminus \fP_0(0,n), \cc{\Gamma_0(u,0)}\leq \frac{n}{C_1(\log n)^{2/3}}}.
\end{align*}
 Sampling the $0$-WUSF using Wilson's algorithm starting with $u$ and applying \cref{lem:weakl1}, we obtain that there exists a constant $C_3$ such that
\begin{align*}
	&\pr{u\in \fP_0(0,2n)\setminus \fP_0(0,n), \cc{\Gamma_0(u,0)}\leq \frac{n}{C(\log n)^{2/3}}} \\
	&= \prstart{\tau_0<\infty,\ n\leq |\ler{}{X^{\tau_0}} |\leq 2n, \cc{\ler{}{X^{\tau_0}}}\leq \frac{n}{C_1(\log n)^{2/3}}}{u}\\
	&\leq 2 \prstart{\tau_0 \leq C_3 n\log n,\ n\leq |\ler{}{X^{\tau_0}} |\leq 2n, \cc{\ler{}{X^{\tau_0}}}\leq \frac{n}{C_1(\log n)^{2/3}}}{u} \\
	&=2 \sum_{t=0}^{\lfloor C_3 n \log n \rfloor} \prstart{\tau^+_0 > t, X_t = u, \ n\leq |\LE(X^t) |\leq 2n, \cc{\LE(X^t)}\leq \frac{n}{C_1(\log n)^{2/3}}}{0},
\end{align*}
where we applied \cref{lem:weakl1} for the inequality and the reversibility of random walk and loop-erased random walk on $\Z^4$ for the last equality.
	Since $C_1$ was taken to be the constant from \cref{pro:o(n)}, we can now sum this estimate over $u \in \Z^4$ and apply that lemma to deduce that
	\begin{multline}
		 \frac{1}{n}\sum_{u\in \Z^4}  \pr{u\in \fP_0(0,2n)\setminus \fP_0(0,n), \cc{\Gamma_0(u,0)}\leq \frac{n}{C_1(\log n)^{2/3}}} \\
		 \leq \frac{2}{n}\E{\#\left\{t\leq C_3 n\log n: |\lr{X^t}|\in [n,2n] \, \text{ and } \, \cpp{\lr{X^t}}\leq \frac{1}{C_1}\cdot \frac{n}{(\log n)^{2/3}}\right\}}\\=o(1)
		 \label{eq:Q03}
	\end{multline}
as required. Putting together \eqref{eq:Q02} and \eqref{eq:Q03} yields \eqref{eq:badpoints1}, completing the proof.
\end{proof}

We now owe the reader the proof of \cref{pro:o(n)}.

\begin{proof}[Proof of \cref{pro:o(n)}]
Fix $A < \infty$ and let $C$ be the constant from \cref{prop:firstconcentration}. 
As in the proof of \cref{prop:WUSFo_ball}, we will break the sum we wish to bound into three pieces. Let $t_1=\lceil n (\log n)^{1/3} /2 \rceil$, $t_2 = \lceil n (\log n)^{8/9} \rceil$, and $t_3 = \lceil A n \log n \rceil$. (Here $8/9$ could safely be replaced by any number strictly between $5/6$ and $1$.) Then we have that
\begin{multline*}
\E{\#\left\{t\leq A n\log n: |\lr{X^t}|\in [n,2n] \, \text{ and } \, \cpp{\lr{X^t}}\leq \frac{1}{4C}\cdot \frac{n}{(\log n)^{2/3}}\right\}}
\\\leq 3+\sum_{t=3}^{t_1} \P\left(|\lr{X^t}| \geq n\right) + \sum_{t=t_1}^{t_2} \P\left(\cpp{\lr{X^t}}\leq \frac{1}{4C}\cdot \frac{n}{(\log n)^{2/3}} \right) \\+ 
\sum_{t=t_2}^{t_3} \P\left(\cpp{\lr{X^t}}\leq \frac{1}{4C}\cdot \frac{n}{(\log n)^{2/3}} \right).
\end{multline*}
For the first sum, Theorem~\ref{lem:amounterased} implies that 
\[
\sum_{t=3}^{t_1} \P\left(|\lr{X^t}| \geq n\right) \lesssim \sum_{t=3}^{t_1} \frac{\log \log t}{(\log t)^{2/3}} \lesssim \frac{t_1  \log \log t_1}{(\log t_1)^{2/3}} \lesssim \frac{n  \log \log n}{(\log n)^{1/3}} = o(n)
\]
as required. Similarly, for the second term, 
 \cref{prop:firstconcentration} with $\eps=1/18$ implies that 
\begin{align*}
	\sum_{t=t_1}^{t_2}	\pr{\cpp{\ler{}{X^t}}\leq \frac{1}{4C}\cdot \frac{n}{(\log n)^{2/3}}}
	&\lesssim \sum_{t=t_1}^{t_2}	\pr{\cpp{\ler{}{X^t}}\leq \frac{1}{C}\cdot \frac{t}{\log t}}\\
	 &\lesssim \sum_{t=t_1}^{t_2} \frac{1}{(\log t)^{17/18}} \lesssim \frac{n(\log n)^{16/18}}{(\log n)^{17/18}} =o(n)
\end{align*}
as required.
Finally, for the third sum, applying \cref{cor:capgeom} with $\alpha=5/9$ and $\eps= 1/18$ yields that
\begin{multline*}
\sum_{t=t_2}^{t_3}  \prstart{\cpp{\ler{}{X^t}}\leq \frac{1}{4C}\cdot \frac{n}{(\log n)^{2/3}}}{0} 
\\\lesssim \sum_{t=t_2}^{t_3}  \prstart{\cpp{\ler{}{X^t}}\leq \frac{1}{2C}\cdot \frac{t}{(\log t)^{1+5/9}}}{0}
 \lesssim \frac{A n \log n }{(\log n)^{19/18}} = o(n)
\end{multline*}
as required. \qedhere
\end{proof}

\begin{remark}
Using \cref{prop:firstconcentration} but not the more difficult estimate of \cref{cor:capgeom} in this proof leads to a proof that the expectation in question is $O(n (\log n)^\eps)$ for every $\eps>0$. Note however that this does \emph{not} suffice to prove that the tail of the intrinsic radius is, say, $O(n^{-1}(\log n)^{1/3+\eps})$ via our strategy. Indeed, any upper bound on this expectation that is larger than $n$ is completely useless as an input to our inductive scheme above.
\end{remark}

Our next goal is to apply \cref{prop:rootedupper} to prove \cref{prop:unrootedupper}.  We define
\[Q(n)  = \pr{\radint(\fP) \geq n} = \pr{\partial \fP(0,n) \neq \emptyset}\] for each $n\geq 0$.
Let $\sI$ be the interlacement process on $\Z^4$ and let $\fT=\AB_0(\sI)$ be the associated sample of the uniform spanning tree at time zero.
Let $\sigma_0$ be the first time that $0$ is hit by a trajectory of the interlacement process.  As above, we will use the decomposition 
	\begin{align}
		Q(2n) = \pr{\partial\fP(0,2n) \neq \emptyset, \sigma_0> \delta} + \pr{\partial\fP(0,2n) \neq \emptyset, \sigma_0 \leq  \delta},
	\end{align}
	for some appropriately chosen $\delta>0$. 	 Unlike in our analysis of the $0$-WUSF, however, bounding the second term by $\pr{\sigma_0\leq  \delta} = O(\delta)$ is no longer sharp, and a more delicate analysis will be necessary.

\begin{lemma}
\label{lem:exceptional_time}
 The bound
\[\pr{\partial\fP(0,n) \neq \emptyset \mid \sigma_0 \leq  \delta} \lesssim \frac{1}{(\log n)^{1/3}}\]
holds for every $n\geq 2$ and $0<\delta \leq (\log n)^{-1/3}$.
\end{lemma}	

The condition that $\delta \leq (\log n)^{-1/3}$ is not really necessary, but is included to simplify the proof. We will apply the lemma with $\delta$ of order $n^{-1}(\log n)^{2/3} \ll (\log n)^{-1/3}$.


\begin{proof}[Proof of \cref{lem:exceptional_time}]
We first describe a coupling of the conditional law of $\sI$ given $\sigma_0 \leq \delta$ with a slightly simpler model.
Let $W$ be a doubly-infinite random walk started at $0$ and conditioned not to return to $0$ at any positive time, independent of the interlacement process $\sI$. (Note that in the definition of the interlacements process, the trajectory $W$ is conditioned to hit $0$ for the first time at time $0$. However, conditioning~$W$ to hit $0$ for the \emph{last} time at time zero gives exactly the same tree by the definition of the Aldous Broder algorithm, which only depends on the trajectories modulo time shifts.)
Fix $n \geq 2$ and let $\delta_0=(\log n)^{-1/3} \lesssim 1$. 
For each $\delta \geq 0$, let $\sI^\delta$ be the point process $\sI^\delta=\sI \cup \{(W,\delta)\}$, let $\fT^\delta=\AB_0(\sI^\delta)$, and let $\partial \fP^\delta(0,n)$ be the set of points in the past of $0$ in $\fT^\delta$ having intrinsic distance exactly $n$ from $0$.
By the standard theory of Poisson point processes, the conditional distribution of $\sI$ given $\sigma_0 = \delta$ (in the Palm probability sense) is equal to the conditional distribution of $\sI^\delta$ given that $0\notin \cI_{[0,\delta]}$, i.e., that $0$ is not hit by $\sI$ in time $[0,\delta]$. 
Thus, we have that
\begin{align*}
\pr{\partial\fP(0,n) \neq \emptyset \mid \sigma_0 \leq  \delta} &\leq \sup_{0\leq \eps \leq \delta} \pr{\partial\fP^\eps(0,n) \neq \emptyset \mid 0 \notin \cI_{[0,\eps]}}
\lesssim 
\sup_{0\leq \eps \leq \delta} \pr{\partial\fP^\eps(0,n) \neq \emptyset}
\\ &\leq \sup_{0\leq \eps \leq \delta} \pr{\partial\fP^\eps(0,n) \neq \emptyset,\, 0 \notin \cI_{[0,\delta_0]}} + \pr{0 \in \cI_{[0,\delta_0]}}
\end{align*}
for every $0<\delta \leq \delta_0$.
Since the second term is $O(\delta_0)=O((\log n)^{-1/3})$ by definition of the interlacement process, it suffices to prove that
%
\begin{equation}
\P(\partial \fP^\delta(0,n) \neq \emptyset,\, 0 \notin \cI_{[0,\delta_0]} ) \lesssim \frac{1}{(\log n)^{1/3}}
\label{eq:notzeroyet}
\end{equation}
for every $0\leq \delta \leq \delta_0$. 

We now claim that the past $\fP^\delta=\fP^\delta(0)$ of $0$ in $\fT^\delta$ is a decreasing function of $0\leq \delta \leq \delta_0$ on the event $0\notin \cI_{[0,\delta_0]}$ in the sense that $\fP^\delta(0)$ is a subgraph of $\fP^\eps(0)$ whenever $0\leq \eps \leq \delta$. 
The proof of this claim is similar to that of \cite[Lemma 3.4]{1804.04120} (our \cref{lem:PastDynamics}). For each $x\in \Z^4$ and $0 \leq \delta \leq \delta_0$, let $\sigma^\delta_0(x)$ be the first time after time zero that $x$ is hit by a trajectory of $\sI^\delta$. Suppose $x$ belongs to the past of $0$ in $\fT^\delta$, and let $x_1,\ldots,x_n$ be the path connecting $x$ to $0$ in $\fT^\delta$. On the event $\{0\notin \cI_{[0,\delta_0]}\}$, it follows by definition of the Aldous--Broder algorithm that the sequence $(\sigma^\delta_0(x_i))_{i=1}^n$ is decreasing and that $\sigma^\delta_0(x_n)=\sigma^\delta_0(0)=\delta$. In particular, every vertex in this path  is either hit for the first time in $\sI^\delta_{[0,\infty)}$ by the trajectory $W$ or by a trajectory of $\sI$ that arrives after time $\delta$. It follows that the first entry edge for each vertex in this path is the same for the two point processes $\sI^\delta$ and $\sI^\eps$, and hence that the path is also contained in $\fP^\eps(0)$ as required. Thanks to this monotonicity, it suffices for us to prove \eqref{eq:notzeroyet} in the case $\delta=0$, i.e., that
\begin{equation}
\label{eq:zero_only}
\P(\partial \fP^0(0,n) \neq \emptyset,\, 0 \notin \cI_{[0,\delta_0]} ) \leq \P(\partial \fP^0(0,n) \neq \emptyset ) \lesssim \frac{1}{(\log n)^{1/3}}
\end{equation}
for every $n\geq 2$.

Let $W^+=W[0,\infty)$ and $W^-=(W(-\infty,0])^\leftarrow$ be the forward and backward parts of the doubly-infinite walk $W$, so that $W^+$ is distributed as a random walk conditioned not to return to the origin and $W^-$ is a simple random walk independent of $W^+$.
As discussed in \cref{subsec:AB_variant}, the oriented tree 
$\AB(W^+)$ contains a unique, backwards-oriented infinite path $Z=\LE^R(W^+)$ that is distributed as an infinite loop-erased random walk started at $0$. Again, be careful to note that $Z$ is \emph{not} the loop-erasure of $W^+$, but is rather a sort of infinite backwards loop-erasure. Indeed, for each $k\geq 0$, the path connecting $W^+_k$ to $0$ in $\AB(W^+)$ is equal to the loop-erasure of the \emph{reversal} of  $(W^+)^k$.
 We call the infinite path $Z$ the \textbf{spine} of $\AB(W^+)$. One way for the past $\fP^0$ to be large is for the walk $W^-(0,\infty)$ to avoid a long initial segment of the spine $Z$; we will argue that this is the \emph{only} way for the past $\fP^0$ to be large up to events of negligibly small probability.

   Let $m_1 = \lceil n^{1/32} \rceil$ and $m_2 = \lceil n^{1/16} \rceil$, so  
    that $\log m_1 \asymp \log m_2 \asymp \log n$. Let $A_1$ be the event that $W^-(0,\infty) \cap Z^{m_1} \neq \emptyset$ and let $A_2$ be the event that
$W^-(0,\infty)$ intersects the path from $W^+_m$ to $0$ in $\AB(W^+)$ for every $m \geq 2m_2$. Then we have by a union bound that 
\begin{align*}
	\pr{\partial \fP^0(0,n)\neq \emptyset} &\leq 
	\pr{A_1^c} + \pr{A_1 \setminus A_2}
	 + \pr{\partial \fP^0(0,n)\neq \emptyset, A_2}.
\end{align*}
Since $Z$ is distributed as a loop-erased random walk, the estimate \eqref{eq:nonintersection_n_outside} of \cref{lem:displacement} implies that $\pr{A_1^c} \asymp (\log n)^{-1/3}$, and
 it therefore suffices to prove that the second two terms above are both $o((\log n)^{-1/3})$; we will see that both estimates hold with a lot of room.

We begin by bounding $\P(A_1 \setminus A_2)$. For each $k \geq 0$, let $\ell_k^\leftarrow$ be maximal such that $W^+_{\ell_k^\leftarrow}=Z_k$. A similar, time-reversed version of the argument used to prove \cite[Lemma 5.1]{1804.04120} and our Lemma~\ref{lem:weakl1} yields that the increments $(\ell_k^\leftarrow-\ell_{k-1}^\leftarrow)_{k\geq 1}$ are conditionally independent given $Z$ and satisfy
\[
\P(\ell_k^\leftarrow-\ell_{k-1}^\leftarrow = t+1 \mid Z) \leq p_t(0,0) \lesssim \frac{1}{t^2}
\]
for every $t,k\geq 1$. Applying the weak triangle inequality for weak-$L^1$ random variables \eqref{eq:weakl1triangle} we deduce that
\begin{equation}
\label{eq:backwardsweakl1}
\P(\ell_{m_1}^\leftarrow \geq m_2 ) \lesssim \frac{m_1 \log m_1}{m_2}.
\end{equation}
 Now observe that if $r \geq 0$ is a cut-time of $W^+$ then $W^+_r=Z_k$ for some $k\geq 0$ and the path connecting $0$ to $W^+_m$ in $\AB(W^+)$ has $Z^k$ as an initial segment for every $m \geq r$. As such, if $\ell_{m_1}^\leftarrow \leq m_2$ and there is a cut-time between $m_2$ and $2m_2$ then $Z^{m_1}$ is an initial segment of the path connecting $0$ and $W^+_m$ in $\AB(W^+)$ for every $m\geq 2m_2$. So in this case, if $W^-(0,\infty)$ intersects~$Z^{m_1}$, then the event $A_2$ must also be satisfied. 
 Applying \eqref{eq:backwardsweakl1} and \cref{lem:Lawlercuttimes}, we deduce that
 \begin{multline}
\P(A_1 \setminus A_2) \leq \P(\ell_{m_1}^\leftarrow \geq m_2 ) + \P(\text{there is no cut-time between $m_2$ and $2m_2$})
\\
\lesssim \frac{m_1 \log m_1}{m_2} + \frac{\log \log m_2}{\log m_2} = o\left(\frac{1}{(\log n)^{1/3}}\right)
 \end{multline}
 as required.

It remains only to bound the probability that
$\partial \fP^0(0,n)\neq \emptyset$ and $A_2$ holds.
The definition of the Aldous--Broder algorithm ensures that $W^+[2m_2,\infty)$ is disjoint from $\fP^0(0)$ on the event $A_2$. As such, if $A_2$ holds and $\partial \fP^0(0,n)\neq \emptyset$ then there must exist $0\leq m \leq 2m_2$ such that $\fP^0(W_m)$ contains a path of length at least $\bar n := n-2m_2$ that is disjoint from $W(m,\infty)$. In particular, this $m$ satisfies $W_m \neq W_{r}$ for every $r > m$. 
Letting $\radint(\fP^0(W_m) \setminus W(m,\infty))$ be the maximal intrinsic distance between $W_m$ and a point in the connected component of $W_m$ in $\fP^0(W_m) \setminus W(m,\infty)$, it follows by a union bound that
\begin{equation*}
\P(\partial \fP^0(0,n)\neq \emptyset, A_2 ) 
\leq \sum_{m=0}^{2m_2}  
\pr{\text{$\radint\left(\fP^0(W_m) \setminus W(m,\infty)\right)\geq \bar n$, $W_m \notin W(m,\infty)$}}.
\end{equation*}
We now claim that
\begin{multline}
\label{eq:index_shifting}
\pr{\text{$\radint\left(\fP^0(W_m) \setminus W(m,\infty)\right)\geq \bar n$, $W_m \notin W(m,\infty)$}} 
\\
=\pr{\text{$\radint\left(\fP^0(W_0) \setminus W(0,\infty)\right)\geq \bar n$},\, W_{-m} \notin W(-m,\infty)}
\end{multline}
for each $m\geq 0$.
To this end, let $\widetilde \P$ be the law of the pair of random variables $(\sI,W)$ defined as above but where $W$ is an unconditioned simple bi-infinite random walk. That is, the law of the pair $(\sI,W)$ as we introduced them previously can be obtained from this modified law $\tilde \P$ by conditioning that $W_0 \notin W(0,\infty)$. This modified model is now stationary under shifting the index of the random walk $W$, so that
\begin{align*}
&\pr{\text{$\radint\left(\fP^0(W_m) \setminus W(m,\infty)\right)\geq \bar n$, $W_m \notin W(m,\infty)$}}
\\
&\hspace{2.5cm}=
\widetilde \P\left(\text{$\radint\left(\fP^0(W_m) \setminus W(m,\infty)\right)\geq \bar n$, $W_m \notin W(m,\infty)$} \;\Big|\; W_0 \notin W(0,\infty)\right)
\\
&\hspace{2.5cm}=
\widetilde \P\left(\text{$\radint\left(\fP^0(W_0) \setminus W(0,\infty)\right)\geq \bar n$, $W_0 \notin W(0,\infty)$} \;\Big|\; W_{-m} \notin W(-m,\infty)\right)\\
&\hspace{2.5cm}=
\widetilde \P\left(\text{$\radint\left(\fP^0(W_0) \setminus W(0,\infty)\right)\geq \bar n$, $W_{-m} \notin W(-m,\infty)$} \;\Big|\; W_{0} \notin W(0,\infty)\right)\\
&\hspace{2.5cm}=
 \P\left(\text{$\radint\left(\fP^0(W_0) \setminus W(0,\infty)\right)\geq \bar n$, $W_{-m} \notin W(-m,\infty)$} \right)
\end{align*}
as claimed, where we used the fact that the two events $\{W_0 \notin W(0,\infty)\}$ and $\{W_{-m} \notin W(-m,\infty)\}$ have the same probability under the measure $\tilde \P$ in the third equality.
We deduce by summing over $m$ in \eqref{eq:index_shifting} that
\begin{equation}
\label{eq:m2union}
\P\left(\partial \fP^0(0,n)\neq \emptyset, A_2 \right) 
\lesssim m_2   
\pr{\text{$\radint(\fP^0(0) \setminus W(0,\infty)) \geq \bar n$}}.
\end{equation}

Now, if we define $\sI^- = \sI \cup \{(W^-,0)\}$ to be the union of $\sI$ with the negative part of $W$, define $\fT^-=\AB(\sI^-)$, and define $\fT'$ by flipping the orientations on the unique infinite path in $\fT^-$ emanating from $0$, the resulting oriented tree $\fT'$ has the distribution of the oriented UST as discussed in~\cref{subsec:AB_variant}. Let $\kappa$ be the last time $W^-$ visits the origin, and observe that if $x$ belongs to 
the connected component of $0$ in $\fP^0(0)\setminus W(0,\infty)$ then $x$ must belong to the past of at least one of the vertices of $W[-\kappa,0]$ in $\fT'$. Indeed, if we let $x=x_0,\ldots,x_n=0$ be the vertices of the path connecting $x$ to $0$, we have by definition of the Aldous--Broder algorithm that 
$x_i \notin W(-\infty,-\kappa)$ for every $0\leq i \leq n$, and moreover that
there exists $0\leq k \leq n$ such that $x_i \notin W(-\infty,\infty)$ for every $0 \leq i < k$ and $x_i \in W[-\kappa,0]$ for every $k \leq i \leq n$. It follows easily from the definitions that $x$ belongs to the past of $x_k \in W[-\kappa,0]$ in $\fT'$ as claimed. Moreover, similar considerations imply that if $x$ has intrinsic distance $\bar n$ from $0$ then there exists $\kappa \leq i \leq 0$ such that $x$ lies in the past of $W_i$ in $\fT'$ and has intrinsic distance at least $\bar n - \kappa$ from $W_i$ in $\fT'$.
Since $\fT'$ is distributed as the uniform spanning tree of $\Z^4$, we obtain the crude bound
\begin{align*}
\pr{\text{$\radint(\fP^0(0) \setminus W(0,\infty)) \geq \bar n$}} &\leq \P(\kappa \geq n^{1/8}) + \P\left(\exists x\in \Lambda_{\lceil n^{1/8}\rceil} \text{ with } \partial \fP(x,\bar n - \lfloor n^{1/8} \rfloor)\neq \emptyset\right)
\\&\lesssim n^{-1/8} + n^{1/2} Q_0\left(\bar n - \lfloor n^{1/8} \rfloor\right) \lesssim n^{-1/8}
\end{align*}
 where we used the stochastic domination property (\cref{lem:domination}) in the second inequality and \cref{prop:rootedupper} in the third, noting that $\bar n - \lfloor n^{1/8} \rfloor \sim n$ as $n\to\infty$. Substituting this estimate into \eqref{eq:m2union} yields that
 \[
\P\left(\partial \fP^0(0,n)\neq \emptyset, A_2 \right)  \lesssim n^{1/16} \cdot n^{-1/8} = o\left( \frac{1}{(\log n)^{1/3}}\right)
 \]
 as required. \qedhere
\end{proof}

We are now ready to complete the proof of \cref{prop:unrootedupper}.

\begin{proof}[Proof of \cref{prop:unrootedupper}]
	Let $Q(n)  =  \pr{\radint(\fP)\geq n} = \pr{\partial \fP(0,n) \neq \emptyset}$ for each $n\geq 0$. By a similar inductive argument to that used in the proof of \cref{prop:unrootedupper}, it suffices to prove that there exists a constant $C$ such that
	\begin{equation}
	\label{eq:unrooted_induction}
	Q(2n) \leq \frac{C (\log n)^{1/3}}{n} + \frac{1}{4} Q(n)
	\end{equation}
	for every $n\geq 2$. (Here $1/4$ could safely be replaced by any number strictly smaller than $1/2$.) Let $\sI$ be the interlacement process on $\Z^4$ and let $\fT_{t}=\AB_t(\sI)$ for each $t\in \R$. For each $t\in \R$ and $u\in \Z^d$, we write $\fP_{t}(u,n)$ for the set of vertices that lie in the past of $u$ in $\fT_{t}$ and have intrinsic distance at most $n$ from $u$, and write $\partial \fP_{t}(u,n)= \fP_{t}(u,n) \setminus \fP_{t}(u,n-1)$. (We note that the notation $\fP_0$ inside this proof will always refer to this dynamic definition with $t=0$ and it should not be confused with the past of $0$ in the $0$-WUSF.)
	Let $C_1$ be the constant from \cref{prop:firstconcentration}.
	Fix $n\geq 1$ and let 
	\[
	\delta = \frac{C_1 (\log n)^{2/3}}{n} \cdot \log 4.
	\] 
	Let $\sigma_0$ be the first time that $0$ is hit by a trajectory of the interlacement process.  
  We consider the decomposition
	\begin{align}
		Q(2n) = \pr{\partial\fP_0(0,2n) \neq \emptyset, \sigma_0> \delta} + \pr{\partial\fP_0(0,2n) \neq \emptyset, \sigma_0 \leq  \delta}.
	\end{align}
	\cref{lem:exceptional_time} yields the bound
	\[
	\pr{\partial\fP_0(0,2n) \neq \emptyset, \sigma_0 \leq  \delta} \lesssim \frac{(\log n)^{1/3}}{n},
	\]
	so that it suffices to prove that
	\begin{align}
\pr{\partial\fP_0(0,2n) \neq \emptyset, \sigma_0> \delta} &\leq \frac{1}{4}Q(n)+o\left(\frac{(\log n)^{1/3}}{n}\right) 
\label{eq:unrootedinduction1}
	\end{align}
	for every $n\geq 1$.
	To this end, we consider the union bound
	\begin{align*}
		 \pr{\partial\fP_0(0,2n) \neq \emptyset, \sigma_0> \delta} \leq  \sum_{u} \pr{u\in \partial\fP_0(0,n), \partial \fP_0(u,n) \neq \emptyset, \sigma_0> \delta}.
	\end{align*}
	By time stationarity, we can shift time to~$-\delta$ and ask that no trajectory hit $0$ during~$[-\delta,0]$. On this event we have by \cref{lem:PastDynamics} that the past of $0$ at time $-\delta$ is the same as the component of the subgraph of the past of $0$ at time $0$ induced by $\Z^4\setminus \mathcal{I}_{[-\delta,0]}$, where $\mathcal{I}_{[-\delta,0]}$ denotes the set of vertices visited by the interlacement process between times $-\delta$ and $0$. Writing $\Gamma_0(x,y)$ for the unique path between $x$ and $y$ in $\fT_0$, we deduce that 
	\begin{multline*}
		\pr{u\in \partial\fP_0(0,n), \partial \fP_0(u,n) \neq \emptyset, \sigma_0> \delta}\\
		\leq \pr{u\in \partial\fP_0(0,n), \partial \fP_0(u,n) \neq \emptyset, \Gamma_0(u,0)\cap \mathcal{I}_{[-\delta,0]}=\emptyset}. 
	\end{multline*}
	We decompose the above probability according to whether $\cc{\Gamma_0(u,0)}$ is larger or smaller than $ \tfrac{n}{C_1(\log n)^{2/3}}$. In the first case we have by definition of the interlacement intensity measure and the fact that $\cI_{[-\delta,0]}$ is independent of $\fT_0$ that
	\begin{multline*}
		\pr{u\in \partial\fP_0(0,n), \partial \fP_0(u,n) \neq \emptyset, \Gamma_0(u,0)\cap \mathcal{I}_{[-\delta,0]}=\emptyset,\cc{\Gamma_0(u,0)}\geq  \frac{n}{C_1(\log n)^{2/3}}} \\
		\leq e^{-\delta n/(C_1(\log n)^{2/3})}\pr{u\in \partial\fP_0(0,n), \partial \fP_0(u,n) \neq \emptyset} = \frac{1}{4} \pr{u\in \partial\fP_0(0,n), \partial \fP_0(u,n) \neq \emptyset},
	\end{multline*}
    where $\delta$ was chosen so that the final equality holds.
	The mass-transport principle implies that
	\begin{multline*}
		\sum_{u\in\Z^4} \pr{u\in \partial\fP_0(0,n), \partial \fP_0(u,n) \neq \emptyset} \\=
		\sum_{u\in \Z^4} \pr{u \text{ is in the future of $0$ at distance $n$ and } \partial \fP_0(0,n) \neq \emptyset }= Q(n),
	\end{multline*}
	where the last equality follows from the fact that there is a unique vertex at distance $n$ in the future of $0$.
	The term when the capacity is large can therefore be bounded
	\begin{multline}
	\label{eq:largecapterm}
	\sum_{u\in \Z^4} \pr{u\in \partial\fP_0(0,n), \partial \fP_0(u,n) \neq \emptyset, \Gamma_0(u,0)\cap \mathcal{I}_{[-\delta,0]}=\emptyset,\cc{\Gamma_0(u,0)}\geq  \frac{n}{C_1(\log n)^{2/3}}}\\ \leq  \frac{1}{4} Q(n).
	\end{multline}
	We next look at the term when the capacity is small, which for each $u\in \Z^4$ can be  upper bounded using Lemma~\ref{lem:domination} by
	\begin{multline*}
		\pr{u\in \partial\fP_0(0,n), \partial \fP_0(u,n) \neq \emptyset, \cc{\Gamma_0(u,0)}\leq  C \frac{n}{(\log n)^{2/3}}} \\\leq 
		Q_0(n)\pr{u\in \partial\fP_0(0,n),\cc{\Gamma_0(u,0)}\leq  \frac{n}{C(\log n)^{2/3}}},
	\end{multline*}
	where, as in the proof of \cref{prop:rootedupper}, $Q_0(n)$ is the probability that the component of the origin in the $0$-WUSF has intrinsic radius at least $n$.
	A similar mass-transport argument to above gives that 
\begin{multline*}
		\sum_{u\in \Z^4} \pr{u\in \partial\fP_0(0,n), \partial \fP_0(u,n) \neq \emptyset, \cc{\Gamma_0(u,0)}\leq  \frac{n}{C(\log n)^{2/3}}} \\\leq 
		Q_0(n) \prstart{\cc{\LE(X)^n}\leq   \frac{n}{C(\log n)^{2/3}}}{0},
	\end{multline*}
	where $X$ is a simple random walk, and applying \cref{prop:rootedupper} and \cref{prop:firstconcentration} yields that
\begin{multline}\sum_{u\in \Z^4} \pr{u\in \partial\fP_0(0,n), \partial \fP_0(u,n) \neq \emptyset, \cc{\Gamma_0(u,0)}\leq   \frac{n}{C(\log n)^{2/3}}}
	\\\lesssim \frac{(\log n)^{2/3} \log \log n}{n} \cdot \frac{\log\log n}{(\log n)^{2/3}} =
	\frac{(\log \log n)^2}{n} = o\left(\frac{(\log n)^{1/3}}{n}\right).
	\label{eq:smallcapterm}
\end{multline}
The claimed inequality \eqref{eq:unrootedinduction1} follows by summing the estimates \eqref{eq:largecapterm} and \eqref{eq:smallcapterm}. \qedhere




\end{proof}


\section{Connections inside a box and the volume of the past}
\label{sec:volume}

In this section we study the geometry of the restriction of the 4d UST to a box and use the resulting estimates together with our results on the tail of the intrinsic diameter of the past to 
study the tail of the \emph{volume} of the past. 
We begin by proving the relevant upper bounds on the tail of the volume, which are straightforward consequences of \cref{prop:unrootedupper,prop:rootedupper}.

\begin{prop}
\label{prop:volume_upper}
Let $\fP=\fP(0)$ and $\fP_0=\fP_0(0)$ be the past of the origin in the uniform spanning tree and $0$-wired uniform spanning forest of $\Z^4$ respectively. Then
\begin{align}
\P(|\fP| \geq n) &\lesssim \frac{(\log n)^{1/6}}{n^{1/2}}  &\text{ and}
\label{eq:volume_unrooted_upper}
\\
\P(|\fP_0| \geq n) &\lesssim \frac{(\log n)^{1/2}\log \log n}{n^{1/2}}&
\label{eq:volume_rooted_upper}
\end{align}
for every $n \geq 3$.
\end{prop}

\begin{proof}
We first prove \eqref{eq:volume_unrooted_upper}. We have by a union bound and Markov's inequality that
\[
\P(|\fP(0)| \geq n) 
\leq \P(\partial \fP(0,r) \neq \emptyset) + \P(|\fP(0,r)|\geq n)
\leq \P(\partial \fP(0,r) \neq \emptyset) + \frac{1}{n} \E{|\fP(0,r)|}
\]
for every $n,r \geq 1$. Applying \cref{prop:unrootedupper,lemma:WUSF_ball} yields that
\[
\P(|\fP(0)| \geq n) 
\lesssim \frac{(\log r)^{1/3}}{r} + \frac{r}{n}
\]
for every $n,r \geq 2$, and the claim follows by taking $r= \lceil n^{1/2} (\log n)^{1/6}\rceil$.
%
For \eqref{eq:volume_rooted_upper}, arguing in a similar way but using \cref{prop:rootedupper,prop:WUSFo_ball} in place of \cref{prop:unrootedupper,lemma:WUSF_ball} yields that
\[
\P(|\fP_0(0)| \geq n) 
\leq \P(\partial \fP_0(0,r) \neq \emptyset) + \frac{1}{n} \E{|\fP_0(0,r)|}
\lesssim \frac{(\log r)^{2/3} \log \log r}{r} + \frac{r (\log r)^{1/3}}{n}
\]
for every $n,r\geq 3$, and the claim follows by taking $r=\lceil n^{1/2} (\log n)^{1/6}\rceil$ as before.
\end{proof}



In the remainder of the section, we prove a matching lower bound for \eqref{eq:volume_unrooted_upper}. We expect that a lower bound matching \eqref{eq:volume_rooted_upper} up to subpolylogarithmic terms can be proven using similar methods but do not pursue this here.

\begin{prop}
\label{prop:volume_lower}
Let $\fP=\fP(0)$ be the past of the origin in the uniform spanning tree  of $\Z^4$. Then
\begin{align}
\P(|\fP| \geq n) &\gtrsim \frac{(\log n)^{1/6}}{n^{1/2}}
\label{eq:volume_unrooted_lower}
\end{align}
for every $n \geq 2$.
\end{prop}

The most natural way to prove this lower bound would be to first prove the second moment estimate
\[
\E{|\fP(0,r)|^2} \lesssim \frac{r^3}{(\log r)^{1/3}}
\]
for every $r\geq 2$.
Indeed, letting $Z_r = |\fP(0,2r) \setminus \fP(0,r)|$, one would then be able to use \cref{lemma:WUSF_ball}, \cref{prop:unrootedupper}, and the Paley-Zygmund inequality to deduce that that there exists a constant $c$ such that
\[
\P\left(|\fP(0)| \geq \frac{c r^2}{(\log r)^{1/3}}\right) \geq \P\left(Z_r \geq \frac{1}{2} \E{Z_r \mid Z_r >0} \right) \geq \frac{\E{Z_r}^2}{4\E{ Z_r^2}} \gtrsim \frac{(\log r)^{1/3}}{r}
\]
for every $r\geq 2$, so that the claim would follow by taking $r= \lceil n^{1/2} (\log n)^{1/6}\rceil$. Unfortunately, estimating the second moment of $|\fP(0,r)|$ seems to be a rather technical matter. To get around this, we will instead perform a similar calculation using the number of points that lie in both an intrinsic annulus and an extrinsic annulus of the appropriate scale.





To proceed further, we study the local geometry of the 4d UST inside a box. The resulting estimates are also of independent interest. Recall that we write $\Lambda_r = [-r,r]^4 \cap \Z^4$ for the box of radius $r$ in $\Z^4$. For each $x,y \in \Z^4$ we write $\Gamma(x,y)$ for the path connecting them in the uniform spanning tree $\fT$. Our next main goal is to estimate the number of points that are connected to the origin inside the box of radius $r$.

\begin{prop}
\label{prop:extrinsic_volume}
Let $\fT$ be the uniform spanning tree of $\Z^4$. Then
\[
\E{ |\{x \in \Lambda_r : \Gamma(0,x) \subseteq \Lambda_r \}|} \asymp \frac{r^4}{\log r}
\]
for every $r\geq 2$. 
\end{prop}

The upper bound of \cref{prop:extrinsic_volume} follows as an immediate corollary of the following estimate concerning intersections of random walks, which is a space-parameterised version of Lawler's time-parameterised estimate~\cite[Theorem~3.3.2]{Law91} and follows by a similar proof.  We defer the proof of this theorem to Section~\ref{sec:spaceintersections}.

\begin{theorem}\label{thm:sumofinters}
	Let $r\geq 1$ and let $X$ and $Y$ be two independent simple random walks in $\Z^4$, where~$X$ starts at a uniform random point $X_0$ of $\Lambda_r$ and $Y$ starts at the origin. Then 
	\[
	\pr{X \cap Y \cap \Lambda_r \neq \emptyset} \asymp \frac{1}{\log r}.
	\]
\end{theorem}


We next derive some complementary estimates that will be used to prove both \cref{prop:volume_lower} and the lower bound of \cref{prop:extrinsic_volume}. The next lemma gives a rather general relationship between the first moment of the volume of an intrinsic ball in the entire forest and the second moment of the volume of an intrinsic ball in the past.

\begin{lemma} 
\label{lem:past_total_comparison}
Let $d\geq 1$ and let $\F$ be the uniform spanning forest of $\Z^d$. For each $v\in \Z^d$ let $\fT(v)$ be the component of $v$ and let $\fP(v)$ be the past of $v$. Then
\begin{multline*}
\sum_{v\in \Lambda_r}\E{|\{x \in \fP(v) : \Gamma(v,x) \subseteq \Lambda_r \text{ and } |\Gamma(v,x)| \leq n\}|^2} \\\leq (n+1) 
\sum_{v\in \Lambda_r}\mathbb{E} |\{x \in \fT(v) : \Gamma(v,x) \subseteq \Lambda_{r} \text{ and } |\Gamma(v,x)| \leq 2n\}| 
\end{multline*}
for every $r,n\geq 1$.
\end{lemma}

\begin{proof}
Fix $r,n \geq 1$. For each $v,x,y,z\in \Lambda_r$ let $A(v,x,y,z)$ be the event that $x$ and $y$ both lie in the past of $v$, that the paths connecting $x$ and $y$ to $v$ are both contained in $\Lambda_r$, and that these two paths meet for the first time at $z$. 
Moreover, for each $v,x,y,z\in \Lambda_r$ and $a,b,c \geq 0$ let $A(v,x,y,z;a,b,c)$ be the event that $A(v,x,y,z)$ holds and that the paths connecting $x$ to $z$, $y$ to $z$, and $z$ to $v$ have lengths $a$, $b$, and $c$ respectively. 
Then we can write
\begin{multline}
\sum_{v\in \Lambda_r} \E{|\{x \in \fP(v) : \Gamma(v,x) \subseteq \Lambda_r \text{ and } |\Gamma(v,x)| \leq n\}|^2}
\\= \sum_{v,x,y,z \in \Lambda_r} \sum_{a=0}^n \sum_{b=0}^n \sum_{c=0}^{n-a \vee b} \P(A(v,x,y,z;a,b,c)).
\end{multline}
Let $B(x,y,z;a,b)$ be the event that the futures of $x$ and $y$ coalesce for the first time at $z$, that the path from $x$ to $y$ is contained in $\Lambda_r$, and that the paths connecting $x$ to $z$ and $y$ to $z$ have lengths $a$ and $b$ respectively. Then $B(x,y,z;a,b)$ contains the disjoint union $\bigcup_{v\in \Lambda_r} A(v,x,y,z;a,b,c)$ for each $c\geq 0$, so that
\begin{multline}
\label{eq:past2ndcomparison}
\sum_{v\in \Lambda_r} \E{|\{x \in \fP(v) : \Gamma(v,x) \subseteq \Lambda_r \text{ and } |\Gamma(v,x)| \leq n\}|^2}
\\\leq (n+1)\sum_{x,y,z \in \Lambda_r} \sum_{a=0}^n \sum_{b=0}^n \P(B(x,y,z;a,b)).
\end{multline}
On the other hand, if $x$ belongs to the set $\{x \in \fT(v) : \Gamma(v,x) \subseteq \Lambda_{r} \text{ and } |\Gamma(v,x)| \leq 2n \}$ and $z$ denotes the point at which the futures of $v$ and $x$ coalesce then $v$ and $x$ both belong to the past of $z$ and the paths connecting $v$ to $z$ and $x$ to $z$ have total length at most $2n$, so that
\begin{multline}
\label{eq:total1stcomparison}
\sum_{v\in \Lambda_r} \E{|\{x \in \fT(v) : \Gamma(v,x) \subseteq \Lambda_r \text{ and } |\Gamma(v,x)| \leq 2n\}|}
\\= \sum_{v,x,z \in \Lambda_r} \sum_{a=0}^{2n} \sum_{b=0}^{2n-a}  \P(B(v,x,z;a,b)) \geq
\sum_{x,y,z \in \Lambda_r} \sum_{a=0}^n \sum_{b=0}^n \P(B(x,y,z;a,b)).
\end{multline}
Comparing \eqref{eq:past2ndcomparison} and \eqref{eq:total1stcomparison} yields the claim.
\end{proof}

The last ingredient needed to complete the proofs of \cref{prop:volume_lower,prop:extrinsic_volume} is as follows.

\begin{prop}
\label{prop:extrinsic_past_first}
Let $\F$ be the uniform spanning tree of $\Z^4$. Then
\[
\frac{1}{|\Lambda_{r}|}\sum_{v\in \Lambda_{r}}\mathbb{E} |\{x \in \fP(v) : \Gamma(v,x) \subseteq \Lambda_{r} \text{ and } |\Gamma(v,x)| \leq n\}| \asymp n
\]
for every $n \geq 2$ and $r \geq n^{1/2} (\log n)^{1/6}$. 
\end{prop}

\begin{proof}
Exchanging the order of summation gives that
\begin{align*}
&\sum_{v\in \Lambda_{r}}\mathbb{E} |\{x \in \fP(v) : \Gamma(v,x) \subseteq \Lambda_{r} \text{ and } |\Gamma(v,x)| \leq n\}|
\\
&\hspace{3cm}=
\sum_{x\in \Lambda_{r}}
\sum_{v\in \Lambda_{r}} \P(\text{$v$ is in the future of $x$,  $\Gamma(v,x) \subseteq \Lambda_{r}$, and $|\Gamma(x,v)| \leq n$})
\\
&\hspace{3cm}= \sum_{x \in \Lambda_{r}} \sum_{m=0}^n \prstart{\LE(X)^m \subseteq \Lambda_{r}}{x}.
\end{align*}
This sum is trivially bounded by $(n+1)|\Lambda_{r}|$. On the other hand, letting $k=k(n)= \lceil 2 n (\log n)^{1/3} \rceil$ and recalling the definition of $\ell_m$ from Section~\ref{subsec:LERWbackground} we have that
\begin{multline*}
\sum_{v\in \Lambda_{r}}\mathbb{E} |\{x \in \fP(v) : \Gamma(v,x) \subseteq \Lambda_{r} \text{ and } |\Gamma(v,x)| \leq n\}| \geq 
\sum_{x \in \Lambda_{\lfloor r/2 \rfloor}} \sum_{m=0}^n \prstart{\LE(X)^m \subseteq \Lambda_{r}}{x}
\\
 \geq \sum_{x \in \Lambda_{\lfloor r/2 \rfloor}} \sum_{m=0}^n \left[1-\prstart{\ell_m \geq k}{x}-\prstart{ \max_{0 \leq i \leq k} \|X_i\|_\infty \geq r/2 }{x} \right].
\end{multline*}
\cref{lem:amounterased} implies that $\prstart{\ell_m \geq k}{x} \leq \prstart{\ell_n \geq k}{x} = o(1)$ for each $m\leq n$ as $n \to \infty$, while the assumption that $r \geq n^{1/2} (\log n)^{1/6}$ implies by Donsker's invariance principle that the probability $\prstart{ \max_{0 \leq i \leq k} \|X_i\|_\infty \geq r/2 }{x}$ is bounded away from $1$ uniformly in $n$, and the claim follows. 
\end{proof}

\begin{proof}[Proof of \cref{prop:volume_lower} and \cref{prop:extrinsic_volume}]
If we sample $\fT$ using Wilson's algorithm starting with the vertices $0$ and $x$ and find that the path $\Gamma(0,x)$ is contained in $\Lambda_r$, then the two random walks started from $0$ and $x$ must intersect at some point in $\Lambda_r$. Thus, the claimed upper bound of \cref{prop:extrinsic_volume} follows immediately from \cref{thm:sumofinters}.

We now turn to the lower bounds of \cref{prop:extrinsic_volume,prop:volume_lower}. Let $c>0$ be the implicit constant from the lower bound of \cref{prop:extrinsic_past_first}. Let $r\geq 1$, let $n = \lfloor r^2 (\log r)^{-1/3} \rfloor$, let $v_0$ be a uniform random point in $\Lambda_r$, and consider the random variable
\[
Z_r = |\{x\in \fP(v_0,n) \setminus \fP(v_0, \lfloor c n / 2\rfloor) : \Gamma(v_0,x) \subseteq \Lambda_r \}|.
\]
We have by \cref{prop:unrootedupper} that
\begin{align*}
	\P(Z_r > 0) \lesssim \frac{(\log n)^{1/3}}{n} \asymp \frac{(\log r)^{2/3}}{r^2}
\end{align*}
and by~\cref{prop:extrinsic_past_first} and \cref{lemma:WUSF_ball} that
\begin{equation*}
\E {Z_r} \geq cn - \E{|\fP(v_0, \lfloor c n / 2\rfloor)|} \geq cn-cn/2 \asymp \frac{r^2}{(\log r)^{1/3}}, \quad \text{ so that } \quad \E{ Z_r \mid Z_r > 0} \gtrsim \frac{r^4}{\log r}.
\end{equation*}
On the other hand, we have by \cref{lem:past_total_comparison} that
\[
\E{Z_r^2} \lesssim 
(n+1) \mathbb{E} |\{x \in \Lambda_r : \Gamma(v_0,x) \subseteq \Lambda_r \}| \asymp \frac{r^2}{(\log r)^{1/3}} \mathbb{E} |\{x \in \Lambda_r : \Gamma(v_0,x) \subseteq \Lambda_r \}|
\]
 and hence by the Paley-Zygmund inequality that there exists a positive constant $c'$ such that
\begin{multline}
\label{eq:volumetailtomoment}
\P\left(\fP(v_0) \geq \frac{c'r^4}{\log r}\right) \geq
\P\left(Z_r \geq \frac{c'r^4}{\log r}\right) \geq \frac{\E{Z_r}^2}{4\E{ Z_r^2}} \gtrsim \frac{r^2}{(\log r)^{1/3}\mathbb{E} |\{x \in \Lambda_r : \Gamma(v_0,x) \subseteq \Lambda_r \}|}
\end{multline}
for every $r\geq 2$.
Since $\Lambda_r \subseteq \Lambda_{2r} - v$ for every $v\in \Lambda_r$ and $\Lambda_r \supseteq \Lambda_{\lfloor r/2 \rfloor}-v$ for every $v\in \Lambda_{\lfloor r/2 \rfloor}$, we have that
\[
\mathbb{E} |\{x \in \Lambda_r : \Gamma(0,x) \subseteq \Lambda_{\lfloor r/2 \rfloor} \}| \lesssim \mathbb{E} |\{x \in \Lambda_r : \Gamma(v_0,x) \subseteq \Lambda_r \}| \leq \mathbb{E} |\{x \in \Lambda_r : \Gamma(0,x) \subseteq \Lambda_{2r} \}|
\]
for every $r\geq 1$.
Substituting this estimate into \eqref{eq:volumetailtomoment}, we see that the upper bound of \cref{prop:extrinsic_volume} implies the claimed lower bound of \cref{prop:volume_lower} and the upper bound of \cref{prop:volume_upper} implies the claimed lower bound of \cref{prop:extrinsic_volume}.
\end{proof}

We have now proven all the bounds stated in \cref{thm:wusf}.

\begin{proof}[Proof of \cref{thm:wusf}]
The upper and lower bounds of \eqref{exponent:intrad} are established by \cref{prop:unrootedupper} and \cref{pro:lowerbounds} respectively, while the upper and lower bounds of \eqref{exponent:vol} are established by \cref{prop:volume_upper}  and \cref{prop:volume_lower} respectively.
\end{proof}

\subsection{Intersections inside a box}\label{sec:spaceintersections}
In this section we give the proof of Theorem~\ref{thm:sumofinters}. We start by computing the first and second moments of the random variable
\[
I_r = \sum_{i=0}^{\infty}\sum_{j=0}^{\infty} \1(X_i=Y_j \in \Lambda_r),
\]
which counts the total number of intersections that happen inside the box $\Lambda_r$. 

\begin{lemma}\label{lem:spaceinters_first_moment}
		Let $X$ and $Y$ be two independent simple random walks in $\Z^4$, where $Y$ starts at the origin and $X$ starts at independent uniform random point $X_0$ of $\Lambda_r$. Then
	\[
	\E{I_{r}} \asymp 1 \qquad \text{ and } \qquad \E{I_{r}^2} \asymp \log r
	\]
	for every $r\geq 2$.
\end{lemma}

\begin{proof}[\bf Proof]
Write $\langle x \rangle = \|x\|_\infty+1$ for each $x\in \Z^4$. For the first moment, we have that
	\begin{align*}
		\E{I_{r}}\asymp r^{-4}\sum_{y,z\in \Lambda_r}  \bG(z-0)\bG(z-y) \asymp r^{-4} \sum_{y,z\in \Lambda_r}  \langle z\rangle^{-2} \langle y-z \rangle^{-2}
	\end{align*}
for every $r\geq 1$.
	For each fixed $x \in \Lambda_r$ and $\ell \geq 1$, there are $\Theta((\ell+1)^3)$ points in $\Lambda_r$ at $\|\cdot\|_\infty$-distance exactly $\ell$ from $x$, and we deduce that
	\[
r^{-4}\sum_{z\in \Lambda_{r}}  \langle z\rangle^{-2} \sum_{y \in \Lambda_{r}}\langle y-z \rangle^{-2} \asymp r^{-4} \left(\sum_{\ell=0}^{r} (\ell+1)^{-2}\cdot (\ell+1)^3 \right)^2 \asymp 1
	\]
	for every $r\geq 1$ as claimed. 
	We now turn to the second moment. As above, we can write
		\begin{align*}
		\E{I_{r}^2}&\asymp r^{-4}\sum_{y,z,w\in \Lambda_r} \left[\bG(z-0)\bG(w-z)+\bG(w-0)\bG(z-w)\right] \\
		&\hspace{7cm} \cdot \left[\bG(z-y)\bG(w-z)+\bG(w-y)\bG(z-w)\right] \\
		&=2r^{-4} \sum_{y,z,w \in \Lambda_r}  \bG(z)\bG(w-z)^2 \bG(z-y) + 2r^{-4}\sum_{y,z,w \in \Lambda_r}  \bG(z)\bG(w-z)^2 \bG(w-y) \\
		&\asymp r^{-4} \sum_{y,z\in \Lambda_r}\bG(z)\bG(z-y)\sum_{\ell=1}^{r} \frac{\ell^3}{\ell^4} \asymp r^{-4} \log r \sum_{w,z\in \Lambda_r} \bG(z)\bG(z-y) \asymp \log r
	\end{align*}
	and this completes the proof.
\end{proof}

\begin{proof}[\bf Proof of \cref{thm:sumofinters}]
The lower bound follows immediately from \cref{lem:spaceinters_first_moment} and the Cauchy-Schwarz inequality $\P(I_{r} > 0) \geq \E{I_r}^2/\E{I_r^2}$.

We now prove the upper bound, for which we follow a similar argument to that used by Lawler in~\cite[Theorem 3.3.2]{Law91}. 
Since a uniform point in $\Lambda_{2r}$ lies in $\Lambda_r$ with probability bounded from below, \cref{lem:spaceinters_first_moment} implies that $\E{I_{2r}} \lesssim 1$.
 Define $\til{D}_r$ to be the conditional expectation of $I_r$ given the random walk $Y$, so that
\begin{align*}
	\til{D}_r= \econd{I_r}{Y} = \sum_{j=0}^{\infty} \1(Y_j \in \Lambda_r) \bG(Y_j),
\end{align*}
where $\bG$ is the Green kernel for simple random walk in $\Z^4$. Following Lawler, we also define 
\begin{align*}
	D_n= \sum_{j=0}^{n} \bG(Y_j)
\end{align*}
for each $n\geq 1$. Note that we are using $r$ as a distance variable and $n$ as a time variable.
Using the estimates on the variance and the first moment of $D_n$ proven by Lawler in~\cite[Proposition~3.4.1]{Law91}, we get that there exists a positive constant $C_1$  such that
\begin{align}\label{eq:boundfromlawler}
	\pr{D_n\leq C_1 \log n}\lesssim \frac{1}{\log n}
\end{align}
for every $n\geq 2$.
Writing $\tau_r$ for the first exit time of the box $\Lambda_r$ by $Y$, we notice that $\tau_r\geq r$, and hence
$\til{D}_r\geq D_{r}$ which gives
\begin{align}
	\pr{\til{D}_r\leq C_1 \log r} \leq  \pr{D_{r}\leq C_1\log r}\lesssim \frac{1}{\log r}.
	\label{eq:boundontild}
\end{align}
Fix $r\geq 2$. Following Lawler again, 
we call a time $j \geq 0$ \textbf{good} if 
\[
\sum_{\ell=0}^{\infty} \1(Y_{j+\ell}-Y_j \in \Lambda_r) \bG(Y_{j+\ell} - Y_j) \geq C_1\log r
\]
and bad otherwise. The bound~\eqref{eq:boundontild} then immediately implies
\begin{align}\label{eq:probjisbad}
\pr{j \text{ is bad}} \lesssim \frac{1}{\log r}
\end{align}
for every $j\geq 0$. Let $\tau=\inf\{i\geq 0: X_i\in Y\cap \Lambda_r\}$ and $\sigma=\inf\{j: Y_j=X_{\tau}\}$.
With these definitions we now have 
\begin{align}\label{eq:boundoninters}
	&\pr{X\cap Y\cap \Lambda_r \neq \emptyset} = 
	 \pr{\tau<\infty, \sigma \text{ is good}} + \pr{\tau<\infty, \sigma \text{ is bad}}.
\end{align}
 Since the event that $j$ is bad is independent of $X$ and $Y^j$ we have by \cref{lem:spaceinters_first_moment} and \eqref{eq:probjisbad} that
\begin{align*}
	\pr{\tau<\infty, \sigma \text{ is bad}} \leq \sum_{i=0}^{\infty}\sum_{j=0}^{\infty} \pr{X_i=Y_j\in \Lambda_r} 
	 \pr{j\text{ is bad}} \lesssim \frac{\E{I_r} }{\log r}{\asymp \frac{1}{\log r}}
\end{align*}
for every $r\geq 1$, and it remains only to prove a similar bound on the first term of \eqref{eq:boundoninters}.
We have that 
\begin{align}\label{eq:firsttermofbound}
	\pr{\tau<\infty, \sigma \text{ is good}} \leq \frac{\E{I_{2r}}}{\E{I_{2r} \mid \tau<\infty, \sigma \text{ is good}}}.
\end{align}
Observe that the random variable $(X_{\tau+i}-X_\tau)_{i\geq 0}$ is conditionally independent of $X^\tau$ and $Y$ given that $\tau<\infty$. Defining
\[
I'_r = \sum_{i=0}^\infty\sum_{i=j}^\infty \mathbf{1} (X_{\tau+i}=Y_{\sigma+j}\in \Lambda_r - Y_\sigma)
\]
and observing that $I_{2r} \geq I_r'$ on the event that $\tau<\infty$, it follows that
\begin{multline*}
	\E{I_{2r}\mid\tau<\infty, \sigma \text{ is good}} \geq \E{I_{r}'\mid\tau<\infty, \sigma \text{ is good}}  \\
	= \E{\sum_{j=0}^{\infty}\1(Y_{j+\sigma} -Y_\sigma\in \Lambda_r) \bG(Y_{j+\sigma}-Y_\sigma) \;\Big|\; \tau<\infty, \sigma\text{ is good}}\geq C_1\log r,
\end{multline*}
where for the final inequality we used that $\sigma$ is a good time and for the first inequality we also used the fact that if  $Y_{j+\sigma}-Y_\sigma\in \Lambda_r$ then $Y_{j+\sigma}\in \Lambda_{2r}$, since $Y_\sigma\in \Lambda_r$ by definition.
Substituting this bound into~\eqref{eq:firsttermofbound} we get that
\begin{align}\label{eq:onepart}
	\pr{\tau<\infty, \sigma \text{ is good}} \leq \frac{\E{I_{2r}}}{C_1\log r}\lesssim \frac{1}{\log r}
\end{align}
as required.
\end{proof}

\begin{remark}
	\rm{
	We note that if $\norm{X_0}\asymp r$ and $Y_0=0$, then we also get 
	\[
	\pr{X\cap Y \cap \Lambda_r\neq \emptyset} \asymp \frac{1}{\log r}.
	\]
	Indeed, the initial starting points only enter the proof of~\cref{thm:sumofinters} via the first and second moments of $I_r$, which are given by~\cref{lem:spaceinters_first_moment}.  It is then easy to check that the proof of~\cref{lem:spaceinters_first_moment} goes through for these initial conditions as well.
	}
\end{remark}

\section{Upper bounds on the extrinsic radius}
\label{sec:extrinsic}

In this section we prove the upper bounds on the tail of the extrinsic radius.

\begin{prop}
\label{prop:extrinsicupper}
Let $\fP(0)$ and $\fP_0(0)$ be the past of the origin in the uniform spanning tree and $0$-wired uniform spanning forest of $\Z^4$ respectively. Then we have that
\begin{align}
\P(\fP(0) \cap \partial \Lambda_r \neq \emptyset) \lesssim_\eps \frac{(\log r)^{2/3+\eps}}{r^2}
\end{align}
and
\begin{align}
\P(\fP_0(0) \cap \partial \Lambda_r \neq \emptyset) \lesssim_\eps \frac{(\log r)^{1+\eps}}{r^2}
\end{align}
for every $\eps>0$ and $r\geq 2$.
\end{prop}

Before proceeding, let us briefly note how our remaining main theorems follow from this proposition together with our earlier results.

\begin{proof}[Proof of \cref{thm:extrinsic}]
The lower bound is established by \cref{pro:lowerbounds}, while the upper bound is established by \cref{prop:extrinsicupper}.
\end{proof}

\begin{proof}[Proof of \cref{thm:wusfo}]
The upper and lower bounds of \eqref{exponent:intrado} are established by \cref{prop:rootedupper} and \cref{pro:lowerbounds0} respectively, the upper and lower bounds of \eqref{exponent:extrado} are established by \cref{prop:extrinsicupper} and \cref{pro:lowerbounds0} respectively, while \eqref{exponent:volo} is established by \cref{prop:volume_upper}.
\end{proof}

The remainder of the section is structured as follows. In \cref{sec:typicaltime} we introduce the notions of the \emph{typical time} and the \emph{time radius}, and prove upper bounds on the tail of the typical time radius (\cref{prop:timeradius}) conditional on a technical lemma whose proof is deferred to \cref{sec:bathpads}. We then use this machinery to complete the proof of \cref{prop:extrinsicupper} in \cref{subsec:extrinsicproof}.

\subsection{Typical times and the time radius}
\label{sec:typicaltime}

 We now begin to develop some machinery that will be used in the proof of \cref{prop:extrinsicupper}.
Let $\eta$ be a finite path connecting two vertices $x$ and $y$ in $\Z^4$.
Recall from the proofs of \cref{lem:weakl1} and \cite[Lemma 5.1]{1804.04120} that if $X$ is a simple random walk started at $x$ then, conditional on the event that $\tau_y<\infty$ and $\LE(X^{\tau_y}) = \eta$, the loop lengths 
 $(\ell_i-\ell_{i-1})_{i=1}^{|\eta|}:=(\ell_i(X^{\tau_y})-\ell_{i-1}(X^{\tau_y}))_{i=1}^{|\eta|}$ are independent and satisfy
\begin{align}\label{eq:boundtobeused}
\P_x(\ell_i-\ell_{i-1} = m+1 \mid \tau_y<\infty,\, \LE(X^{\tau_y})=\eta)&\leq 
\P_{\eta_{i-1}}\left(X^m \cap \eta^{i-2} = \emptyset,\, X_m = \eta_{i-1} \right)
\nonumber
\\
&\leq p_m(0,0) \lesssim \frac{1}{m^2}
\end{align}
for each $1\leq i \leq |\eta|$.
Writing $a\wedge b = \min\{a,b\}$, we define the \textbf{typical time} of $\eta$ to be
\[
T(\eta):= 
\mathbb{E}_x\left[ \sum_{i=1}^{|\eta|} \left(\ell_i(X^{\tau_y})-\ell_{i-1}(X^{\tau_y})\right)\wedge |\eta| \;\Bigg|\; \tau_y<\infty,\,\LE(X^{\tau_y})=\eta \right].
\]
This terminology is justified by the following lemma.

\begin{lemma}[Conditional concentration of the time around the typical time] 
\label{lem:typicaltime}
There exists a constant $C$ such that if $x,y\in \Z^4$ and $\eta$ is a simple path from $x$ to $y$ then
\begin{equation}
\label{eq:typicaltime1}
\P_x\Bigl(\bigl|\tau_y-T(\eta)\bigr| \geq \lambda |\eta| \;\Big|\; \tau_y<\infty, \LE(X^{\tau_y})=\eta\Bigr) \leq \frac{C}{\lambda}
\end{equation}
for every $\lambda \geq 1$.
\end{lemma}

\begin{remark}
Four dimensions should be marginal for a concentration result of this form to hold. In dimension $d\leq 3$, we expect the sum of the loop lengths $\ell_i(X)-\ell_{i-1}(X)$ to be dominated by its maximum and therefore not to be concentrated. 
\end{remark}

The proof of this lemma
 will use Bennett's inequality \cite[Theorem 2.9]{MR3185193}, which states that if $Z_1,\ldots,Z_n$ are independent random variables with $|Z_i| \leq M$ for every $1 \leq i \leq n$  and we define $S=\sum_{i=1}^n Z_i$, $\mu = \E{S} = \E {\sum_{i=1}^n Z_i}$, and $\sigma^2 = \operatorname{Var}(S) = \sum_{i=1}^n \operatorname{Var}(Z_i)$ then
\begin{equation}
\label{eq:Bennett}
\P\left(|S - \mu| \geq t\right) \leq 2 \exp\left[-\frac{\sigma^2}{M^2} \cdot h\!\left(\frac{Mt}{\sigma^2}\right) \right]
\end{equation}
for every $t>0$, 
where $h:(0,\infty)\to (0,\infty)$ is the increasing function $h(x)=(1+x)\log(1+x)-x$.

\begin{proof}[Proof of \cref{lem:typicaltime}]
Let $\eta$ be a simple path between $x$ and $y$ of length $n \geq 0$.
The claim holds vacuously when $n=0$, so we may assume that $n\geq 1$.
 Let $M \geq 1$ be a parameter to be optimized over. Write $\ell_i=\ell_i(X^{\tau_y})$ for each $0 \leq i \leq |\eta|$. We will apply Bennet's inequality to the random variables 
 \[Z_i=Z_i(M)=\left(\ell_i-\ell_{i-1}\right)\wedge M
\qquad i =1 ,\ldots,n\] conditioned on the event $\Omega=\{\tau_y < \infty$, $\LE(X^{\tau_y}) = \eta\}$. Let $S=S(M)=\sum_{i=1}^n Z_i$. As discussed above, the random variables $Z_1,\ldots,Z_n$ are conditionally independent given $\Omega$ and by~\eqref{eq:boundtobeused} there exists a constant $C_1$ such that $\P(Z_i \geq t \mid \Omega) \leq C_1 t^{-1}$ for every $t\geq 1$ and $1 \leq i \leq n$. It follows in particular that
\[
\sum_{i=1}^n\operatorname{Var}(Z_i \mid \Omega) \leq \sum_{i=1}^n
\E{Z_i^2 \mid \Omega} \leq 2
 \sum_{i=1}^n\sum_{t=1}^{\lfloor M \rfloor} t \cdot \P(Z_i \geq t \mid \Omega) \leq 2C_1 M n
\]
for every $1\leq i \leq n$ and $M \geq 1$ and, since $\tau_y=\ell_n$, that 
\[
\P\left(\tau_y \neq S \mid \Omega\right) = \P( \ell_i-\ell_{i-1} > M \text{ for some $1\leq i \leq n$} \mid \Omega) \lesssim \frac{n}{M}
\]
for every $M \geq n$. Letting $\mu(M) = \sum_{i=1}^n \E{Z_i \mid \Omega }$, it follows by a union bound and Bennett's inequality that
\[
\P\left(|\tau_y - \mu(M)| \geq \lambda n \mid \Omega \right) \lesssim \frac{n}{M} + \exp\left[-\frac{2C_1n}{M} \cdot h\left(\frac{\lambda}{2C_1}\right)\right]
\]
for every $M \geq n$. Taking the approximately optimal choice of $M$ 
\[M(\lambda) := n \vee \frac{2C_1 n h (\lambda/2C_1)}{\log h(\lambda/2C_1)}  \asymp \lambda n\]
yields that there exists a positive constant $C_2$ such that
\[
\P\left(|\tau_y - \mu(M(\lambda))| \geq \lambda n \mid \Omega\right) \lesssim \frac{\log h(\lambda / 2C_1)}{h(\lambda/2C_1)} \lesssim \frac{1}{\lambda}
\]
for every $\lambda \geq C_2$. By increasing the value of the implicit constant if necessary we may take this inequality to hold for all $\lambda \geq 1$. On the other hand, we have that
\[
\mu(M(\lambda)) - T(\eta) \leq \sum_{i=1}^{n} \sum_{t=n+1}^{\lfloor M(\lambda) \rfloor} \P(Z_i \geq t \mid \Omega) \lesssim n\log \frac{M(\lambda)}{n} \lesssim n\log \lambda.
\]
It follows that there exists a constant $C_3$ such that if $\lambda \geq C_3$ then $|\mu_n(M_n(\lambda)) - T_n(\eta)| \leq \lambda n$ and hence that
\[
\P\left(|\tau_y - T(\eta)| \geq 2\lambda n \mid \Omega\right)
\leq \P\left(|\tau_y - \mu(M(\lambda))| \geq 2\lambda n \mid \Omega\right) \lesssim \frac{1}{\lambda}
\]
for every $\lambda \geq C_3$. This is easily seen to imply the claim.
\end{proof}

Note that we trivially have that $T(\eta) \geq |\eta|$ for every finite simple path $\eta$. On the other hand, if $\eta$ is a path of length $n\geq 2$ we also have as above that
\begin{align}
T(\eta) 
&\leq 
\sum_{i=1}^{n} \sum_{k=1}^{n} k  \mathbb{P}_{\eta_{i-1}}\left( X_{k-1} = \eta_{i-1} \text{ and } X^{k-1}\cap \eta^{i-2}=\emptyset  \right) \lesssim \sum_{i=1}^{n} \sum_{k=1}^{n} \frac{1}{k} \asymp n \log n.
\end{align}
We remark that the lower bound $T(\eta) \gtrsim |\eta|$ is attained up to constants when $\eta$ is a space-filling curve, while the upper bound $T(\eta) \lesssim |\eta| \log |\eta|$ is attained when $\eta$ is a straight line. In all the regimes we are interested in we have that $T(\eta) \gg |\eta|$, so that \cref{lem:typicaltime} can indeed be thought of as a concentration estimate. Indeed, when $\eta$ is a loop-erased random walk of length $n$ we will typically have that $T(\eta) \asymp n(\log n)^{1/3}$.  

It will be helpful for us to replace the typical time by a simpler quantity that is easier to estimate. 
Given a finite path $\eta$ of length $n$ and $k\geq 0$ we define $\operatorname{Esc}_k(\eta)=\P_{\eta_n}(X^k \cap \eta^{n-1} = \emptyset)$ to be the probability that the first $k$ steps of a random walk started at the endpoint of $\eta$ do not intersect the other points of $\eta$. Let $a=\lfloor k/3 \rfloor$ and $b=k-1-\lfloor k/3 \rfloor$. 
Splitting the random walk of length $k$ into three pieces of roughly equal length and applying a time-reversal to the final piece, we deduce that if $\eta$ is a (finite or infinite) simple path and $0 \leq i \leq |\eta|$ then
\begin{align*}
&\mathbb{P}_{\eta_i}\left( X_{k-1} = \eta_i \text{ and } X^{k-1}\cap \eta^{i-1}=\emptyset  \right) \\
&\hspace{2.3cm}
\leq \sum_{x,y}\prstart{X\left[0,a\right] \cap \eta^{i-1}=\emptyset, X_{a}=x, X_{ b}=y, X\left[ b,k-1\right] \cap \eta^{i-1}=\emptyset, X_{k-1}=\eta_{i} }{\eta_{i}}
\\&\hspace{2.3cm}\leq \operatorname{Esc}_{\lfloor k/3 \rfloor}(\eta^i)^2 \sup_{x,y \in \Z^4} \P_x(X_{b-a} = y)\lesssim \frac{1}{k^2} \operatorname{Esc}_{\lfloor k/3 \rfloor}(\eta^i)^2
\end{align*}
for every $k\geq 1$.
Thus, if $\eta$ is a finite simple path of length $n$ and we define $A_{i}(\eta) = \sum_{k=1}^n \frac{1}{k} \operatorname{Esc}_{k}(\eta^i)^2$ for each $0\leq i \leq n-1$ and $\widetilde T(\eta)=\sum_{i=0}^{n-1}A_i(\eta)$ then
\[\widetilde T(\eta) := \sum_{i=0}^{n-1} A_{i}(\eta) =\sum_{i=0}^{n-1} \sum_{k=1}^n \frac{1}{k} \operatorname{Esc}_{k}(\eta^i)^2 \gtrsim T(\eta).\]


Let $\fT$ and $\fF_0$ be the uniform spanning tree and $0$-wired uniform spanning forest of $\Z^4$.
For each vertex $x$ in the past $\fP=\fP(0)$ we define $T(x)=\widetilde T(\Gamma(x,0))$, where $\Gamma(x,0)$ is the path connecting $x$ to $0$ in $\fT$. Similarly, for each vertex $x$ in the past $\fP_0=\fP_0(0)$ we define $T_0(x)=\widetilde T(\Gamma_0(x,0))$ where $\Gamma_0(x,0)$ is the path connecting $x$ to $0$ in $\fF_0$. We define the \textbf{time radii} $\sT(0)$ and $\sT_0(0)$ of $\fP(0)$ and $\fP_0(0)$ to be
\[
\sT(0) = \max\{T(x) : x\in \fP(0)\} \qquad \text{ and } \qquad \sT_0(0) = \max \{T_0(x) : x\in \fP_0(0)\}
\]
respectively.
 Intuitively, we typically expect that $T(x)$ should be of order $n (\log n)^{1/3}$ when $x \in \partial\fP(0,n)$, and hence that the time radius and the intrinsic radius of the past should be related in a similar way with high probability.
 Our next goal is to prove the following tail estimate, which makes this intuition precise with regards to tail probabilities.

\begin{prop}
\label{prop:timeradius}
 Let $\sT(0)$ and $\sT_0(0)$ be the time radius of the past of the origin in the uniform spanning tree and $0$-wired uniform spanning forest of $\Z^4$ respectively. Then
\begin{equation}
 \P(\sT(0) \geq t) \lesssim_\eps \frac{(\log t)^{2/3+\eps}}{t} 
\qquad\text{and}\qquad \P(\sT_0 (0) \geq t) \lesssim_\eps \frac{(\log t)^{1+\eps}}{t} 
\label{eq:timeradius_rooted}
\end{equation}
for every $t \geq 2$ and $\eps>0$.
\end{prop}

We now introduce some further  definitions that will be used in the proof of \cref{prop:timeradius}. 
Given $\delta>0$, we say that a finite path $\eta$ of length $n \geq 0$ is $\delta$\textbf{-good} if 
\[
\sum_{i=0}^{n-1} A_{i}(\eta) \mathbbm{1}\left( A_{i} \geq (\log n)^{1/3+\delta}\right) \leq \delta n.
\]
(Of course this definition is highly arbitrary.)
We say that $\eta$ is $\delta$\textbf{-bad} if it is not $\delta$-good. Observe that if $\eta$ is a $\delta$-good simple path of length $n \geq 2$ then
\[
\widetilde T(\eta) \leq \delta n + \sum_{i=0}^{n-1}A_{i}(\eta)\mathbbm{1}\left(A_{i} < (\log n)^{1/3+\delta}\right) \lesssim n (\log n)^{1/3+\delta}.
\]
 The following technical lemma, which relies on the results of Lawler \cite{Lawlerlog}, is proved in Section~\ref{sec:bathpads}.

 \begin{lemma}[Bad paths are highly unlikely]\label{lem:geomfixedtime}
Let $\delta>0$ and $p\geq 0$. Then
	\begin{align*}
		\frac{1}{n}\sum_{k=0}^{n} \prstart{\LE(X^k) \text{\emph{ is $\delta$-bad}}}{0} 
	\lesssim_{\delta,p} \frac{1}{(\log n)^p}.
	\end{align*}
  for every $n\geq 2$.
\end{lemma}

In the remainder of this subsection we complete the proof of \cref{prop:timeradius} conditional on \cref{lem:geomfixedtime}. We begin by deducing that it is very hard for the past of the origin to contain a long $\delta$-bad path.

\begin{lemma}
\label{lem:roughpaths2}
Let $\delta>0$ and $p \geq 0$. Then
\[
\P\left(\exists x \in \fP_0(0,2n)\setminus \fP_0(0,n) \text{\emph{ such that $\Gamma_0(x,0)$ is $\delta$-bad}}\right) \lesssim_{\delta,p} \frac{1}{n(\log n)^p}\]
for every $n \geq 2$.
\end{lemma}

\begin{proof}
Suppose $x \in \fP_0(0,2n)\setminus \fP_0(0,n)$ is such that $\Gamma_0(x,0)$ is $\delta$-bad. It is readily verified from the definitions that there exists a constant $n_0=n_0(\delta)$ 
such that if $n \geq n_0$ and $m_0=\lfloor n/(\log n)^2 \rfloor$, $m_0 \leq m \leq 2m_0$, and $y$ is the vertex $m$ steps into the future of $x$ then $\Gamma_0(y,0)$ is $\delta/2$-bad. Thus, on the event in question, there must exist at least $m_0$ points $y \in \fP(0,2n)$ such that $\Gamma_0(y,0)$ is $\delta/2$-bad and $\partial \fP_0(y,m_0) \neq \emptyset$. It follows by the stochastic domination property (\cref{lem:domination}) and Markov's inequality that if $n \geq n_0$ then
\begin{multline*}
\P\left(\exists x \in \fP_0(0,2n)\setminus \fP_0(0,n) \text{ such that $\Gamma_0(x,0)$ is $\delta$-bad}\right) \\
\leq \frac{1}{m_0}\sum_y \P(\text{$y \in \fP_0(0,2n)$ and $\Gamma_0(y,0)$ is $\delta/2$-bad}) Q_0(m_0),
\end{multline*}
where $Q_0(m_0)$ is the probability that $\partial \fP_0(0,m_0) \neq \emptyset$.  If we consider sampling the $0$-wired uniform spanning forest using Wilson's algorithm starting with $y$, it follows from \cref{lem:weakl1} that there exists a constant $C$ such that
\begin{align*}
\P(\text{$y \in \fP_0(0,2n)$ and $\Gamma_0(y,0)$ is $\delta/2$-bad}) 
&= \P_y\bigl(|\LE(X^{\tau_0})| \leq 2n, \LE(X^{\tau_0}) \text{ is $\delta/2$-bad} \bigr)
\\
&\leq 2 \sum_{k=0}^{\lceil C n \log n \rceil} \P_y\bigl(X_k=0, \LE(X^k) \text{ is $\delta/2$-bad} \bigr).
\end{align*}
Applying the mass-transport principle and \cref{prop:rootedupper} yields that
\begin{align*}
&\P\left(\exists x \in \fP_0(0,2n)\setminus \fP_0(0,n) \text{ such that $\Gamma_0(x,0)$ is $\delta$-bad}\right) \\
&\hspace{6.5cm}\leq \frac{1}{m_0}\sum_{k=0}^{\lceil C n \log n \rceil} \P_0\bigl(\LE(X^k) \text{ is $\delta/2$-bad} \bigr) Q_0(m_0) 
\\&\hspace{6.5cm}\lesssim_{\delta,p} \frac{(\log n)^{2}}{n} \cdot (n\log n) \cdot \frac{1}{(\log n)^p} \cdot \frac{(\log n)^{2+2/3} \log \log n}{n}
\end{align*}
for every $n \geq n_0$, $\delta >0$ and $p \geq 0$, which clearly implies the claim. 
\end{proof}



\begin{proof}[Proof of \cref{prop:timeradius}]
We will prove the claim concerning $\sT(0)$, the claim concerning $\sT_0(0)$ following similarly. 
Fix $\delta>0$. It follows from the definitions that there exist positive constants $c_0$ and $c_1=c_1(\delta)$ such that the following hold for every $t\geq 2$:
\begin{enumerate}
\item If $\eta$ is any finite path of length at most $c_0 t / \log t$ then $\widetilde T(\eta) < t$. 
\item If $\eta$ is a $\delta$-good path of length at most $c_1t/(\log t)^{1/3+\delta}$ then $\widetilde T(\eta) < t$.
\end{enumerate}
Fix $t\geq 2$ and define $n_0=\lceil c_0 t/\log t \rceil$ and $n_1 = \lfloor c_1t/(\log t)^{1/3+\delta}\rfloor$.
Then we have by a union bound that
\begin{align*}
\P(\sT(0) \geq t) &\leq \P( \partial \fP(0,n_1) \neq \emptyset) + \P(\exists x \in \fP(0,n_1) \setminus \fP(0,n_0) \text{ such that } \Gamma(x,0) \text{ is $\delta$-bad})\\
&\leq \P( \partial \fP(0,n_1) \neq \emptyset) + \P(\exists x \in \fP_0(0,n_1) \setminus \fP_0(0,n_0) \text{ such that } \Gamma_0(x,0) \text{ is $\delta$-bad}),
\end{align*}
where we used the stochastic domination property (\cref{lem:domination}) in the second line.
Letting $k_0=\lfloor \log_2 n_0 \rfloor$ and $k_1 = \lceil \log_2 n_1 \rceil$, we deduce via a further union bound that
\begin{multline*}
\P(\sT(0) \geq t) \leq 
\P( \partial \fP(0,n_1) \neq \emptyset) \\+ \sum_{k=k_0}^{k_1} \P(\exists x \in \fP_0(0,2^{k+1}) \setminus \fP_0(0,2^k) \text{ such that } \Gamma_0(x,0) \text{ is $\delta$-bad}).
\end{multline*}
Applying \cref{prop:unrootedupper} to bound the first term and~\cref{lem:roughpaths2}  to bound the second yields that
\[
\P(\sT(0) \geq t) 
\lesssim_{\delta,p} \frac{(\log t)^{2/3+\delta}}{t} + \sum_{k=k_0}^{k_1} \frac{1}{k^p 2^k} \lesssim_{\delta,p} \frac{(\log t)^{2/3+\delta}}{t} + \frac{(\log t)^{1-p}}{t},
\]
and the claim follows easily by choosing an appropriately large value of $p$.
\end{proof}

\subsection{The extrinsic radius}
\label{subsec:extrinsicproof}

We now wish to apply the typical-time technology we developed in the previous subsection to prove \cref{prop:extrinsicupper}. (Strictly speaking, the proofs in this section are still conditional on \cref{lem:geomfixedtime}, which is proven in \cref{sec:bathpads}.)
Let $\fT$ and $\F_0$ be the uniform spanning tree and $0$-wired uniform spanning forest of $\Z^4$ respectively. 
For each $r \in [0,\infty)$ and $s,t \in [0,\infty]$, we define the events
\[\sE(r,s,t)=\Bigl\{\text{there exists $x\in \fP(0, \lfloor s \rfloor)$ with $T(x) \leq t$ and $\|x\|_\infty \geq r$}\Bigr\}\]
and similarly
\[\sE_0(r,s,t)=\Bigl\{\text{there exists $x\in \fP_0(0, \lfloor s \rfloor)$ with $T_0(x) \leq t$ and $\|x\|_\infty \geq r$}\Bigr\}.\]
Let $E(r,s,t)$ and $E_0(r,s,t)$ be the respective probabilities of these two events, which are decreasing in $r$ and increasing in $s$ and $t$. The stochastic domination property,~\cref{lem:domination}, implies that $E(r,s,t) \leq E_0(r,s,t)$ for every $r \in [0,\infty)$ and $s,t \in [0,\infty]$.   We have by \cref{prop:rootedupper,prop:timeradius} that 
\begin{equation}
\P(\fP_0(0)\cap \partial \Lambda_r \neq \emptyset) = E_0(r,\infty,\infty) \lesssim_\eps E_0(r,s,t) + \frac{(\log s)^{2/3+\eps}}{s} + \frac{(\log t)^{1+\eps}}{t}
\label{eq:E0_union_1}
\end{equation}
for every $\eps>0$ and $r,s,t \geq 2$, and similarly by \cref{prop:unrootedupper,prop:timeradius} again that
\begin{multline}
\P(\fP(0)\cap \partial \Lambda_r \neq \emptyset) = E(r,\infty,\infty) \lesssim_\eps E(r,s,t) + \frac{(\log s)^{1/3}}{s} + \frac{(\log t)^{2/3+\eps}}{t} \\\leq E_0(r,s,t) + \frac{(\log s)^{1/3}}{s} + \frac{(\log t)^{2/3+\epsilon}}{t} 
\label{eq:E_union_1}
\end{multline}
for every $\eps>0$ and $r,s,t\geq 2$. Thus, in order to prove \cref{prop:extrinsicupper}, it suffices to prove that 
\begin{equation}
\label{eq:lesscrude}
E_0\left(r,\frac{r^2}{(\log r)^{1/3} (\log \log r)^2},\frac{r^2}{(\log \log r)^2}\right) \lesssim_p \frac{1}{(\log r)^pr^2}
\end{equation}
for every $p\in \R$ and $r\geq 3$. 
(Indeed, it would suffice for our purposes to prove this estimate with $p=-1+\epsilon$; we state it this way to emphasise that there is a lot of room.)

We begin by applying the weak $L^1$ method to prove the following crude estimate. The only important feature of this estimate is that the right hand side is very small when $s \ll r^2 (\log r)^{-C}$ for sufficiently large  $C$.
%
\begin{lemma} 
\label{lem:crude}
There exists a positive constant $c$ such that 
$E_0(r,s,\infty) \lesssim s \log s \exp\left[-\frac{c r^2}{s \log s}\right]$ for every $r,s\geq 2$.
\end{lemma}

\begin{proof}[Proof of \cref{lem:crude}] By Markov's inequality, the stochastic domination property (\cref{lem:domination}) and Wilson's algorithm we have that
\[
E_0(r,s,\infty) \leq
\sum_{v\in \partial \Lambda_r}  \P\left(v \in \fP_0(0,s)\right) =
\sum_{v\in \partial \Lambda_r} \P_v(\tau_0 <\infty, |\LE(X^{\tau_0})| \leq s).
\]
Applying \cref{lem:weakl1} yields that there exists a constant $C$ such that
\[
E(r,s,\infty) \leq 2 \sum_{v\in \partial \Lambda_r} \P_v(\tau_0 \leq C s \log s) = 
2 \sum_{v\in \partial \Lambda_r} \sum_{m=1}^{\lceil Cs\log s\rceil}\P_0(\tau_0^+ >m, X_m = v),
\]
and the claim follows from Azuma's inequality \eqref{eq:Azuma}.
\end{proof}

\begin{proof}[Proof of \cref{prop:extrinsicupper}]
It remains to only prove the estimate \eqref{eq:lesscrude}, the claim then following from \eqref{eq:E0_union_1} and \eqref{eq:E_union_1}. 
All the estimates in this proof will hold with a large amount of room.
Fix $r \geq 3$ and let $t_0 =  r^2/(\log \log r)^2$. For each 
$k\geq 1$ let $a_k = \lfloor \exp(k^{1/100})\rfloor$, let $\sD_k$ be the event that there exists $x \in \fP_0(0,a_{k+1})\setminus \fP_0(0,a_k)$ such that $T_0(x) \leq t_0$ and $\|x\|_\infty \geq r$, and let $D_k$ be the probability of $\sD_k$. If we define 
\[k_0= \left\lfloor \left(\log \frac{r^2}{(\log r)^3}\right)^{100}\right\rfloor \qquad \text{and} \qquad k_1 = \left\lceil \left(\log \frac{r^2}{(\log r)^{1/3} (\log \log r)^2}\right)^{100}\right\rceil\] then we have by a union bound that
\begin{equation}
\label{eq:Dk_union}
E_0\left(r,\frac{r^2}{(\log r)^{1/3} (\log \log r)^2},\frac{r^2}{(\log \log r)^2}\right)
\leq E_0\left(r,\frac{r^2}{(\log r)^3},\infty\right) + \sum_{k=k_0}^{k_1} D_k.
\end{equation}
The first term is superpolynomially small in $r$ by \cref{lem:crude}. We now bound the second term. For each $k_0 \leq k \leq k_1$ define 
\[b_k=a_{k+1}-a_k \asymp \frac{a_k}{k^{99/100}} \asymp \frac{a_k}{(\log a_k)^{99}}.\]
If the event $\sD_k$ holds then there must exist at least $b_k$ vertices $x \in \fP_0(0,a_k-b_k) \setminus \fP_0(0,a_k-2b_k)$ such that there exists $y \in \fP_0(x,3b_k) \setminus \fP_0(x,b_k)$ with $\|y\|_\infty \geq r$ and $T_0(x) \leq T_0(y) \leq t_0$. Decomposing according to whether or not $\|x\|_\infty \geq r/2$ and applying the stochastic domination property (\cref{lem:domination}) we deduce that
\begin{multline*}
D_k \leq \frac{1}{b_k} \E{\#\{x \in \fP_0(0,a_k-b_k) \setminus \fP_0(0,a_k-2b_k) : T_0(x) \leq t_0, \|x\|_\infty \geq r/2\}} Q_0(b_k)
\\
+ \frac{1}{b_k}\E{\#\{x \in \fP_0(0,a_k-b_k) \setminus \fP_0(0,a_k-2b_k) : T_0(x) \leq t_0 \}} E_0\left(\frac{r}{2},3b_k,\infty\right)
\end{multline*}
for every $k_0\leq k \leq k_1$, where again we write $Q_0(m)$ for the probability that $\partial \fP_0(0,m) \neq \emptyset$. We have by Wilson's algorithm and the definition of $T_0(x)$ that there exists a constant $C_1$ such that
\begin{align*}
\P\left(x \in \fP_0(0,a_k-b_k) \setminus \fP_0(0,a_k-2b_k), T_0(x) \leq t_0 \right)
&\leq \P(x \in \fP_0(0), T_0(x) \leq t_0) \\&\leq \P_x\left(\tau_0<\infty, T\bigl(\LE(X^{\tau_0})\bigr) \leq C_1 t_0\right),
\end{align*}
where $T\bigl(\LE(X^{\tau_0})\bigr)$ is the typical time of $\LE(X^{\tau_0})$, so that applying \cref{lem:typicaltime} yields that there exists a constant $C_2$ such that
\begin{multline*}
\P\left(x \in \fP_0(0,a_k-b_k) \setminus \fP_0(0,a_k-2b_k), T_0(x) \leq t_0 \right)
\leq 2\P_x(\tau_0 \leq C_2 t_0)
 =2\sum_{m=0}^{\lceil Ct_0 \rceil} \P_0(X_m=x)
\end{multline*}
for every $x\in \Z^4$. Since $t_0$, $b_k$, and $r^2$ all agree up to polylogarithmic factors of bounded exponent, we can sum over $x$ and apply Azuma's inequality \eqref{eq:Azuma} to deduce that there exist positive constants $c_1,c_2$, $c_3$, and $C_3$ such that
\begin{multline*}
D_k \lesssim \frac{t_0}{b_k} e^{-c_1 r^2/t_0} Q_0(b_k)
+ \frac{t_0}{b_k} E_0\left(\frac{r}{2},3b_k,\infty\right) 
\\\lesssim 
\frac{(\log r)^{C_3}}{r^2} \exp\left[-c_2 (\log \log r)^2\right]
+ r^2(\log r)^{C_3} \exp \left[ -c_3 (\log r)^{98}\right]
\\\lesssim 
\frac{(\log r)^{C_3}}{r^2} \exp\left[-c_2 (\log \log r)^2\right]
\end{multline*}
for every $k_0 \leq k \leq k_1$, where we used \cref{prop:rootedupper} to bound $Q_0(b_k)$ and used \cref{lem:crude} to bound $E_0(r/2,3b_k,\infty)$. Summing this estimate over $k_0\leq k \leq k_1$, substituting the resulting bound into \eqref{eq:Dk_union}, and using that $k_1$ is at most polylogarithmic in $r$, we deduce that there exists a constant $C_4$ such that
\begin{equation}
E_0\left(r,\frac{r^2}{(\log r)^{1/3} (\log \log r)^2},\frac{r^2}{(\log \log r)^2}\right) \\
\lesssim \frac{(\log r)^{C_4}}{r^2} \exp\left[-c_2 (\log \log r)^2\right],
\end{equation}
which decays faster than $r^{-2} (\log r)^{-p}$ for any $p\in \R$ as required.
\end{proof}

\subsection{Bad paths are highly unlikely}\label{sec:bathpads}

In this section we complete the proofs of \cref{prop:extrinsicupper,prop:timeradius} by proving Lemma~\ref{lem:geomfixedtime}. We start by recalling the following theorem proven implicitly by Lawler \cite{Lawlerlog}; stronger versions of the same theorem also appear explicitly in the work of Lawler, Sun, and Wu \cite{LawlerSunWu}.

\begin{theorem}[Lawler, \cite{Lawlerlog}]\label{thm:lawlerlog}
Let $S$ and $X$ be two independent simple random walks in $\Z^4$ both started from~$0$. For each $p\geq 1$ we have that 
\[
\E{\prcond{S[0,n^2]\cap \LE(X[1,\infty))=\emptyset}{X}{}^p } \asymp_p \frac{1}{(\log n)^{p/3} }
\]
for every $n\geq 2$.
\end{theorem}

\begin{remark}
\rm{
The exact statement proven in~\cite{Lawlerlog}
is that 
\[
\E{ \prcond{S[1,n^2]\cap \LE(X[0,\infty))=\emptyset}{X}{}^p } \asymp_p \frac{1}{(\log n)^{p/3}}
\]
for every $p \geq 1$ and $n \geq 2$. 
Exactly the same proof as in~\cite{Lawlerlog} works to prove Theorem~\ref{thm:lawlerlog}. The two forms of the estimate are also easily seen to imply each other.
}
\end{remark}

In the following corollary, we use Lawler's result to get control on the non-intersection probability when the walk $X$ is only run for finite time.

\begin{corollary}\label{cor:lawlerlogonethrid}
	Let $S$ and $X$ be two independent simple random walks in $\Z^4$ started from $0$. For each $p\geq 1$ we have that
	\[
	\E{\prcond{S[0,i]\cap \ler{}{X[1,k]}=\emptyset}{X}{}^p}\lesssim_p \frac{1}{(\log (k\wedge i))^{p/3}}
	\]
  for every $i,k \geq 2$.
%
%
%
\end{corollary}

\begin{proof}[Proof]
Since the left hand side is decreasing in $i$ it suffices to prove the claim in the case $i \leq k^{1/4}$.
Fix $1 \leq i \leq k^{1/4}$, let $\eta=\ler{}{X[1,\infty)}$ and let $Z=\mathbb{P}(S[0,i]\cap \ler{}{X[1,k]}=\emptyset \mid X)$ be the random variable whose $p$-th moment we wish to estimate.
Let $\eta=\ler{}{X[1,\infty)}$.
Recall that $B(0,r)$ denotes the Euclidean ball of radius $r$. We define the event
\[
 E=\left\{ \eta \cap B(0,i)\subseteq \ler{}{X[1,k]}   \right\}.
\]
If $E$ does not hold then $X$ must hit $B(0,i)$ after time $k$, so that
\begin{align}\label{eq:proofeik}
	\pr{E^c} & \leq \pr{X[k,\infty)\cap B(0,i)\neq \emptyset} \lesssim \frac{i^2}{k} \leq \frac{1}{i^2}
\end{align}
by a standard random walk calculation (see e.g.\ the proof of \cite[Lemma 4.4]{1804.04120}).
Since $S[0,i]$ is certainly contained in the ball $B(0,i)$, we deduce by Minkowski's inequality (the triangle inequality for the $L^p$ norm) that
\begin{align*}
	\E{Z^p}^{1/p} &\leq \E{\prcond{S[0,i]\cap \ler{}{X[1,k]}=\emptyset}{X}{}^p\mathbbm{1}(E)}^{1/p} + \P(E^c)^{1/p}
  \\&
   \lesssim \E{\prcond{S[0,i]\cap \eta=\emptyset}{X}{}^p}^{1/p} + \frac{1}{i^{2/p}}
\end{align*}
for every $p\geq 1$, so that the claim follows from \cref{thm:lawlerlog}.
\end{proof}

We will deduce \cref{lem:geomfixedtime} from the following variation on the same estimate in which the walk is run for a geometric random number of steps.

\begin{lemma}\label{lem:badgeom}
Let $T$ be a geometric random variable with mean $t$, and let $X$ be an independent random walk in $\Z^4$ started from $0$. Then 
\[
\pr{\LE(X^T) \text{\emph{ is $\delta$-bad}}}\lesssim_{\delta,p} \frac{1}{(\log t)^p}
\]	
for every $\delta>0$, $p\geq 1$, and $t \geq 2$.
\end{lemma}

Before starting the proof, we observe that the same argument used in the proof of \cref{lem:weakl1} also implies that there exists a constant $C$ such that if $T$ is a geometric random time (of any finite mean) and $\eta$ is a finite simple path starting at $x \in \Z^4$ then
\begin{equation}
\label{eq:weakl1geometric}
\P( T \geq m \mid \LE(X^T)=\eta) \leq \frac{C |\eta| \log(|\eta|+1)}{m}
\end{equation}
for every $m\geq 1$.

\begin{proof}[Proof of \cref{lem:badgeom}]
Fix $\delta>0$, $t\geq 2$, and $p\geq 1$. Then we have the union bound
\begin{multline*}
	\pr{\text{$\LE(X^T)$ is $\delta$-bad}} \leq \pr{\text{$\LE(X^T)$ is $\delta$-bad},\, |\LE(X^T)| \geq \frac{t}{(\log t)^{p+1}}, \text{ and $T\leq t(\log t)^2$}} 
	\\+ \pr{|\LE(X^T)|<\frac{t}{(\log t)^{p+1}}}
	 + \pr{T>t(\log t)^2}.
\end{multline*}
The final term is superpolynomially small in $t$ and hence is negligible for our purposes. Meanwhile, for the second term, \eqref{eq:weakl1geometric} implies that there exists a constant $C$ such that
\begin{align*}
	\pr{|\LE(X^T)|<\frac{t}{(\log t)^{p+1}}} \leq 2 \pr{T\leq  \frac{Ct}{(\log t)^{p}}} \lesssim \frac{1}{(\log t)^{p}}.
\end{align*}
as required.
  It therefore suffices to prove that
	\begin{align}\label{eq:deltabadgeom}
		\pr{\text{$\LE(X^T)$ is $\delta$-bad},\, |\LE(X^T)| \geq \frac{t}{(\log t)^{p+1}}, \text{ and $T\leq t(\log t)^2$}} \lesssim_{\delta,p} \frac{1}{(\log t)^p}.
	\end{align}

Let $E$ denote the event whose probability we wish to bound, let $\rho =|\LE(X^T)|$, and for each $0\leq i \leq \rho-1$ and $1\leq k \leq \rho$ let $A_{i} = \sum_{k=1}^\rho \frac{1}{k}\operatorname{Esc}_k(\LE(X^T)^i)^2$. In order for $\LE(X^T)$ to be $\delta$-bad, we must have by the definitions that
\begin{align*}
\sum_{i=0}^{\rho-1} A_{i}\mathbbm{1}\left(A_{i} \geq (\log \rho)^{1/3+\delta}\right)
>\delta \rho.
\end{align*}
Let $M = \lceil t (\log t)^2 \rceil$ and for each $0\leq m \leq M$ and $1 \leq k \leq M$ consider the random variables
\[
W_{m,k} =\operatorname{Esc}_k\left(\LE(X^m)\right) \qquad \text{ and } \qquad Z_m = \sum_{k=1}^{M} \frac{1}{k} W_{m,k}^2.
\]
For each $0\leq i \leq \rho$, let $\ell_i=\ell_i(X^T)$ be the last time that $X^T$ visits $\LE(X^T)_i$.
Since $\LE(X^{\ell_i})=\LE(X^T)^i$ for every $0\leq i \leq \rho$, we have that $Z_{\ell_i}=A_i$ for every $0\leq i \leq \rho-1$ and hence that if $\LE(X^T)$ is $\delta$-bad then
\begin{align*}
\sum_{m=0}^{T-1}Z_m \mathbbm{1}(Z_m \geq (\log \rho)^{1/3+\delta})
   > \delta \rho.
\end{align*}
It follows that there exists a positive constant $c$ such that if $E$ holds then
\begin{align*}
	\sum_{m=0}^{M-1} Z_m \mathbbm{1}(Z_m \geq c(\log t)^{1/3+\delta})  > \frac{\delta t}{(\log t)^{p+1}}.
\end{align*}
To conclude the proof of \eqref{eq:deltabadgeom}, it therefore suffices by Markov's inequality to prove that 
\begin{equation}
\sum_{m=0}^{M-1} \E{Z_m \mathbbm{1}(Z_m \geq c(\log t)^{1/3+\delta})} \lesssim_{\delta,p} \frac{t}{(\log t)^{2p+1}}. 
\end{equation}
This estimate will be deduced from \cref{cor:lawlerlogonethrid} via a simple computation with $L^q$ norms. It follows from the reversibility of LERW and~\cref{cor:lawlerlogonethrid} that 
\begin{align*}
	\E{W_{m,k}^q} = \E{\left(\prcond{S[0,k]\cap \LE(X[1,m])=\emptyset}{X}{0}\right)^q}\lesssim_q \frac{1}{(\log (k\wedge m))^{q/3}}.
\end{align*}
for every $q \geq 1$, $0\leq m \leq M$, and $1 \leq k \leq M$. Using Minkowski's inequality  to sum this estimate over $k$, we obtain that
%
\begin{multline}\label{eq:boundonzinorm}
	\norm{Z_m}_q \leq \sum_{k=1}^{M}\frac{1}{k}\norm{W_{m,k}^2}_q \lesssim_q \sum_{k=1}^{m}\frac{1}{k} \cdot \frac{1}{(\log k)^{2/3}} + \sum_{k=m}^{M}\frac{1}{k}\cdot \frac{1}{(\log m)^{2/3}}\\\lesssim (\log m)^{1/3} + \frac{\log M}{(\log m)^{2/3}} 
\end{multline}
for every $0\leq m \leq M$ and $q\geq 1$. Letting $m_0 = \lfloor t/(\log t)^{2p+2} \rfloor$, we deduce 
that if $m_0 \leq m \leq M$ then $\|Z_m\|_q \lesssim_q (\log M)^{1/3}$ for every $q\geq 1$. It follows by H\"older's inequality that 
\begin{align*}\label{eq:holderineq}
	\E{Z_m \1(Z_m \geq c(\log t)^{1/3+\delta})} &\leq
  \E{Z_m^{q}}^{1/q} \cdot \pr{Z_m\geq c(\log t)^{1/3+\delta}}^{(q-1)/q}\\
	&\lesssim_q (\log t)^{1/3} \cdot\left(\frac{\E{Z_m^{u}}}{(\log t)^{u/3+ u\delta}} \right)^{(q-1)/q} \lesssim_{q,u} (\log t)^{1/3-\delta u(q-1)/q}
\end{align*}
for every $m_0 \leq m \leq M$ and $q,u> 1$. Choosing $q,u> 1$ appropriately yields that 
\[
\E{Z_m \1(Z_m \geq c(\log t)^{1/3+\delta})} \lesssim_{p,\delta} \frac{1}{(\log t)^{2p+3}}
\]
for every $m_0 \leq m \leq M$. 
Since we also trivially have that $Z_m \lesssim \log t$ for every $0\leq m \leq M$, we obtain that 
\[
\sum_{m=0}^M \E{Z_m \1(Z_m \geq c(\log t)^{1/3+\delta})} \lesssim_{\delta,p} m_0 \log t + \frac{M}{(\log t)^{2p+3}} \lesssim_{\delta,p} \frac{t}{(\log t)^{2p+1}}
\]
as required, completing the proof.
\end{proof}

\begin{proof}[Proof of Lemma~\ref{lem:geomfixedtime}]
Fix $\delta>0$ and $p\geq 1$.
Let $T$ be a geometric random variable of mean $n\geq 2$ independent of $X$. Summing 
 over the possible values of $T$ and applying \cref{lem:badgeom} gives that
\begin{multline*}
	\frac{1}{n}\sum_{k=0}^{n} \prstart{\LE(X^k) \text{ is $\delta$-bad}}{0} \lesssim \sum_{k=0}^{n} \frac{1}{n}\left( 1 -\frac{1}{n}\right)^k \cdot \prstart{\LE(X^k) \text{ is $\delta$-bad}}{0} \\ \leq \prstart{\LE(X^T) \text{{ is $\delta$-bad}}}{0} \lesssim_{\delta,p} \frac{1}{(\log n)^{p}}
\end{multline*}
for every $n\geq 2$ as required.
\end{proof}

\subsection*{Acknowledgments} This work was carried out while TH was a Herchel Smith Postdoctoral Research Fellow at the University of Cambridge  and a Junior Research Fellow at Trinity College Cambridge. PS's research was supported by the Engineering and Physical Sciences Research Council: EP/R022615/1.

\phantomsection

\addcontentsline{toc}{part}{References}

%

 \setstretch{1}
 \footnotesize{
  \bibliographystyle{abbrv}
  \bibliography{TomsRefs,biblio}
  }

  Data availability statement: This manuscript has no assosciated data.

\end{document}